\documentclass[11pt]{amsart}
\usepackage{amsmath}
\usepackage{amsfonts}
\usepackage{amsthm}
\usepackage{amssymb}
\usepackage{mathrsfs}
\usepackage{latexsym}
\usepackage{diagrams}
\usepackage{a4wide}
\usepackage{enumerate}
\usepackage{tikz}
\usepackage{tikz-cd}
\usepackage{enumitem}
\usetikzlibrary{matrix}
\usepackage{hyperref}

\def\Kli{\mathrm{Kli}}
\def\Sph{\mathrm{Sph}}
\def\Iw{\mathrm{Iw}}
\def\epsbar{\overline{\eps}}
\def\PSL{\mathrm{PSL}}
\def\Kbar{\overline{K}}
\def\Htw{\widetilde{H}}
\def\Htwb{\overline{H}}

\def\Gammabar{\overline{\Gamma}}
\def\MMM{L}

\def\unW{\mathcal{W}}
\def\CA{\mathcal{A}}
\def\an{\mathrm{an}}
\def\pr{\mathrm{pr}}
\def\EE{\mathcal{E}}

\def\w{\mathbf{w}}

\def\con{\mathrm{con}}
\def\wT{\widetilde{\T}}
\def\wI{\widetilde{I}}

\def\can{\mathrm{can}}
\def\sub{\mathrm{sub}}

\def\PP{\mathbf{P}}
\def\Forms{\mathcal{F}}
\def\XX{\mathcal{X}}
\def\bb{\mathfrak{b}}
\def\gg{\mathfrak{g}}

\def\uu{\mathfrak{u}}
\def\Rmin{R^{\mathrm{min}}}
\def\Sym{\mathrm{Sym}}
\def\T{\mathbf{T}}

\def\Tan{\T^{\mathrm{an}}}
\def\mE{\mathfrak{m}_{\emptyset}}

\def\wt{\widetilde}
\def\CL{\mathcal{L}}
\def\ra{\rightarrow}
\def\lra{\longrightarrow}
\def\onto{\twoheadrightarrow}
\def\iso{\stackrel{\sim}{\rightarrow}}

\def\Tor{\mathrm{Tor}}
\def\m{\mathfrak{m}}
\def\T{\mathbf{T}}
\def\Z{\mathbf{Z}}
\def\Ext{\mathrm{Ext}}
\def\OL{\mathcal{O}}

\def\Q{\mathbf{Q}}
\def\F{\mathbf{F}}
\def\Ind{\mathrm{Ind}}
\def\PGL{\mathrm{PGL}}
\def\R{\mathbf{R}}
\def\C{\mathbf{C}}

\def\Hom{\mathrm{Hom}}

\def\End{\mathrm{End}}

\def\GSp{\mathrm{GSp}}
\def\Sp{\mathrm{Sp}}
\def\GL{\mathrm{GL}}
\def\WT{\widetilde{\T}}
\def\Qbar{\overline{\Q}}
\def\Fbar{\overline{\F}}
\def\rhobar{\overline{\rho}}
\def\ad{\mathrm{ad}}
\def\SL{\mathrm{SL}}

\def\eps{\epsilon}
\def\G{G}
\def\A{\mathbb{A}}

\def\Frob{\mathrm{Frob}}
\def\Spec{\mathrm{Spec}}

\def\CC{\mathcal{C}}
\def\CO{\mathcal{O}}
\def\onto{\twoheadrightarrow}
\def\into{\hookrightarrow}
\def\liso{\stackrel{\sim}{\longrightarrow}}

\DeclareMathOperator\im{Im}
\DeclareMathOperator\diag{diag}

\newtheorem{theorem}{Theorem}[section]

\newtheorem{df}[theorem]{Definition}
\newtheorem{lemma}[theorem]{Lemma}
\newtheorem{prop}[theorem]{Proposition}
\newtheorem{corr}[theorem]{Corollary}

\newtheorem{example}[theorem]{Example}
\newtheorem{remark}[theorem]{Remark}
\newtheorem{assumption}[theorem]{Assumption}

\def\ord{\mathrm{ord}}
\def\Gal{\mathrm{Gal}}
\def\rbar{\overline{r}}
\DeclareMathOperator\cts{cts}

\def\can{\mathrm{can}}
\def\sub{\mathrm{sub}}
\def\CE{\mathcal{E}}
\def\CV{\mathcal{V}}
\def\CH{\mathcal{H}}

\def\dR{\mathrm{dR}}
\DeclareMathOperator{\Lie}{\mathrm{Lie}}

\def\CA{\mathcal{A}}

\def\cusp{\mathrm{cusp}}
\def\temp{\mathrm{temp}}
\def\Res{\mathrm{Res}}

\def\gh{\mathfrak{h}}
\def\gg{\mathfrak{g}}
\def\gk{\mathfrak{k}}
\def\gp{\mathfrak{p}}
\def\tw{\widetilde{w}}

\def\W{\mathcal{W}}
\def\Art{\mathrm{Art}}

\newcommand{\Se}{\mathcal{S}}

\newcommand{\UZ}{U \kern-.1em{Z}}

\newcommand{\VF}{V \kern-.07em{F}}

\def\Sp{\mathrm{Sp}}

\begin{document}

  \newcommand{\need}[1]{{\tiny *** #1}}
\newcommand{\mar}[1]{\marginpar{\raggedright\tiny FIXME: #1 }}

\title{Minimal Modularity Lifting 
For non-regular symplectic representations}
\author{Frank Calegari \and David Geraghty} 
\thanks{The first author was supported in part by NSF Career Grant DMS-0846285 and NSF  Grant
  DMS-1404620 and NSF Grant   DMS-1701703. The second author was
  supported in part by NSF Grants DMS-1200304 and DMS-1128155.}

{\footnotesize
\maketitle
\tableofcontents
}

\section{Introduction}

In this paper, we prove a minimal modularity lifting theorem for
Galois representations (conjecturally) associated to
 Siegel modular forms $\pi$ for the group $\GSp(4)/\Q$ such that
 $\pi_{\infty}$ is a holomorphic limit of discrete series. 
An example of what we can prove with these methods is the following:

\begin{theorem} \label{theorem:char0}
Let $r: G_{\Q} \rightarrow \GSp_4(\Qbar_p)$ be a continuous irreducible
representation satisfying the following conditions:
\begin{enumerate}
\item $r|G_{\Q_p}$ is ordinary with
Hodge--Tate weights $[0,0,j-1,j-1]$ for some integer $j$ satisfying $p-1 > j \ge 4$.
\item If $\alpha$ and $\beta$ are the unit root eigenvalues 
of Frobenius on $D_{\mathrm{cris}}(V)$, then
$$(\alpha^2 -1)(\beta^2 - 1)(\alpha - \beta)(\alpha^2 \beta^2  -1) \not\equiv 0 \mod p.$$
\item 
The
image of $\rbar|G_{\Q(\zeta_p)}$ contains~$\Sp_4(\F_p)$.
\item For a prime~$x \ne p$, the image of inertia at~$x$ is unipotent, and
the image of any generator of tame inertia has the same number of Jordan blocks mod~$p$ as
it does in characteristic zero.
\item $\rbar$ is  modular of level $N(\rbar)$ and weight~$(j,2)$. \medskip
\end{enumerate}
Then $r$ is modular, that is, there exists a cuspidal Siegel modular Hecke eigenform
$F$ of weight $(j,2)$ such that 
$$L(r,s) = L(F,s),$$
 where $L(F,s)$ is the 
spinor $L$-function of $F$.  
\end{theorem}

We deduce Theorem~\ref{theorem:char0} from our main result, which
we now state. 
 (We shall refer to \S~\ref{section:deformations} and
 \S~\ref{sec:gal-rep-low} for precise details concerning ramification
 behaviour, level subgroups and the exact definition of minimal deformations.)
Let $\eps$ denote the $p$-adic cyclotomic character. Let  $\OL$ be the ring of integers
of a finite extension $K$ of $\Q_p$.
Let 
$\displaystyle{\rbar:G_{\Q} \rightarrow \GSp_4(k)}$
be a continuous irreducible representation whose similitude character
$\nu(\rbar)$ on inertia at $p$ is the mod-$p$ reduction of $\eps^{1-j}$.
Suppose that
$\rbar | G_{\Q_p}$ contains an unramified subspace of dimension two on which
$\Frob_p$ acts by the scalars $\alpha$ and $\beta$ respectively,
where 
$$(\alpha^2 - 1)(\beta^2 - 1)(\alpha - \beta)(\alpha^2 \beta^2 - 1) \ne 0.$$
Suppose  further that
$\rbar$ has big image (explicitly, satisfies Assumption~\ref{assumption:bigimage})
and that $\rbar|G_{\Q_x}$ for a prime $x \ne p$ is either unramified or
is one  of the types listed in Assumption~\ref{assumption:ramification}.
Let $Y_1(N)$ denote the  (open) Siegel modular variety of level $N = N(\rbar)$ over
$\Spec(\OL)$, where
$N$ is determined by $\rbar$ as
in \S~\ref{section:SMF}, and let $\omega(j,2)$ denote the coherent sheaf on
$Y_1(N)$ whose complex sections define Siegel modular forms of weight $(j,2)$
for some integer $p - 1 > j \ge 4$.
Let $\T$ denote the subring of endomorphisms
of
$$e H^0(Y_1(N),\omega(j,2) \otimes K/\OL) \simeq \lim_{\rightarrow} e H^0(Y_1(N),\omega(j,2) \otimes \OL/\varpi^n)$$
(where~$e  = e_{\alpha,\beta}$ is a certain ordinary projection, see section~\ref{section:ordinaryprojection})
 generated by Hecke operators at
primes not dividing $Np$. Let $\Rmin$ denote the universal minimal deformation
ring  of $\rbar$
(see   Definition~\ref{defn:minimal} for  more details).

\begin{theorem} \label{theorem:realmain}
Suppose that there exists a maximal ideal $\m$ of $\T$
and a corresponding representation $\rbar_{\m}:G_{\Q} \rightarrow \GSp_4(k)$ which
is isomorphic to $\rbar$.
Let $\Rmin$ denote the universal minimal ordinary deformation ring of $\rbar$. 
Suppose that $p-1 > j \ge 4$. Then there is an isomorphism
$\Rmin \iso \T_{\m}$, and moreover, the $\T_{\m}$ module
$e H^0(Y_1(N),\omega(j,2) \otimes K/\OL)_{\m}^{\vee}$ is free as a $\T_{\m}$-module. 
\end{theorem}

The proof follows the strategy of~\cite{CG}. The main ingredients are showing that there
exists  a map from $\Rmin$ to $\T_{\m}$ (see Theorem~\ref{theorem:localglobal})
and proving that the cohomology of the subcanonical extension $\omega(j,2)^{\sub}$ of $\omega(j,2)$ to
a smooth toroidal compactification $X_1(N)$ of $Y_1(N)$ vanishes outside degrees $0$ and $1$
(see Theorem~\ref{thm:lan-suh} ---
 in the case of classical modular curves this step was trivial).

\subsection{Comparison with previous methods}
The first modularity theorems which applied to non-regular motives were the results of
 Buzzard--Taylor and Buzzard~\cite{BuzzT,BuzzWild} on 
 two dimensional odd Artin  representations $V$.
The  idea of these papers can roughly be described as follows.
Using known cases of Serre's conjecture, one deduces that $\rhobar$
is modular, where $\rho$ is the representation associated to some $p$-adic
realization of $V$ for some $p$.
Using modularity theorems in \emph{regular} weight, one then proves that
a big Hecke algebra is modular. Specializing to weight one, one deduces the
existence of an \emph{overconvergent} eigenform $f$ corresponding to  $V$. Under
a 
non-degeneracy assumption on $V$ ($\rhobar$ is $p$-distinguished), one constructs (using companion forms) a
second
 Hida family
which specializes to a \emph{second} eigenform $g$. Using the geometric
properties of $U$, one shows that $f$ and $g$ converge deeply into
the supersingular locus. The Fourier coefficients $a_n$ of $f$ and $g$ for $(n,p) = 1$
are determined by $V$. One then constructs a suitable linear combination
$h = (\alpha f - \beta g)/(\alpha - \beta)$ which converges over the entire modular curve,
and is thus classical by rigid GAGA. The formal rigid geometry employed
by these papers have been generalized  by various authors, in particular by Kassaei~\cite{Kassaeiold}.
One may well ask whether this approach can be applied to Siegel modular forms
of weight $(2,2)$ --- work of Tilouine and his collaborators has
made great progress in this direction. The modularity lifting result for (regular weight) Hida families
has been established in many cases by Genestier--Tilouine~\cite{TG} (see also
Pilloni~\cite{Pilloni}). Significant progress has also been made in the theory
of canonical subgroups and the geometry of Siegel modular varieties.
One difficulty, however, is that the Fourier expansions of Siegel modular forms
are \emph{not} determined by the Hecke eigenvalues. This is a difficulty which
must be overcome in such an approach.
(Various classicality results for overconvergent forms can be established
without using $q$-expansions, see for example~\cite{PayGlue,PillStroh}, but these
results  only apply to forms of sufficiently non-critical slope.)
The difficulty of dealing with $q$-expansions manifests itself for
our approach also --- we are forced to prove (``by hand'') various properties
of Fourier expansions of Hecke eigenforms in
\S~\ref{section:explicit}.

\subsection{Abelian Surfaces}
\label{spec}
It would be desirable to weaken the assumption $j \ge 4$ in the main theorem
to $j \ge 2$, since the case $j = 2$ includes the representations associated
to the Tate modules of abelian surfaces.
The only point in our arguments in which we use the fact that $j \ge 4$ is to
deduce that $H^2(X_1(N),\omega(j,2)^{\sub}) = 0$ for the subcanonical extension $\omega(j,2)^{\sub}$
of $\omega(j,2)$ to a smooth toroidal compactification $X_1(N)$ of $Y_1(N)$.
 If this vanishing holds for $j = 2$ then our
theorem would also apply to these cases.
On the other hand, one does \emph{not} expect vanishing here, because one expects
that singular Siegel modular forms should contribute cohomology in these degrees.
However,
 we need only the weaker result that the image of 
$H^2(X_1(N),\omega(j,2)^{\can})$ in $H^2(X_1(N),\omega(j,2)^{\sub})$
is zero after localization at a sufficiently non-Eisenstein maximal ideal~$\m$. We expect this to always be true for~$j = 2$, although we were not able to prove it.
On the other hand, using the ideas of Khare--Thorne~\cite{KT}, one can dispense with proving this under the very strong supplementary hypothesis that
there exists a~\emph{characteristic zero} form of weight~$N = N(\rbar)$ which gives rise to~$\rbar$. In particular, by using the arguments of
the proof of Theorem~6.29 of~\emph{ibid}, one should  be able to prove the analogue of Theorem~\ref{theorem:char0} in weight~$(2,2)$
assuming the existence of an auxiliary Siegel modular form~$G$ of the same level also of weight~$(2,2)$ with~$\rbar_{G} = \rbar$.

\subsection{Recent Developments} ({\bf Added: January, 2019\rm}) \label{january}
Very recently, there have been a number of developments related to the main theme
of this paper, in particular, in the preprints~\cite{pilloniHidacomplexes} and~\cite{BCGP}, the latter  which establishes
the potential modularity of abelian surfaces over totally real fields. The introduction
to~\cite{BCGP} explains a number of innovations which made those results possible,
so we shall confine ourselves here to only a few salient remarks.
The first is that the vanishing conjecture for~$H^2$ localized at~$\m$
mentioned in~\S\ref{spec} remains unresolved, and the methods of~\cite{BCGP} blend the techniques of this paper (and~\cite{CG}) with arguments from~\cite{pilloniHidacomplexes}.  A second point is that the paper~\cite{pilloniHidacomplexes} develops
a conceptual method to define (normalized) Hecke operators at~$p$, and in particular establishes the action of these
operators on higher cohomology (which is essential for the main results of~\cite{pilloniHidacomplexes} and~\cite{BCGP}).
In this paper, it suffices to construct the action of these Hecke operators on~$H^0$ which is significantly easier.
The methods we use in~\S\ref{section:hecke}  to do this are  admittedly disagreeable, relying as they do on arguments 
using~$q$-expansions. Thus the reader is encouraged to consult~\cite[\S7]{pilloniHidacomplexes}  
and~\cite[\S4.5]{BCGP} for a more geometric construction of these operators. An analysis of the normalization
factors for Hecke operators required in~\cite{pilloniHidacomplexes} also sheds some light on another phenomenological
feature of this paper which readers may find surprising: On the Galois side, there is essentially no
difference (in the ordinary setting) between working in the (irregular) weight~$(j,2)$ for~$j > 2$ and working in the (irregular) weight~$(2,2)$.
On the other hand, the Hecke operators at~$p$ (particularly~$T_{p,2}$)  behave quite differently in weight~$(2,2)$.
In our context, this arises most noticeably in~\S\ref{relationship} (via Lemma~\ref{lemma:verify}), which one should compare to~\cite[\S11.1]{pilloniHidacomplexes} (warning:
the convention of that paper is that Pilloni's~$T_{p,1}$ is our~$T_{p,2}$ and \emph{vice versa}, and the spherical  version
of the operator~$T$ in~\cite{pilloniHidacomplexes} is equal in weight~$(2,2)$ up to translation by a multiple of~$T_{p,0}$ to the operator we call~$Q_2$). Finally,
the paper~\cite{BCGP} develops a geometric version of the doubling argument (see~\S5 of~\emph{ibid}.) 
This provides a much more robust  explanation (in a slightly different setting)
 for what in this paper occupies
most of~\S\ref{section:qstuff} and consists of  a sequence of
tricky and not entirely intuitive series of manipulations with~$q$-expansions. (Note that the geometric doubling
argument of~\cite{BCGP} is  only written for weight~$(2,2)$ but the method applies in principle to the weights~$(j,2)$ 
which we consider in this paper.) Finally, the very observant reader will notice that the doubling argument of~\cite{BCGP}
applies in weight~$(2,2)$ to the space of ordinary forms at Klingen level, whereas in this paper we essentially
prove (in the same weight) a tripling result at spherical  level.
 Neither of these results immediately imply the other.  The ``extra'' copy of the space of forms 
 can be interpreted as  giving rise to a space
of \emph{non-ordinary} forms of weight~$(p+1,p+1)$. See Remark~\ref{trip} for further discussion on
this point, which we also discuss in a different context below.

It is natural to ask  whether one should expect any genuine difficulties in modifying the 
geometric doubling argument of~\cite{BCGP} to the setting of this paper. We now offer
some speculative remarks to address this point (using notation from~\cite{BCGP}).
Let~$\pi_p$ be a smooth admissible irreducible unramified representation of~$\GL_2(\Q_p)$ (over~$\C$) which
is not trivial. (For example, $\pi_p$ could be the local constituent of an automorphic
representation~$\pi$ associated to a classical modular form.)
Let~$\Sph = \GL_2(\Z_p)$
and let~$\Iw$ denote the Iwahori subgroup of~$\Sph$. The classical theory of oldforms is a reflection
of the fact that
$$\dim \pi^{\Iw}_p = 2 = 2 \cdot \dim \pi^{\Sph}_p$$
and the characteristic zero version of doubling is the statement that the span of the spherical  vector~$v$
under the operator~$U_p$ is all of~$\pi^{\Iw}_p$.  The integral version of this statement is false in general.
For example, given a classical ordinary modular eigenform~$f$ of weight~$k \ge 2$, the span of~$f \mod p$
under~$U_p$ is simply~$f$, because~$T_p = U_p \mod p$ in these weights. However, some version
of this result \emph{does} hold in weight~$k = 1$, and it is this property which is leveraged to prove
local-global compatibility results in~\cite{CG}. Let us now replace~$\GL_2(\Q_p)$ by~$\GSp_4(\Q_p)$,
and let~$\Kli$ and~$\Iw$ denote the Klingen and Iwahori subgroups respectively of~$\Sph = \GSp_4(\Z_p)$ 
(denoted elsewhere in this paper by~$\Pi$ and~$I$ respectively.) Now (for the~$\Pi_p$ of interest) we will have
$$\dim \Pi^{\Iw}_p = 8 = 2 \cdot 4 = 2 \dim \Pi^{\Kli}_p = 8 \dim \Pi^{\Sph}_p.$$
The factor~$8$ here may be interpreted as the order of the Weyl group of~$\GSp(4)$.
More prosaically, the oldforms in~$\Pi^{\Iw}_p$ correspond to a choice of eigenvalues~$\alpha$
and~$\alpha \beta$ for the Hecke operators~$U_{\Iw(p),1}$ and~$U_{\Iw(p),2}$ respectively,
whereas the oldforms  in~$\Pi^{\Kli}_p$ correspond to a choice of eigenvalues~$\alpha + \beta$
and~$\alpha \beta$ for the Hecke operators~$U_{\Kli(p),1}$ and~$U_{\Kli(p),2} = U_{\Iw(p),2}$.
When one passes from~$\pi^{\Sph}_p$ to~$\pi^{\Iw}_p$ for weight one modular forms or~$\Pi^{\Kli}_p$
to~$\Pi^{\Iw}_p$ for weight~$(2,2)$ Siegel modular forms, the property of of being ordinary turns out to be
automatically preserved on the corresponding space of old forms. However, this is not
\emph{a priori} true
when passing from~$\Pi^{\Sph}_p$ to~$\Pi^{\Iw}_p$, and so  one would have to see in any
geometric version of this argument a way of dealing with the non-ordinary forms.

\subsection{Results of Arthur}
In Section~\ref{sec:balanced-property}, we make use of the results of
\cite{arthur-gsp4}, which sketches how the results of
\cite{Arthur} on orthogonal and symplectic groups can be extended to
the general symplectic group $\GSp_4$. At the time of the initial submission of this paper, 
 these results of Arthur are conditional on the stabilization of the twisted trace formula.
(We direct the reader to~\cite{GT} for the most up to date status of these results for~$\GSp_4$.)

\subsection{Acknowledgements}
We would like to thank George Boxer for some very helpful comments
related to the proofs of Theorems~ \ref{theorem:boxer}  and~\ref{theorem:theta}.
We would also like to thank Olivier Ta\"{\i}bi for answering some technical questions arising in~\S\ref{sec:balanced-property}.
We  would also like to acknowledge useful conversations with Kevin Buzzard,  Ching-Li Chai, Matthew Emerton, Toby Gee,  Michael Harris,
Kai-Wen Lan, Vincent Pilloni, and Jack Thorne. 
We also like thank many of the participants of the Bellairs workshop in number theory in~2014,
where an earlier version of this paper was discussed. Finally, we thank the referees, whose detailed comments
very much helped to improve this manuscript.

\section{Notation}
\label{sec:notation}

We fix a prime $p$ and let $\OL$ be the ring of integers
of a finite extension $K$ of $\Q_p$ with residue field $k$. We let
$\CC_{\CO}$ denote the category of complete local Noetherian
$\CO$-algebras $R$ with residue field isomorphic to $k$ (via the
structural homomorphism $\CO \to R$).

We let 
\[ \eps : G_{\Q} \to \Z_p^{\times} \]
denote the cyclotomic character. The Hodge--Tate weight of
$\eps|G_{\Q_p}$ is $-1$.

If $L$ is a finite extension of $\Q_l$ for some prime $l$. We let
$\Art_L : L^\times \to W_L^{\mathrm{ab}}$ denote the Artin map,
normalized so that uniformizers correspond to geometric Frobenius elements.
If $\gamma$
is an element of some ring $R$, then we define the character
\[ \lambda(\gamma) : G_L \lra R^\times \]
to be the unramified character which takes the geometric Frobenius
element $\Frob_L$ to $\gamma$, when this character is well defined.

\subsubsection{The group $\GSp_4$}
\label{sec:group-gsp_4}

Let 
$G = \GSp_4 = \{M \in \GL_4: M^{t} J M = \nu \cdot J \text{\ for some\
}
\nu\in\GL_1\}$, where
$$\displaystyle{J:=\left(\begin{matrix}  0 &0 & 0 & 1 \\ 0 & 0 & 1 &  0 \\ 0 & -1 & 0 & 0 \\ -1 & 0 & 0 & 0 \end{matrix} \right)}.$$
The group $\Sp_4$ is the subgroup consisting of elements with $\nu =
1$. We let $B \subset G$ be the Borel subgroup consisting of upper
triangular matrices. The Lie algebras of $G$ and $B$ are denoted $\gg$
and $\bb$ while those of $\Sp_4$ and $B\cap \Sp_4$ are denoted $\gg^0$
and $\bb^0$. Let $P\subset \G$ denote the Siegel
parabolic, that is, the stabilizer of the plane spanned by the first
two standard basis vectors. Let $\Pi \subset \G$ denote the Klingen
parabolic, which is the stabilizer of the line spanned by the first
standard basis vector. We denote the Levi subgroup of $P$ (resp.\
$\Pi$) by $M=M_P$ (resp.\ $M_\Pi$). We have $M \cong \GL_2\times\GL_1$.

Let $T$ denote the diagonal torus in $\GSp_4$ and $X^*(T)$ its
character group. We identify $X^*(T)$ with the lattice $\Z^3$ by associating to $(a,b;c)$ the
character
\[ \diag(t_1,t_2,\nu t_2^{-1},\nu t_1^{-1}) \mapsto t_1^a t_2^a
  \nu^{c}.\] 
We identify the cocharacter group $X_*(T)$ with $\Z^3$ by associating
the triple $(\alpha,\beta;\gamma)$ with the cocharacter:
\[ t \mapsto \diag(t^\alpha, t^\beta , t^{\gamma-\beta} ,
  t^{\gamma-\alpha}) .\]
The natural pairing on $X^*(T)\times
X_*(T)$ is then: $\langle (a,b;c),(\alpha,\beta,\gamma)\rangle
\mapsto a\alpha+b\beta+c\gamma$.

The positive roots of $\G$ with respect to the Borel $B$ are given by
$\alpha_1 :=(1,-1;0)$, $\alpha_2:=(0,2;-1)$, $\alpha_3=(1,1;-1)$ and
$\alpha_4=(2,0;-1)$. Of these, $\alpha_1$ and $\alpha_2$ are the simple
roots. We let $\rho = (2,1;-3/2)$ denote the half-sum of the positive
roots. The coroots are: $\alpha_1^{\vee} = (1,-1;0)$, $\alpha_2^{\vee}
= (0,1;0)$, $\alpha_3^{\vee} = (1,1;0)$ and $\alpha_4^{\vee} = (1,0;0)$.
The intersection
$B\cap M$ is a Borel subgroup of $M$. The
corresponding positive root is $\alpha_1$.

\begin{df}
  \label{defn:dominant-weights}
We define the set $X^*(T)^+_{\G}$ to be the set $\{ \lambda\in X^*(T) :
\langle \lambda, \alpha_i^{\vee}\rangle \geq 0 \text{\ } \forall
i\}$ of weights which are dominant with respect to $B$. Explicitly
\[ X^*(T)^+_{\G} =  \{(a,b;c)\in X^*(T) :a \geq b \geq 0\}.\]  
Similarly, we define the set of 
weights $X^*(T)^+_{M}:= \{(a,b;c)\in X^*(T):\langle
\lambda,\alpha_1^{\vee}\rangle \geq 0\}$ which are dominant with respect
to $B\cap M$. Explicitly, this is:
\[ X^*(T)^+_{M} = \{(a,b;c)\in X^*(T) :a \geq b\}.\]
\end{df}

Note that the
natural action of $M$ on the plane spanned by the first two (resp.\
the last two) standard basis vectors is the irreducible representation
of highest weight $(1,0;0)$ (resp.\ $(0,-1;1)$).

We let $W_{\G}=N_{\G}(T)/T$ denote
the Weyl group of $\G$ and we define $W_{M}$ and $W_{M_{\Pi}}$
similarly. Let $s_0,s_1$ denote the generators for the Weyl group
$W_{\G}$ given in~\cite[\S 2]{HerzigTilouine}. We fix a set of
Kostant representatives $W^M= \{ \tw_0,\tw_1,\tw_2,\tw_3\} $ for $
W_{M}\backslash W_{\G}$ by setting $\tw_0=1$, $\tw_1=s_1$,
$\tw_2=s_1s_0$ and $\tw_3=s_1s_0s_1$. Note that each $\tw_i$ has length
$i$. We let $w\in W_{\G}$ act on $X^*(T)$ by
$(w\lambda)(t)=\lambda(w^{-1}tw)$. Then we have:
\begin{align*}
  \tw_1(a,b;c) &= (a,-b;b+c)\\
  \tw_2(a,b;c) & = (b,-a;a+c) \\ 
  \tw_3(a,b;c) & = (-b,-a;a+b+c).
\end{align*}
The longest element of $W_{\G}$ which we denote by $w_0$ acts via $w_0(a,b;c)=(b,a;c)$.

Note that the collection of representatives $W^M$ is precisely the set
of $w\in W_{\G}$ such that $w(X^*(T)^+_{\G})\subset X^*(T)^+_M$.
We let $C_0 \subset X^*(T)_{\R}:= X^*(T)\otimes_{\Z}\R$ denote the
closed dominant Weyl chamber. In other words, $C_0 = \{ (a,b;c)\in
\R^3 : a \geq b \geq 0\}$. For $i=1,2,3$, we define the chambers $C_i:= \tw_i(C_0)$.

\subsubsection{The group $\GSp_4(\R)$}
\label{sec:group-gsp_4r}

Let 
\[ h : \Res_{\C/\R}(\mathbf G_m)(\R)= \C^{\times} \to \G(\R)= \GSp_4(\R) \]
be the homomorphism sending $x+iy$ to the matrix
\[
\begin{pmatrix}
  xI_2 & yS \\
  -yS & xI_2
\end{pmatrix}
\]
where  
\[S:=
\begin{pmatrix}
  0 & 1\\1&0
\end{pmatrix}.\]
Let $K^h$ denote the
centralizer of $h$ in $\G(\R)$ (acting by conjugation). Then since
$h(i) = J$, we see that $K^h = \R^\times K_\infty$ where $K_\infty$ is
the maximal compact subgroup of $\G(\R)$ given by the fixed points of
the Cartan involution $g \mapsto (g^{t})^{-1}$. The similitude
character restricts to a surjective map $\nu : K_\infty \to \{\pm 1\}$
and whose kernel $K_{\infty,1}$ is the connected component of the
identity. Then we have explicitly,
\[ K_{1,\infty} = \left\{
  \begin{pmatrix}
    SAS & SB \\
    -BS & A
  \end{pmatrix} \in \G(\R) : A^t A + B^t B = I_2, A^t B = B^t A \right\}.\]
 The map:
\begin{align*}
 K_{\infty,1} & \longrightarrow \GL_2(\C) \\
  \begin{pmatrix}
    SAS & SB \\
    -BS & A
  \end{pmatrix} &\longmapsto  A +iB 
\end{align*}
 induces an isomorphism between $K_{\infty,1}$ and $U(2)$. We let
 $H_1 \subset K_{\infty,1}$ denote the preimage of the diagonal
 compact torus in $U(2)$ and let $H:= \R^{\times}_{>0} H_1 \subset
 K^h$. Let $\gh = \Lie H$, $\gk^h = \Lie K^h$ and so on. Then we have
\[ \gh =\left\{ h(t_1,t_2;z):=
\begin{pmatrix}
  z & 0 & 0 & t_1 \\ 0 & z & t_2 & 0 \\ 0 & -t_2 & z & 0\\ 
  -t_1 & 0 & 0 & z
\end{pmatrix}: t_1,t_2,z \in \R \right\}.\]
We use subscripts to denote complexifications of Lie algebras and Lie
groups; thus $H_{\C}$ and $\gh_{\C}$ denote the complexifications of
$H$ and $\gh$. Then $\gh_{\C} = \Lie
H_{\C} = \{ h(t_1,t_2;t) : t_1,t_2,z \in \C\}$ and the surjective map $\exp : \gh_{\C} \to H_{\C}$ sends $h(t_1,t_2;z)$
to
\[ \exp(z)
\begin{pmatrix}
  \cos t_1 & 0 & 0 & \sin t_1 \\
  0 & \cos t_2 & \sin t_2 & 0 \\
  0 & -\sin t_2 & \cos t_2 & 0 \\
  -\sin t_1 & 0 & 0 & \cos t_1
\end{pmatrix}.\]
Thus its kernel is $\{ h(t_1,t_2;z) : t_1,t_2 \in 2\pi\Z, z\in 2\pi i
\Z\}$. We define the lattice $X^*(H_{\C}) \subset \gh_{\C}^*$ to
be the subspace consisting of differentials of (complex analytic) characters of
$H_{\C}$. Equivalently, $X^*(H_{\C})$ is the subset of $X^*(\C^\times
\times H_{1,\C})=\{ \lambda \in \gh_{\C}^* :
\lambda (\ker( \exp:\gh_{\C}\to H_{\C})) \subset 2\pi i \Z \}$
consisting of differentials of characters of 
$\C^\times\times K_{1,\C}$ which factor through the multiplication map $\C^\times\times
H_{1,\C} \to H_{\C}$. We fix
an isomorphism 
\[\{ (a,b;c) \in\Z^3:a+b\equiv c \bmod 2\} \iso X^*(H_{\C})\] by letting $(a,b;c)$
correspond to the linear form 
\[ h(t_1,t_2;z) \mapsto at_1 i + b t_2 i + cz \]
on $\gh_{\C}$. This extends by linearity to an isomorphism $\C^3 \to \gh_{\C}^*$.

Let $V^{\pm} \subset \C^4$ be the subspace where $h(i)$ acts via $\pm
i$. Then each $V^{\pm}$ is isotropic and we have an orthogonal direct
sum $\C^4 = V^- \oplus V^+$. Let $Q^- \subset \G(\C)$ denote the
stabilizer of $V^-$. 
 Consider the Hodge decomposition
\[ \gg_{\C} = \gg^{0,0}\oplus \gg^{-1,1} \oplus \gg^{1,-1} \]
where $\gg^{p,q}$ is the subspace on which $h(z)$ acts via
$z^{-p}\overline z^{-q}$. Then we have $\gg^{0,0} = \gk_{\C}^h$ and we
let $\gp^+=\gg^{-1,1}$, $\gp^- = \gg^{1,-1}$. We also let $P^{\pm}$
denote the subgroup of $G(\C)$ generated by $\exp(\gp^{\pm})$. 
 Then
we have 
\[ Q^- = K^h_{\C} P^{-} \text{\ and\ } \Lie Q^- = \gk_{\C}^h \oplus \gp^{-}.\]
Moreover, $K^h_{\C}$ is the Levi component of $Q^-$ and $P^-$ is its
unipotent radical. Let
\[ f_1 =
\begin{pmatrix}
  1 \\ 0 \\ 0 \\ -i
\end{pmatrix}, 
f_2 =
\begin{pmatrix}
  0 \\ 1 \\ -i \\ 0
\end{pmatrix}, 
f_3 =
\begin{pmatrix}
  0 \\ -i \\ 1 \\ 0
\end{pmatrix},
f_4 =
\begin{pmatrix}
  -i \\ 0 \\ 0 \\ 1
\end{pmatrix}\in \C^4.\]
Then $f_1,f_2$ are a basis of $V^-$ and $f_3,f_4$ are a basis of
$V^+$. With respect to the basis $f_1,\dots,f_4$ of $\C^4$, an element
\[ k = \begin{pmatrix}
    SAS & SB \\
    -BS & A
  \end{pmatrix} \in K_{1,\infty} \]
acts via 
\[ C^{-1} k C
 =
\begin{pmatrix}
  SAS-iSBS & 0 \\
  0 & A+iB
\end{pmatrix} \text{\ where } C := \begin{pmatrix}
  I_2 & -iS \\
  -iS & I_2
\end{pmatrix} .  \]
Note that the Cayley transform $C$ conjugates the Siegel parabolic
$P(\C)$ to $Q^-$.  
Let $\Phi \subset X^*(H_{\C})$ denote the root system defined
by the adjoint action of $H_{\C}$ on $\gg_{\C}$. The compact roots
$\Phi_c$ are those appearing in $\gk_{\C}^h$, while the non-compact
roots $\Phi_n$ are those appearing in $\gp^+\oplus \gp^-$. We choose a system of
positive roots $\Phi^+$ in such a way that the set of positive non-compact
roots $\Phi^+_n = \Phi^+ \cap \Phi_n$ coincides with the roots in $\gp^+$. (We do this in order to be consistent with the conventions
of \cite[\S 2.4]{blasius-harris-ramak}.)  We are then forced to take
$\Phi^+$ to be the set of roots appearing in $C(\Lie
\overline{B})C^{-1}$  where $\overline{B} \subset \G$ is the Borel
subgroup of \emph{lower} triangular matrices. With respect to the
identification of $X^*(H_{\C})$ as a subset of $\Z^3$ given above, we then
have:
\begin{align*}
  \Phi^+_c &= \{ (1,-1;0) \} \\
  \Phi^+_n &= \{ (0,2;0), (1,1;0) , (2,0;0) \}.
\end{align*}
This can be seen easily from the fact that $C^{-1}h(t_1,t_2;0)C =
\diag(-it_1,-it_2,it_2,it_1)$. 
\begin{df}
  \label{defn:dominant-weights-H}
We let 
$X^*(H_{\C})^+_{K^h_{\C}}$ denote the set of which are dominant with respect the
system of positive roots $\Phi^+_c$. In other words,
$X^*(H_{\C})^+_{K^h_{\C}} = \{
(a,b;c)\in X^*(H_{\C}) : a\geq b\}$.

 This set parameterizes
the irreducible complex
analytic representations of $K_{\C}^h$. For $\mu \in
X^*(H_{\C})^+_{K^h_\C}$, we let $V_\mu$ denote the corresponding
irreducible representation of highest weight $\mu$.

\end{df}

 We note that natural
representation of $K^h_{\C}$ on $V^-$ (resp.\ $V^+$) is the irreducible
representation of highest weight $(0,-1;1)$ (resp.\ $(1,0;1)$). Note
also that the similitude character $\nu : H_{\C} \to \C^\times$ has
weight $(0,0;2)$.

\section{Some Commutative Algebra}

We recall here some formalism from~\cite{CG} for proving modularity
lifting results in contexts where the Hecke algebra has ``co-dimension $1$''
over the ring of diamond operators.
The notion of ``balanced'' below plays the role of ``codimension one''
for the non-regular group rings $S_N:=\CO[(\Z/p^N\Z)^q]$.

\subsection{Balanced Modules}

Let $S$ be a Noetherian local ring with residue field $k$ and let $M$ be a
finitely generated $S$-module.

\begin{df}
  \label{defn:defect}
We define the \emph{defect} $d_S(M)$ of $M$ to be 
\[ d_S(M)= \dim_k \Tor^0_S(M,k)-\dim_k \Tor^1_S(M,k)=\dim_k M/\m_S M -
\dim_k \Tor^1_S(M,k).\]
\end{df}

Let
\[ \dots \ra P_i \ra \dots \ra P_1 \ra P_0 \ra M \ra 0\]
be a (possibly infinite) resolution of $M$ by finite free
$S$-modules. Assume that the image of $P_i$ in $P_{i-1}$ is
contained in $\m_S P_{i-1}$ for each $i\geq 1$. (Such resolutions
always exist and are often
called `minimal'.) Let $r_i$ denote the rank of $P_i$. Tensoring
the resolution over $S$ with $k$ we see that $P_i/\m_S P_i \cong
\Tor^i_S(M,k)$ and hence that $r_i = \dim_k \Tor^i_S(M,k)$.

\begin{df}
  \label{defn:balanced}
We say that $M$ is \emph{balanced} if $d_S(M)\geq 0$.
\end{df}

If $M$ is balanced, then we see that it admits a presentation 
\[ S^d \ra S^d \ra M \ra 0 \]
with $d = \dim_k M/\m_S M$.

\subsection{Patching}
\label{sec:patching}

We recall the abstract  Taylor--Wiles style patching
result from~\cite{CG}.

\begin{prop}
\label{prop:patching}
Suppose that
\begin{enumerate}
\item $R$ is an object of $\CC_\CO$ and $H$ is a finite $R$-module which is also finite over~$\CO$;
\item $q \geq 1$ is an integer, 
                                and for each integer $N\geq 1$,  
$S_N:=\CO[(\Z/p^N\Z)^q]$;
\item  $R_\infty:= \CO[[x_1,\dots,x_{q-1}]]$;
\item for each $N\geq 1$, $\phi_N: R_\infty \onto R$ is a surjection
  in $\CC_\CO$ and $H_N$ is an $R_\infty\otimes_\CO
  S_N$-module
\end{enumerate}
and that for each $N\geq 1$ the following conditions are satisfied
\begin{enumerate}[label=(\alph*)] 
 \item\label{cond-image} the image of $S_N$ in $\End_\CO(H_N)$ is contained in the image
   of $R_\infty$, and moreover, the image of the augmentation ideal of
   $S_N$ in $\End_{\OL}(H_N)$ is contained in the image of $\ker(\phi_N)$;
 \item\label{cond-coninvts} there is an isomorphism $\psi_N: (H_N)_{\Delta_N} \iso H$ of
   $R_\infty$-modules, where $R_\infty$ acts on $H$ via $\phi_N$ and
   $\Delta_N = (\Z/p^N\Z)^q$;
 \item\label{cond-balanced} $H_N$ is finite and balanced over $S_N$ (see Definition \ref{defn:balanced}).
\end{enumerate}
Then $H$ is a free $R$-module.
\end{prop}

\begin{proof} This is Prop.~2.3 of~\cite{CG}.
\end{proof}

\section{Deformations of Galois representations}
\label{section:deformations}

Let
$$\rbar: G_{\Q} \rightarrow \GSp_4(k)$$
be a continuous, odd, absolutely irreducible Galois representation
with similitude character of the form $\nu(\rbar) =
\overline{\eps}^{-(a-1)}$ where $a \geq 2$.
Let us suppose that there exist $\alpha$ and $\beta$ in $k$ such that
$$\rbar | G_p \sim
\left( \begin{matrix} 
\lambda(\alpha)  & 0 & * & * \\ 
0 &  \lambda(\beta) & * & * \\
0 & 0 & \nu(\rbar) \cdot \lambda( \beta^{-1}) & 0 \\
0 & 0 & 0 &  \nu(\rbar)  \cdot \lambda(\alpha^{-1})
\end{matrix} \  \ \right),$$
and moreover $(\alpha^2 - 1)(\beta^2 - 1)(\alpha^2 \beta^2 - 1)(\alpha
- \beta) \ne 0$. 
Let $S(\rbar)$ denote the set of primes of $\Q$ away from $p$ at which $\rbar$ is
ramified. 

The group $\GSp_4$ admits a $11$-dimensional adjoint
representation on its Lie algebra $\gg$.
Let $\ad(\rbar)$ denote the composition of $\rbar$ with this representation.
For~$p > 2$, the representation~$\ad(\rbar)$ admits a decomposition~$\ad(\rbar) = \ad^0(\rbar) \oplus \nu$,
where~$\nu$ is the similitude character of~$\rbar$.

We make the following further assumptions on $\rbar$:

\begin{assumption}[Big Image] \label{assumption:bigimage}
The restriction of $\rbar$ to $G_{\Q(\zeta_p)}$ satisfies the following conditions, cf. ~\S5.7 of~\cite{Pilloni}:
\begin{enumerate}
\item[H1:] The field~$\Q(\ad^0(\rbar))$ does not contain~$\zeta_p$,
\item[H2:] For any~$m$, there exists an element~$\sigma \in G_{\Q(\zeta_{p^m})}$
such that~$\rbar(\sigma)$ has four distinct eigenvalues and such that the action of~$\sigma$
on each irreducible representation of~$\ad^0(\rbar)$  over~$G_{\Q(\zeta_{p^m})}$
contains~$1$ as an eigenvalue.
\item[H3:] Neither the image~$\Gamma$ of~$\ad^0(\rbar)$ nor the image of~$\ad^0(\rbar)(1)$ admits
a quotient of degree~$p$.
\end{enumerate}
\end{assumption}

If this assumption holds, we say that~$\rbar$ has \emph{big image}, although condition~(H1)
depends on more than the group-theoretic image of~$\rbar$ or even~$\rbar|_{G_{\Q(\zeta_p)}}$.

\begin{assumption}[Neatness]  \label{assumption:neatness}
There exists a~$\sigma \in G_{\Q}$ with~$\eps(\sigma)  = q  \not\equiv 1 \mod p$
such that the ratio of any two eigenvalues of~$\rbar(\sigma)$ is not equal to~$q  \mod p$.
\end{assumption}

This condition is  imposed to  avoid  dealing with stacks. If~$p \ge 5$,  any surjective
representation~$\rbar: G_{\Q} \rightarrow \GSp_4(\F_p)$ whose similitude character is a power of~$\epsbar$
will be neat.
By assumption, the image contains an element~$\rbar(\sigma)$ which is scalar with
eigenvalue~$\lambda \ne \pm 1$. If~$q = \eps(\sigma) \equiv 1 \mod p$, then the similitude
character would also equal~$1$. But the similitude character of the scalar matrix~$\lambda$ is~$\lambda^2 \not\equiv
1 \mod p$.

\begin{assumption}[Ramification] \label{assumption:ramification}
If $x \in S(\rbar)$, then $\rbar | G_x$ is one of the following types:
\begin{enumerate} 
\item{\bf U3 \rm}: $\rbar | I_x$ has unipotent image, and $\rbar |
  I_x$ is conjugate to the group generated by $\exp(N_3)$, where 
$$ N_3 = \left( \begin{matrix} 0 & 1  & 0 & 0 \\ 0 & 0 & 1 & 0 \\ 0 & 0 & 0 & -1 \\ 0 & 0 & 0 & 0 \end{matrix} \right).$$
\item{\bf U2 \rm}: $\rbar | I_x$ has unipotent image, and $\rbar |
  I_x$ is conjugate to the group generated by $\exp(N_2)$, where
$$ N_2 = \left( \begin{matrix} 0 & 1  & 0 & 0 \\ 0 & 0 & 0 & 0 \\ 0 & 0 & 0 & -1 \\ 0 & 0 & 0 & 0 \end{matrix} \right).$$
 \item{\bf U1 \rm}: $\rbar | I_x$ has unipotent image, and $\rbar |
   I_x$ is conjugate to the group generated by $\exp(N_1)$, where
 $$ N_1 = \left( \begin{matrix} 0 & 0  & 0 & 0 \\ 0 & 0 & 1 & 0 \\ 0 & 0 & 0 & 0 \\ 0 & 0 & 0 & 0 \end{matrix} \right).$$
\item{\bf P\rm}: $\rbar | G_x$ is a direct sum of characters, and $\rbar | I_x$ has the form
$$\left( \begin{matrix} 1 & 0  & 0 & 0 \\ 0 & 1 & 0 & 0 \\ 0 & 0 & \chi_x & 0 \\ 0 & 0 & 0 & \chi_x \end{matrix} \right)$$
  for some non-trivial character $\chi_x$ of $I_x$. Both the plane of invariants under~$I_x$  and the plane on which~$I_x$ acts by~$\chi_x$
  are isotropic. 
  Moreover $x -1$ is prime to $p$.
\item{\bf H\rm}: $\rbar | I_x$ is absolutely irreducible and $x^4 - 1$ is prime to $p$.
\end{enumerate}
\end{assumption}

\begin{remark} \emph{Since we are assuming that the similitude character of~$\rbar$ is
a    power of the cyclotomic character, it turns out that~$\rbar|I_x$ can never be of type~{\bf P\rm}. 
We expect that our arguments can also be adapted to deal with representations~$\rbar$ with more general (odd) similitude characters,
but we made this assumption to simplify some of the arguments  involving~$q$-expansions (in particular, to avoid
various Nebentypus characters). 
}
\end{remark}

Note that non-trivial unipotent representations are not direct sums,
so a prime $x \in S(\rbar)$ is either of type {\bf U\rm},  {\bf P\rm},
or {\bf H\rm}, but never simultaneously any two of these
types. Moreover, $x$ is of type {\bf U2\rm} or {\bf
  U3\rm} if and only if $\rbar(I_x)$ is generated by an element
$\exp(N)$ where $N$ is nilpotent of rank $2$, or $3$ respectively.

Let $Q$  denote a finite set of primes of $\Q$ disjoint from
$S(\rbar)\cup\{p\}$. We assume that for each $x\in Q$ the following
hold:
\begin{itemize}
\item $x\equiv 1 \mod p$,
\item $\rbar | G_x$ is a direct sum of four pairwise distinct
  characters. Label these characters as
  $\lambda(\alpha_x),\lambda(\beta_x),\lambda(\gamma_x)$, and $\lambda(\delta_x)$
  such that the planes   $\lambda(\alpha_x)\oplus\lambda(\beta_x)$ and
  $\lambda(\gamma_x)\oplus\lambda(\delta_x)$ are isotropic and
  $\alpha_x\delta_x = \beta_x\gamma_x = \nu(\rbar)(\Frob_x)$.
\end{itemize}

(By abuse of notation, we sometimes use $Q$ to denote the product of primes in $Q$.)
For objects $R$ in $\CC_\CO$, a \emph{deformation} of $\rbar$ to
$R$ is a $\ker(\GSp_4(R)\to \GSp_4(k))$-conjugacy class of continuous
lifts $r : G_\Q\to\GSp_4(R)$ of $\rbar$. We will often refer to
the deformation containing a lift $r$ simply by $r$. 

\begin{remark} 
\emph{When deforming Galois representations over~$\Q$,
we could work either with fixed or varying similitude character ---  both give  rise to deformation problems with~$l_0 = 1$. We make
the (somewhat arbitrary) choice to work with deformations with fixed similitude character in this paper, because it is the ``correct'' approach
for general totally real fields --- for  totally real fields other than~$\Q$, the invariant~$l_0$ increases (by~$[F:\Q] - 1$) when deforming
the similitude character.
}
\end{remark}

\begin{df}
  \label{defn:minimal}
We say that
a deformation $r:G_\Q\to \GSp_4(R)$ of $\rbar$ is \emph{minimal
  outside $Q$}
if it satisfies the following properties:
\begin{enumerate}
\item\label{det} The similitude character $\nu(r)$ is equal to
  $\eps^{-(a-1)}$.
    \item\label{outside-NQ} If $x\not\in Q\cup S(\rbar) \cup \{p\}$ is a prime of $\Q$, then $r|G_x$ is unramified.
\item\label{at-special} If $x \in S(\rbar)$ is of type {\bf U1}, {\bf U2\rm
  \ or \bf U3\rm}, then $r|I_x$ has unipotent image and its image is
  topologically generated by an element $\exp(N)$ where $N$ is
  nilpotent of rank $1$, $2$ or $3$ respectively.
\item\label{at-principle} If $x \in S(\rbar)$ is of type {\bf P\rm}, then $r(I_x) \iso \rbar(I_x)$.
\item\label{at-Q} If $x\in Q$, then $r|G_x \cong V_1 \oplus V_2$ where
  each $V_i$ is an isotropic plane in $R^4$ and $V_1$ lifts
  $\lambda(\alpha_x)\oplus \lambda(\beta_x)$ while $V_2$ lifts
  $\lambda(\gamma_x)\oplus \lambda(\delta_x)$. Moreover, $I_x$ acts by 
  scalars (via some character) on~$V_1$ and by scalars via the inverse of
  this character on~$V_2$.
  \item\label{at-p} The representation $r$ has the following shape at $p$: 
  $$r | G_p  \sim \left( \begin{matrix}
  \chi_{\alpha} \psi^{-1} & 0 & * & * \\
  0 & \chi_{\beta}  \psi^{-1} & * & * \\
  0 & 0 & \eps^{-(a-1)} \chi_{\beta}^{-1} \psi & 0 \\
  0 & 0 & 0 & \ \eps^{-(a-1)} \chi_{\alpha}^{-1} \psi 
  \ \end{matrix} \  \ \right),$$
where $\chi_{\alpha}$ and $\chi_{\beta}$ are unramified characters
lifting $\lambda(\alpha)$ and $\lambda(\beta)$ respectively, and~$\psi$ is  an unramified character
which is
trivial modulo the maximal ideal.
\end{enumerate}
If $Q$ is empty, we will refer to such deformations simply as being
\emph{minimal}. 
If $r$ satisfies conditions \eqref{outside-NQ}--\eqref{at-principle}, then we say $r$ is \emph{weakly minimal} outside $Q$.
\end{df}

\begin{remark} \emph{The local condition at $p$ is equivalent to
    asking that $r$ is ordinary (of fixed weight). When~$a = 2$ it is also
equivalent to being finite flat. 
This is because, for unramified characters ${\psi_1}$ and $\psi_2$, the group $\Ext^1({\psi_1},\psi_2)$ in this category is
trivial, and the group
 $\Ext^1(\eps {\psi_1},\psi_2)$ is the same whether it is computed
in the category of finite flat group schemes or as $G_p$ modules, as long as~${\psi_1} {\psi_2}^{-1} \not\equiv 1 \mod p$. The latter
condition follows (for all the relevant extensions) from the assumption $(\alpha \beta - 1)(\alpha^2 - 1)(\beta^2 - 1)(\alpha - \beta) \ne 0$.}
\end{remark}

 The functor that associates to each object $R$ of
$\CC_\CO$ the set of deformations of $\rbar$ to $R$ which are
minimal outside $Q$ is represented by a complete Noetherian local
$\CO$-algebra $R_Q$. This follows from the proof
of Theorem 2.41 of ~\cite{DDT}. If $Q=\emptyset$, we will
sometimes denote $R_Q$ by $R^{\min}$. 

\medskip

Let $H^1_{Q}(\Q,\ad^0(\rbar))$
denote the Selmer group defined as the kernel of the map
\[ H^1(\Q,\ad^0(\rbar)) \lra \bigoplus_{x} H^1(\Q_x,\ad^0(\rbar))/L_{Q,x}\]
where $x$ runs over all primes of $\Q$ and:
\begin{itemize}
\item  If $x \not \in Q \cup p$, then
  $L_{Q,x} = H^1(G_x/I_x,(\ad^0(\rbar))^{I_x})$.
  \item If $x \in Q$, then $H^1(G_x,\ad^0(\rbar))$ is isomorphic to the
    subspace of 
\[  H^1(G_x,\ad \  \lambda(\alpha_x))
\oplus H^1(G_x,\ad \lambda(\beta_x)) \oplus H^1(G_x,\ad \ 
\lambda(\beta_x)^{-1}) \oplus H^1(G_x,\ad \  \lambda(\alpha_x)^{-1})  \]
consisting of elements $(c_1,c_2,d_2,d_1)$ with $c_1+d_1 =
c_2+d_2$. (Note that each summand is a copy of $\Hom_{\cts}(G_x,k)$.) 
We let $L_{Q,x}$ denote the subspace corresponding to elements
$(c_1,c_2,d_2,d_1)$ with $c_1-c_2$ and $d_1 - d_2$  and~$c_1+d_1$ (equivalently, $c_2 + d_2$) unramified. 
\item If $x = p$, then we define $L_{Q,p}=L_p$ as follows:
let $\uu\subset \bb^0$ be the subspace of matrices whose
non-zero entries appear in the upper right $2\times 2$ block. We
define $L'_p = \ker (H^1(G_p,\bb^0) \to H^1(I_p,\bb^0/\uu))$ and
$L_p=L_{Q,p}=\im(L'_p \to H^1(G_p,\gg^0))$.
\end{itemize}
Let $H^1_{Q}(\Q,\ad^0(\rbar(1)))$ denote the corresponding
dual Selmer group.

\begin{lemma}
  \label{lem:FL}
We have $\dim_k L_p - \dim_k H^0(G_p,\ad^0(\rbar)) = 3$.
\end{lemma}

\begin{proof}
The subspace $L'_p\subset H^1(G_p,\bb^0)$ is precisely set of elements mapping
to the subspace $H^1(G_p/I_p,(\bb^0/\uu)^{I_p})\subset
H^1(G_p,\bb^0/\uu)$. We have $\bb^0/\uu \cong  1\oplus 1\oplus
  \lambda(\beta)\lambda(\alpha)^{-1}$ as a $k[G_p]$-module and hence
  $H^1(G_p/I_p,(\bb^0/\uu)^{I_p})$ is 2-dimensional since $\alpha\neq
  \beta$. The condition $(\alpha^2 - 1)(\beta^2 - 1)(\alpha^2 \beta^2 - 1)(\alpha -
\beta) \ne 0$ implies that $h^2(G_p,\uu)=0$ and hence $H^1(G_p,\bb^0)
\onto H^1(G_p,\bb^0/\uu)$.
It follows that
  $\dim_k L'_p = 2 + h^1(G_p,\bb^0) - h^1(G_p,\bb^0/\uu)$.
  Thus
\[ \dim_k L_p'- h^0(G_p,\bb^0) = 2 + h^1(G_p,\uu) -
h^0(G_p,\bb^0/\uu) - h^0(G_p,\uu).\]
We have $h^0(G_p,\uu)=0$ and $h^0(G_p,\bb^0/\uu)=2$. The Euler characteristic
formula implies that $h^1(G_p,\uu)=3$.  Thus
\[ \dim_k L_p'- h^0(G_p,\bb^0) = 2 + 3 -
2- 0 = 3.\]
Finally, the condition on $\alpha$ and $\beta$ 
implies that $h^0(G_p,\gg^0/\bb^0)=0$. It follows that
$h^0(G_p,\bb^0)=h^0(G_p,\gg^0)$ and $L'_p \iso L_p$. This concludes
the proof.
\end{proof}

\begin{prop}
\label{prop:tangent-space-w1}
 The reduced tangent space $\Hom(R_Q/\m_\CO,k[\epsilon]/\epsilon^2)$ of $R_{Q}$ has
  dimension  
$$\dim_k H^1_{Q}(\Q,\ad^0(\rbar(1))) - 1 + \# Q. $$
\end{prop}

\begin{proof} The argument is very similar to that of
Corollary~2.43 of~\cite{DDT}. The reduced tangent space has dimension
$\dim_k H^1_Q(\Q,\ad^0(\rbar))$. By Theorem 2.18 of
\emph{loc.\ cit.}\ this is equal to 
\begin{align*}  \dim_k H^1_Q(\Q,\ad^0(\rbar(1))) + \dim_k
 & H^0(\Q,\ad^0(\rbar))-\dim_k H^0(\Q,\ad^0(\rbar(1))) \\ 
+ \sum_{x} (\dim_k L_{Q,x}-\dim_k &
H^0(\Q_x,\ad^0(\rbar))) -\dim_k H^0(G_\infty,\ad^0(\rbar)),
\end{align*}
where $x$ runs over all finite places of $\Q$. 
The second term is equal to~$0$ and the third term
vanishes (by the absolute irreducibility of $\rbar$ and the fact that
$\rbar\not \cong \rbar\otimes \epsilon$). Now, we have:
\begin{itemize}
\item $\dim_k L_{Q,x} - \dim_k 
H^0(\Q_x,\ad^0(\rbar))=0$ for $x\not \in Q\cup \{ p\}$;
\item $\dim_k L_{Q,x} - \dim_k 
H^0(\Q_x,\ad^0(\rbar))= 3$ for $x=p$; 
\item $\dim_k L_{Q,x} - \dim_k 
H^0(\Q_x,\ad^0(\rbar))= 1$ for $x\in Q$ (by \cite[Prop.\ 10.4.1]{TG}); and
\item $\dim_k H^0(G_\infty,\ad^0(\rbar))=4$.
\end{itemize}
This concludes the proof.
\end{proof}

The next result (on the existence of Taylor-Wiles primes) follows from
the previous proposition and the proof of \cite[Prop.\ 5.6]{Pilloni}.

\begin{prop} 
\label{prop:tw-primes-w1}
Let $q =\dim_k H^1_{\emptyset}(\Q,\ad^0(\rbar(1)))$ and recall that we are supposing $\rbar$ satisfies
Assumption \ref{assumption:bigimage}.
Then $q\geq 1$ and for any integer $N\geq 1$ we can find a set $Q_N$
  of primes of $\Q$ such that 
\begin{enumerate}
\item $\# Q_N =q$.
\item $x \equiv 1 \mod p^N$ for each $x\in Q_N$.
\item For each $x\in Q_N$, $\rbar$ is unramified at $x$ and
  $\rbar(\Frob_x)$ has four pairwise distinct eigenvalues.
\item $H^1_{Q_N}(\Q,\ad(\rbar(1)))=(0)$.
\end{enumerate}
 In particular, the reduced tangent space of $R_{Q_N}$ has dimension
  $q-1$ and $R_{Q_N}$ is a quotient of a power series
  ring over $\CO$ in $q-1$ variables.
\end{prop}

\begin{example}[Examples of representations with big image] \label{inducedexample}
Suppose that~$p \ge 5$.
\begin{enumerate}
\item Let $K/\Q$ be an imaginary quadratic field not contained in~$\Q(\zeta_p)$.
Let
$$\rhobar: G_{K} \rightarrow \GL_2(\F_p)$$
is a representation with  determinant~$\eps^{1-k}$ for some integer~$k$ such that the images of~$\rhobar$
and~$\rhobar^{c}$ for any complex conjugation~$c \in \Gal(\Qbar/\Q)$ both contain~$\SL_2(\F_p)$
 are have totally disjoint fixed fields over~$K(\zeta_p)$.
 Then the representation
$$\rbar = \Ind^{\Q}_{K} \rhobar$$
preserves a symplectic form and has big image.
\item Suppose the image of~$\rbar$ is~$\GSp_4(\F_p)$. Then~$\rbar$ has big image.
\end{enumerate}
\end{example}

\begin{proof} 
The second claim follows immediately for~$p \ge 5$ by~\cite{Pilloni}, Prop~5.8.
For the first claim, it is an easy
consequence of the fact that~$\SL_2(\F_p)$ is perfect for~$p \ge 5$ that~$H3$ holds,
and similarly, assuming that~$K \not\subset \Q(\zeta_p)$, that~$H1$ holds.
Hence it suffices to find an element in the image with distinct eigenvalues
and with~$1$ as an eigenvalue for every irreducible constituent of~$\ad^0(\rbar)$.
We first compute the representation~$\ad^0(\rbar)$.
Note that the dual of~$\rhobar$ and~$\rhobar^c$ can be identified with~$\rhobar \times \eps^{k-1}$
and~$\rhobar^c \otimes \eps^{k-1}$ respectively.
Over~$K$, we have an identification
$$\ad^0(\rbar) |_{G_K} = (\rhobar \otimes \rhobar^c) \otimes \eps^{k-1} \oplus \ad^0(\rhobar) \oplus \ad^0(\rhobar^c),$$
and over~$\Q$, we have an identification
$$\ad^0(\rbar) = \mathrm{As}(\rhobar) \otimes \eps^{k-1} \oplus \Ind^{\Q}_{K} \ad^0(\rhobar),$$
where~$\mathrm{As}$ is the Asai representation.
Over~$\Q(\zeta_{p^m})$ for any~$m$, the character~$\eps^{k-1}$ is trivial, and hence
 the image of~$\rbar |_{G_{\Q(\zeta_{p^m})}}$ under our assumptions
is the group~$\SL_2(\F_p)^2 \rtimes \Z/2\Z$.  Since~$1$ and~$-1$ are always
eigenvalues of any element acting on~$\Ind^{\Q}_{K} \ad^0(\rhobar)$, 
it suffices to find an element~$\sigma \in \SL_2(\F_p)^2 \rtimes \Z/2\Z$ which has distinct
eigenvalues under~$\rbar$ and has an eigenvalue~$1$ in~$\mathrm{As}(\rhobar)$.
To be more precise, since we haven't been careful about distinguishing the Asai
representation from its quadratic twist, we shall find an element with eigenvalues both~$1$ and~$-1$.
One can explicitly realize the Asai representation as follows.
Let~$V$  be the  standard representation of~$\SL_2(\F_p)$ over~$\F_p$, and let~$V \otimes V$ be the representation of the exterior product~$\SL_2(\F_p) \times \SL_2(\F_p)$.
The element~$(g,h)$ acts on~$v \otimes w$ via~$(g,h)(v \otimes w) = (gv \otimes hw)$.
The Asai representation is determined uniquely by the action of
a fixed lift of complex conjugation~$c \in \Gal(\Qbar/\Q)$, which acts on~$V \otimes V$
by the formula~$c(v \otimes w) = w \otimes v$.

Consider the elements~$g,h  \in \SL_2(\F_p)$
such that, with respect to some chosen basis~$V = \{u,v\}$,
$$g = \left( \begin{matrix} x & 0 \\ 0 & x^{-1} \end{matrix} \right), \qquad
h = \left( \begin{matrix} y & 0 \\ 0 & y^{-1} \end{matrix} \right).$$
Then~$c \cdot (g,h)$ acts on~$\rbar$ via the matrix
$$\left( \begin{matrix} 0 & 0 & 1 & 0 \\ 0 & 0 &  0 & 1 \\ 1 & 0 & 0 & 0 \\ 0 & 1 & 0 & 0 \end{matrix} \right)\left( \begin{matrix} x & 0 & 0 & 0 \\ 0 & x^{-1} &  0 & 0 \\ 0 & 0 & y  & 0 \\ 0 & 0 & 0 & y^{-1} \end{matrix} \right)
$$
with eigenvalues~$\pm (xy)^{1/2}$ and~$\pm (xy)^{-1/2}$. On the other hand, the action of this element via the Asai representation (and basis~$u \otimes u$, $v \otimes v$, $u \otimes v$, $v \otimes u$) is
$$\left( \begin{matrix} 1 & 0 & 0 & 0 \\ 0 & 1 &  0 & 0 \\ 0 & 0 & 0 & 1 \\ 0 & 0 & 1 & 0 \end{matrix} \right)\left( \begin{matrix} xy & 0 & 0 & 0 \\ 0 & (xy)^{-1} &  0 & 0 \\ 0 & 0 & (x/y)  & 0 \\ 0 & 0 & 0 & (x/y)^{-1} \end{matrix} \right)
$$
with eigenvalues~$xy$, $(xy)^{-1}$, and~$\pm 1$. The four eigenvalues are distinct
as long as~$\pm (xy)^{1/2} \ne \pm (xy)^{-1/2}$, or equivalently if~$(xy)^2 \ne 1$.
One can now choose~$x = 2$ and~$y = 1$ in~$\F^{\times}_p$.
\end{proof}

\begin{remark} \emph{Suppose that~$K$ is an imaginary quadratic field,
and suppose that $E/K$ is an elliptic curve which neither has~CM nor
is isogenous (over~$\Kbar$) to its Galois conjugate~$E^{c}/K$.
We claim that  Example~\ref{inducedexample} applies to the 
mod~$p$ representations~$\rhobar: G_K \rightarrow \GL_2(\F_p)$
associated to the dual of~$E[p]$ for sufficiently large~$p$. The representations~$\rbar$
in this case are the duals of the representations~$A[p]$ associated to the abelian 
surface~$A = \Res^{\Q}_{K}(E)$. 
By~\cite{MR0387283}, the Galois representations~$\rhobar_p, \rhobar^c_p: G_K \rightarrow \GL_2(\F_p)$
associated to the duals of~$E[p]$ and~$E^c[p]$
 have images~$\GL_2(\F_p)$ and determinants~$\eps^{1-2}$
for all sufficiently large~$p \ge 5$. Let~$F/K$ and~$F^c$ denote the corresponding
extensions, so~$\Gal(F/K)$ and~$\Gal(F^c/K)$ are both isomorphic to~$\GL_2(\F_p)$,
and~$\Gal(F/K(\zeta_p))$ and~$\Gal(F^c/K(\zeta_p))$ 
are both isomorphic to~$\SL_2(\F_p)$.
By the simplicity of~$\PSL_2(\F_p)$ for~$p \ge 5$, the only non-trivial quotients of~$\SL_2(\F_p)$
are~$\PSL_2(\F_p)$ and~$\SL_2(\F_p)$. This implies that if~$H:= F \cap F^c \supseteq K(\zeta_p)$
is strictly larger than~$K(\zeta_p)$, then 
 then either~$\Gal(H/K) = \GL_2(\F_p)$, or~$\Gal(H/K) = \GL_2(\F_p)/\pm I$.
 In either case, the projective representations associated to~$\rhobar_p$ and~$\rhobar^c_p$
 both factor through~$\Gal(H/K)$.
 Since all automorphisms of~$\PGL_2(\F_p)$ are inner, 
 this implies that projective representations of~$\rhobar_p$ and~$\rhobar^c_p$ are isomorphic, and hence
$\rhobar_p \simeq \rhobar^c_p \otimes \chi_p$ for some character~$\chi_p$ which (by
comparing determinants) is at most quadratic. 
Assume~$p$ is sufficiently large  so that~$E$
has good reduction at all primes above~$p$ and moreover that~$p$ is unramified in~$K$.
Then~$\rhobar_p$ and~$\rhobar^c_p$ are both finite flat at~$v|p$, which forces~$\chi_p$
to be unramified at all primes above~$p$. But this implies that~$\chi_p$
is unramified outside primes dividing the conductor~$N$ and~$N^c$ of~$E$ and~$E^{c}$
respectively.
There are only finitely many such quadratic characters by class field theory. 
Hence, if there are infinitely primes~$p$
for which the assumptions of~Example~\ref{inducedexample} do not occur,
then there exists a \emph{fixed} character~$\chi$ with~$\chi^2 = 1$ and isomorphisms
$\rhobar_p \simeq \rhobar^c_p \otimes \chi$ for infinitely many~$p$.
Such an isomorphism  (for a single~$p$) implies that~$a_{v} = \chi(v) a_{v^c} \mod p$ for all
pairs of conjugate primes~$v$ and~$v^c$ of good reduction for~$E$,
and hence, given infinitely many such~$p$, one deduces the equality~$a_v = \chi(v) a_{v^c}$.
If~$L/K$ 
is the (at most) quadratic extension in which~$\chi$ splits, this implies
(by Cebotarev) that the Tate modules (for any fixed prime) of~$E$ and~$E^c$ are isomorphic,
and hence
(by  Faltings~\cite{MR718935}) that~$E$ and~$E^c$ are isogenous over~$L$.
}
\end{remark}

\section{Siegel threefolds}
\label{section:SMF}

\subsection{Level Structure}
\label{sec:level-str}
Recall that there are two conjugacy classes of maximal parabolic
subgroups of $\GSp(4)$ represented by the \emph{Siegel parabolic} $P$ which is
block upper triangular with
Levi
$$M = M_P := \left\{ \left( \begin{matrix} A &  \\   & \lambda B\end{matrix} \right):
\lambda \in \GL_1, A \in \GL_2, B =  S {}^{t}A^{-1} S, S = \left( \begin{matrix}  & 1 \\ 1 & \end{matrix} \right) \right\},$$
and the \emph{Klingen parabolic} $\Pi$ which is block upper triangular with Levi
$$M_{\Pi}:= \left\{ \left( \begin{matrix} \lambda & & \\ & A & \\ & & \lambda^{-1} \det(A) \end{matrix} \right):
\lambda \in \GL_1, A \in \GL_2 \right\}.$$
These both contain the Borel subgroup~$B$.
For each prime $x$, these give rise to parahoric subgroups $P(x)$, $\Pi(x)$, and~$I(x)$ of $\GSp_4(\Z_x)$,
namely, the inverse image of the corresponding parabolic subgroups over $\F_x$.
(The group~$I(x)$ is called the Iwahori subgroup.)
The Klingen parahoric subgroup contains a normal subgroup $\Pi(x)^{+}$ with
$\Pi(x)/\Pi(x)^{+} \simeq (\Z/x\Z)^{\times}$ (via projection onto
$\lambda \bmod x$).  
For each prime~$x$, we also have the~\emph{Paramodular group}~$K(x)$, which is the stabilizer in~$\GSp_4(\Q_x)$
of~$\Z_x \oplus \Z_x \oplus \Z_x \oplus x \Z_x$,  and is the intersection
$$\left( \begin{matrix} * & * & * & */x \\
* & * & * & */x \\
* & * & * & */x \\
x* & x* & x* & x \end{matrix} \right) \cap \GSp_4(\Q_x)$$
for values~$* \in \Z_x$.

\subsection{Cohomology of Siegel \texorpdfstring{$3$}{3}-folds}
\label{sec:cohom}
Let $S$ and $Q$ be finite sets of primes of $\Q$ which are disjoint
from each other and do not contain $p$. By a slight abuse of notation,
we will sometimes denote the product of the primes in $Q$ by the same
symbol $Q$.
 For each $x \in S$, let $K_x
\subset \GSp_4(\Z_x)$ equal one of $S(x)$, $\Pi(x)$, $K(x)$, $\Pi(x)^+$,
$I(x)$ or the full congruence subgroup of level $x$. For $x \not \in
S$, we let $K_x = \GSp_4(\Z_x)$ and we define $K :=\prod_x K_x \subset
\GSp_4(\A^\infty)$. For $x\in Q$, we let $K_{x,0}=\Pi(x)$ and
$K_{x,1}= \Pi^+(x)$. Let $K_i(Q)=\prod_{x\not\in Q}K_x \times\prod_{x\in
  Q}K_{x,i}$ for $i=0,1$.

We assume that the subgroup $K$ is neat. (This will be the case if $S$
contains a prime $x \geq 3$ where $K_x$ is the full congruence
subgroup of level $x$.) We let $Y_K\to \Spec(\CO)$ (resp.\ $Y_{K_i(Q)}\to\Spec(\CO)$) denote
the Siegel moduli space of level $K$ (resp.\ $K_i(Q)$). This scheme
classifies principally polarized abelian varieties together with a
$K$-level structure
(resp.\ $K_i(Q)$-level structure). (See \cite[\S 4.1]{Pilloni-hida}.)
In each case we denote the universal abelian variety by $\CA$.

If $Y$ denotes one of the above spaces, we can choose a toroidal
compactification $X \to \Spec(\CO)$ of $Y$. The abelian scheme $\CA$
then extends to a semi-abelian scheme $\pi: \CA \to X$ and the sheaf
$\EE := \pi_*\Omega^1_{\CA/X}$ is a locally free $\CO_{X}$-module of
rank 2. For integers $a \ge b$, we let
$\omega(a,b):=\Sym^{a-b}\EE \otimes \det^{b}\EE$.  We
also~denote~$\det \EE$ by~$\omega$, so, for example,
$\omega(a,a) = \omega^a$ is a line bundle.  If $M$ is an $\CO$-module,
we will let $\omega(a,b)_M$ denote the sheaf
$\omega(a,b)\otimes_{\CO}M$.  The coherent cohomology groups
$H^i(X,\omega(a,b)_M)$ are independent of the choice of toroidal
compactification $X$ (see \cite[Lemma 7.1.1.4]{lan} and the proof of
\cite[Lemma 7.1.1.5]{lan}). 
The Koecher principle states that there is an isomorphism
$$H^0(Y,\omega(a,b)_{M}) \simeq H^0(X,\omega(a,b)_{M}).$$
We may therefore pass freely between the open variety $Y$ and the
(any) smooth projective toroidal compactification $X$ without comment
when dealing with $H^0$.

We choose toroidal compactifications $X_K$ and $X_{K_0(Q)}$ so that
the natural map $Y_{K_0(Q)} \to Y_K$ extends to a map
$X_{K_0(Q)} \to X_K$. As explained in \S~4.1.2 of~\cite{Pilloni-hida},
the universal subgroup $H \subset \CA[Q]$ over $Y_{K_0(Q)}$ extends to
$X_{K_0(Q)}$. We then define the toroidal compactification
$X_{K_1(Q)} = \mathrm{Isom}_{X_{K_0}(Q)}(\Z/Q, H)$. The resulting map
$X_{K_1(Q)} \to X_{K_0(Q)}$ is then finite \'etale with Galois group
$\Delta_Q := (\Z/Q)^\times$.

\subsection{Vanishing results}
\label{sec:vanishing-results}

Let $X$ denote one of the toroidal
compactifications defined in the previous section. We first record some consequences of a
vanishing theorem of Lan and Suh.

\begin{theorem}
  \label{thm:lan-suh}
  \hspace{2em}
  \begin{enumerate}
\item Suppose that $a\geq 3$ and $2\leq a-b \leq p-2$. Then
\[ H^i(X,\omega(a,b)(-\infty)_k)=0\]
for $i>2$.
\item Suppose that $a+b \geq 6$ and $2\leq a-b \leq p-2$. Then
\[ H^i(X,\omega(a,b)(-\infty)_k)=0\]
for $i>1$.

\item Suppose that $b \geq 4$ and $0 \leq a-b \leq p-4$. Then
\[ H^i(X,\omega(a,b)(-\infty)_k) =0 \]
for $i>0$.

  \end{enumerate}
\end{theorem}

\begin{proof}
  This follows from \cite[Cor.\ 7.24]{LanSuh} after unwinding
  definitions. We take the group scheme $\mathrm{G}_1/R_1$ (in the notation of~\cite{LanSuh}) to be our
  $\G/\CO$. The groups $\mathrm{M}_1\subset \mathrm{P}_1 \subset \mathrm{G}_1$ correspond to
  the Siegel Levi and parabolic:  $M \subset P \subset
  \G$. The set of dominant weights $X_{\mathrm{G}_1}^+$ (resp.\
  $X_{\mathrm{M}_1}^+$) is our $X^*(T)^+_{\G}$ (resp.\ $X^*(T)^+_M$)
  from Definition~\ref{defn:dominant-weights}.

  In this paragraph, we show that the subset $X_{\mathrm{G}_1}^{+,<_{\mathrm{re}}p}\subset
  X_{\mathrm{G}_1}^+$ as defined in~\cite[Defn.\ 6.3]{LanSuh-compact}
  corresponds to the set of those $\mu=(a,b;c)\in
  X^*(T)_{\G}^+$ such that $a + b < p - 3$. 
  As an intermediate step, we first show that
  $X_{\mathrm{G}_1}^{+,<_{\mathrm{re}}p}$ corresponds to those $\mu=(a,b;c)\in
  X^*(T)_{\G}^+$ such that:
  \begin{itemize}
  \item $\langle \mu+\rho , \pm\alpha_i^{\vee} \rangle \leq p $ for $i=1,\dots,4$;
  \item $a+b + 3  < p$.
  \end{itemize}
  To see this, we note the following: to lie in
  $X_{\mathrm{G}_1}^{+,<_{\mathrm{re}}p}$, by definition, the element
  $\mu$ must satisfy $|\mu|_L + d < p$ and must also lie in
  $X_{\mathrm{G}_1}^{+,<_{\mathrm{W}}p}$. The definition of $|\mu|_L$
  in Definition
  3.2 of \cite{LanSuh-compact} boils down to $|\mu|_L = a+b$ (the set
  $\Upsilon$ in our case consists of the single embedding $\Z \into \CO$
  and the norm $|\mu| = a+b$ is defined near the beginning of \S 2.5). The dimension $d$ is defined in Definition 3.9 of
  \cite{LanSuh-compact} to be $\dim_{\CO}(X)$ which is 3 in our case.
  Next, the set
  $X_{\mathrm{G}_1}^{+,<_{\mathrm{W}}p}$ is defined in Definition
  3.2 to consist of those $\mu \in
  X_{\mathrm{G}_1}^{+,<p}$ for which $|\mu|_L < p$. Finally, the set
  $X_{\mathrm{G}_1}^{+,<p}$ is defined in Definition 2.29 to consist of all dominant $\mu \in
  X^+_{\mathrm{G}_1}$ which satisfy the first condition above. This
  establishes the intermediate step.
  Now, if $\mu \in X^*(T)^+_{\G}$, then the largest of the
  $\langle \mu + \rho , \pm \alpha_i^{\vee}\rangle$ is $\langle \mu + \rho,
  \alpha_3^{\vee}\rangle = a+b +3$. Thus, we see that $\mu \in
  X_{\mathrm{G}_1}^{+,<_{\mathrm{re}}p}$ if
  and only if $a\geq b\geq0$ and $a+b < p-3$.

  The set $X_{\mathrm{M}_1}^{+,<p}$, by Definition 2.29 of 
\cite{LanSuh-compact},
 is $\{ \mu \in X^*(T)^{+}_M : \langle \mu + \rho,
  \alpha_1^{\vee}\rangle \leq p\} = \{ (a,b;c)\in X^*(T)^+_M :
  (a+2)-(b+1) \leq p\}$. By Lemma 7.2, Definition 7.14 (which is
  vacuous in our case) and Proposition 7.15 of \cite{LanSuh-compact}, a weight
  $\mu = (a,b;c)$ lies in $
  X_{\mathrm{M}_1}^{+,<p}$ and is \emph{positive parallel} if and only if
  $a=b>0$.

  If $\mu = (a,b;c) \in X^*(T)^+_M$, then a pair of vector bundles
  ${\unW}_\mu^*$, for $*\in \{\can,\sub\}$ is defined in
  \cite{LanSuh}. Indeed $\mu$ determines an algebraic representation
  of $M\cong \GL_2\times \GL_1$ over $\CO$ with highest weight
  $(a,b;c)$ (namely $(\Sym^{a-b}S_2 \otimes \det^b S_2 )\otimes
  S_1^{\otimes c}$ where $S_i$ is the standard representation of
    $\GL_i$) and the corresponding bundles are then defined by
  \cite[Defn.\ 4.12]{LanSuh}. We claim that
  \[ {\unW}_{\mu}^{\can} = \omega(a,b).\] (We note that the parameter
  $c$ does not change the underlying vector bundle, but does change
  the Hecke action on cohomology by a power of the similitude
  character.) Let $\mu = (0,-1;1)$, let $L$ denote the standard representation of $\G$
  and $L_0^{\vee}(1)\subset L$ the subspace spanned by the first two
  standard basis vectors. Then $L_0^{\vee}(1)$ is the standard
  representation of the $\GL_2$-factor of $M$ and is the
  representation of $M$ corresponding to $(1,0;0)$. The representation
  $L_0 = (L_0^{\vee}(1))^{\vee}(1)$ thus corresponds to $\mu =
  (0,-1;1)$. By~\cite[Example 1.22]{LanSuh-compact}, we have (in the notation of
  that paper) 
  $\CE_{\mathrm{M}_1}(L_0) =
  \Lie_{\CA/Y}$. However, $\W_{\mu} = \CE_{\mathrm{M}_1}(L_0)$ by
  definition, and we have $\Lie_{\CA/Y} = \CE^{\vee} = \omega(0,-1)$. It follows that
   $\W_{(0,-1;1)}^{\can} \cong \omega(0,-1)$. We
   deduce that $\omega(a,b) = (\Sym^{a-b}\otimes\det{}^b)(\omega(1,0))
   = \W_{(a,b;-a-b)}$, as required.

  With these preliminaries out of the way, we now apply \cite[Cor.\
  7.24]{LanSuh}. We take $\mu = (\alpha,\beta;\gamma) \in
  X_{\mathrm{G}_1}^{+,<_{\mathrm{re}}p}$. (The condition that
  $\max(2,r_{\tau})<p$ when $\tau = \tau\circ c$ boils down to $2<p$
  in our case.) We take $\nu = (t,t;0)$ a positive parallel weight. We
  therefore have $t>0$, $\alpha\geq \beta \geq 0$ and $\alpha+\beta < p-3$.

  We
  now apply part 2 of \cite[Cor.\
  7.24]{LanSuh} successively with $w \in W^{M_1}$ taken to equal each
  of the elements $\tw_1,\tw_2,\tw_3$ from
  Section~\ref{sec:notation}. Note that each $\tw_i$ has length
  $i$. If we take $w = \tw_1$, then (ignoring the third component):
  \[ \tw_1 \cdot \mu - \nu = \tw_1(\alpha+2,\beta+1) - (2,1) - (t,t) = (\alpha-t,
    -\beta-2-t). \]
  Thus $({\unW}^{\vee}_{\tw_1\cdot\mu - \nu})^{\sub} =
  \omega(\beta+2+t, -\alpha +t)(-\infty)$.
  Then \cite[Cor.\
  7.24]{LanSuh} implies
  \[ H^i(X, \omega(\beta+2+t,-\alpha+t)(-\infty)_k) = 0 \]
  for each $i>2$. Taking $a = \beta+2+t$ and $b =-\alpha+t$ gives the
  first part of our proposition.

  Similarly, if $w = \tw_2$, then:
\[ \tw_2 \cdot \mu - \nu = \tw_2(\alpha+2,\beta+1) - (2,1) - (t,t) =
  (\beta-1-t, -\alpha-3-t). \]
  Hence
\[ H^i(X, \omega(\alpha+3+t,1 - \beta+t)(-\infty)_k) = 0 \]
  for $i>1$. This gives the second part of the proposition.

  Finally, we take $w = \tw_3$, then:
  \[ \tw_3 \cdot \mu - \nu = \tw_3(\alpha+2,\beta+1) - (2,1) - (t,t) =
  (-\beta-3-t, -\alpha-3-t). \]
  Hence
\[ H^i(X, \omega(\alpha+3+t, \beta+3+t)(-\infty)_k) = 0 \]
  for $i>0$. This gives the last part of the proposition.
\end{proof}

It is interesting to compare the above vanishing result in characteristic
$p$ with the following characteristic 0 vanishing results due to
Blasius--Harris--Ramakrishnan, Mirkovi\'c, Williams and Schmid. 
We have an identification
\[ Y(\C) = \G(\Q) \backslash (G(\R)/K^h \times G(\A^\infty)/U) \]
where $U\subset \G(\A^\infty)$ is the open compact subgroup used to
define $Y$ and $K^h$ is the compact-mod-center subgroup defined in Section~\ref{sec:group-gsp_4r}. To any finite
dimensional $\C$-representation $(\sigma,V_{\sigma})$ of $K^h_{\C}$, there
is an associated vector bundle $\CV_{\sigma}$ on $Y(\C)$ which is
defined in \cite[Defn.\ 1.3.2]{blasius-harris-ramak}. This bundle has extensions
 $\CV_{\sigma}^{\sub}\subset \CV_{\sigma}^{\can}$ to
$X(\C)$. In \cite{blasius-harris-ramak}, the bundle $\CV_{\sigma}^{\can}$
is denoted $\wt\CV_{\sigma}$. We have
$\CV^{\sub}_{\sigma}=\CV_{\sigma}^{\can}(-\infty)$. For each $i\geq
0$, we define:
\[ \overline{H}^i(X(\C),\CV_{\sigma}) := \im
(H^i(X(\C),\CV_\sigma^\sub)\to H^i(X(\C),\CV_\sigma^\can)).\]
Let~$\Htw^i(\CV_\sigma^\sub)$ and~$\Htw(\CV_\sigma^\can))$ 
denote the direct limit of~$H^i(X(\C),\CV_\sigma^\sub)$ and~$H^i(X(\C),\CV_\sigma^\can))$
respectively over all levels~$K$. Let~$\Htwb^i(X(\C),\CV_{\sigma})$ denote the corresponding 
limit of~$\overline{H}^i(X(\C),\CV_{\sigma})$ (including both an overline and a tilde in the notation
was too cumbersome, hopefully no confusion will result).

Let $\CA_{(2)}(\G)$ denote the space of
automorphic forms on $\G(\Q)\backslash \G(\A)$ which are square integrable modulo the centre
$Z_{\G}(\A)$. Let $\CA_0(\G)\subset \CA_{(2)}(\G)$ denote the space of
cusp forms. For $(\sigma,V_\sigma)$ a representation of $K^h_{\C}$ as
above and $i\geq 0$, we define: 
\begin{align*}
  \CH^i_{(2),\sigma} &= H^i(\Lie Q^-, 
  K^h;\CA_{(2)}(\G)\otimes
  V_{\sigma}) \\
 \CH^i_{\cusp,\sigma} &= H^i(\Lie Q^-, 
   K^h;\CA_{0}(\G)\otimes
  V_{\sigma}).
\end{align*}
Then we have the following result of Harris:

\begin{theorem}
  \label{thm:coherent-cohom-lie-alg-cohom}
There are canonical maps, forming a commutative diagram:
\[
\begin{tikzcd}
   \CH^i_{\cusp,\sigma} \arrow{r}\arrow{d}  & \CH^i_{(2),\sigma}
   \arrow{d} \\
\Htw^i(\CV_\sigma^\sub) \arrow{r} & \Htw^i(\CV_\sigma^\can)
\end{tikzcd}.
\]
Moreover:
\begin{enumerate}
\item The composition $\CH^i_{\cusp,\sigma} \to
  \Htwb^i(\CV_\sigma)$ is injective for all $i$, and
  is an isomorphism for $i=0,3$.
\item The image of $\CH^i_{(2),\sigma}$ in
  $\Htw^i(\CV_\sigma^{\can})$ contains $\Htwb^i(\CV_{\sigma})$.
\end{enumerate}
\end{theorem}

\begin{proof}
  This follows from \cite[Theorem 2.7 \& Prop.\ 3.2.2]{harris-ann-arb}.
\end{proof}

For $*\in \{ \cusp, (2)\}$, we then define
$\Htw^i(\CV_{\sigma}^{\can})_{*}$ to be the image of the space 
$\CH^i_{*,\sigma}$ in $\Htw^i(\CV_\sigma^{\can})$. Thus we have
\[ \CH^i_{\cusp,\sigma}\cong \Htw^i(\CV_{\sigma}^{\can})_{\cusp} \subset
\Htwb^i(\CV_{\sigma}) \subset \Htw^i(\CV_{\sigma}^{\can})_{(2)}.\]

For $*\in \{ \cusp, (2)\}$, the space $\CA_{*}(\G)$ is semisimple as a
$\G(\A)$-representation and we decompose:
\[ \CA_{*}(\G) = \bigoplus_{\pi}m_*(\pi) \pi^{\infty}\otimes
\pi_{\infty}. \]
We let $\CA_{*}(\G)_{\temp}$ denote the subspace
\[ \bigoplus m_*(\pi) \pi^{\infty}\otimes
\pi_{\infty} \]
where the sum is over all those $\pi$ such that $\pi_\infty$ is
essentially tempered. 
We define $\CH^i_{*,\sigma,\temp} \subset \CH^i_{*,\sigma}$ by
replacing $\CA_{*}(\G)$ with $\CA_{*}(\G)_{\temp}$ in the definition
of $\CH^i_{*,\sigma}$. We then define
\[ \Htw^i(\CV_{\sigma}^{\can})_{*,\temp} \subset
\Htw^i(\CV_\sigma^\can)_* \]
to be the image of $\CH^i_{*,\sigma,\temp} \to \Htw^i(\CV_\sigma^{\can})$.
We may also define analogous spaces
\[ H^i(X(\C),\CV_{\sigma}^{\can})_{*,\temp} \subset
H^i(X(\C),\CV_\sigma^\can)_* \]
by applying~$K$-invariants to the constructions above, where~$K$ is the level of~$X(\C)$. 

Suppose now that $(\sigma,V_\sigma)$ is the irreducible representation
of $K^h_{\C}$ of highest weight $\mu = (a,b;c) \in X^*(H_\C)$, with respect to the system of
positive weights fixed in \S~\ref{sec:notation}. We first of all
observe that the bundle $\CV_\sigma$ does not depend on $c$.  Indeed,
let $(\tau,V_\tau)$ be the irreducible representation of highest weight
$(a,b;c+2)$. Consider the $\G(\Q)$-equivariant bundles
$\CV_{\sigma}^{\vee} = \G(\C)\times_{Q^-} V_{\sigma}$ and
$\CV_{\tau}^{\vee}= \G(\C)\times_{Q^-}V_{\tau}$ on $\G(\C)/Q^-$ defined
in \cite[\S 1.3]{blasius-harris-ramak}. (The superscripted $^{\vee}$'s do not
refer to dual bundles here.) Then by the definition of
$\CV_\sigma$, it suffices to show that $\CV_\sigma^\vee \iso
\CV_\tau^\vee$ as $\G(\Q)$-equivariant bundles. 
 We have that $\tau = \sigma \otimes \nu$, so we may take the
 underlying space of $\tau$ to be $V_\tau
= V_{\sigma}$ and the action to be $\tau(g) = \nu(g)\sigma(g)\in \End(V_{\sigma})$ for
all $g \in K^h_{\C}$.  Then the map
\begin{align*}
\G(\C)\times_{Q^-} V_{\sigma} & \longrightarrow \G(\C)\times_{Q^-} V_{\tau} \\
(g,w) & \longmapsto (g,\nu(g)^{-1}w)
\end{align*}
gives the required isomorphism $\CV_\sigma^{\vee} \iso
\CV_\tau^\vee$. (Note however that the Hecke action on the cohomology
of $\CV_{\sigma}$ will depend on $c$ -- changing the value of $c$
introduces a corresponding twist by a power of the similitude
character in the Hecke action.)

For $\mu \in X^*(H_\C)^+_{K^h_{\C}}$ a dominant weight, we let
$\CV_{\mu}$ denote the vector bundle associated to the irreducible
$K^h_{\C}$-representation $W_\mu$. We would like to compare these
bundles to the bundles introduced in the proof of
Theorem~\ref{thm:lan-suh}. 

\begin{df}
  \label{df:lan-suh-bundle-notn}
Let $\mu = (a,b;c)\in
X^*(T)^+_M$. We let $\W_{\mu}$ denote the canonical extension
$\W_{\mu}^{\can}$ in the notation of the proof
of Theorem~\ref{thm:lan-suh}, and we let $\W_{\mu}^{\sub} =
\W_{\mu}(-\infty)$. 
\end{df}
We saw above that, as vector bundles over $X$, we have:
\[ \W_\mu \cong  \omega(a,b), \]
though the Hecke action on the cohomology of $\W_{\mu}$ will depend on
$c$.

\begin{lemma}
  \label{lem:identifying-Vsigma}
Let $\mu = (a,b;c)\in
X^*(T)^+_M$. Then, over $X(\C)$, we have:
\[ \W_{(a,b;c)} \cong \CV_{(-b,-a;a+b+2c)},\]
compatibly with Hecke actions on cohomology.
\end{lemma}

\begin{proof}
It suffices to prove the isomorphism over $Y$. Consider the short
exact sequence:
\[ 0 \lra \Lie_{\CA^{\vee}/Y}^{\vee} \lra \underline{H}_1^{\dR}(\CA/Y)
\lra \Lie_{\CA/Y} \lra 0 \]
and the Poincar\'{e} duality pairing
\[ \langle\ ,\ \rangle : \underline{H}_1^{\dR}(\CA/Y) \otimes
\underline{H}_1^{\dR}(\CA/Y) \lra \CO_Y(1). \]
(See \cite[\S 1.2]{LanSuh-compact}). 

Expressed in terms of the functor $\W_{\mu}$ of Lan--Suh, the short exact sequence
becomes:
\[ 0 \lra \W_{(1,0;0)} \lra \underline{H}_1^{\dR}(\CA/Y)
\lra \W_{(0,-1;1)} \lra 0 \]
and the bundle $\CO_Y(1)$ becomes $\W_{(0,0;1)}$. (See \cite[Example 1.22]{LanSuh-compact}.) 

Similarly, over $Y(\C)$ the
short exact sequence becomes
\[ 0 \lra \CV_{(0,-1;1)} \lra \underline{H}_1^{\dR}(\CA/Y)
\lra \CV_{(1,0;1)} \lra 0 \]
and $\CO_Y(1)$ is identified with $\CV_{(0,0;2)}$. This follows
from \cite[Example III.2.4]{milne-ann-arb}: if we take the point $o\in \check{X}$ to be
$h(i) = J$ in the notation of Section~\ref{sec:group-gsp_4r}, then the
isotropic subspace corresponds to $V^-$ and $V/W$ corresponds to
$V^+$. As remarked at the end of Section~\ref{sec:group-gsp_4r}, we
have $V^- = W_{(0,-1;1)}$, $V^+ = W_{(1,0;1)}$ and the similitude
character corresponds to $W_{(0,0;2)}$. Note also that the notation
$\CH_{\dR}(\CA)$ of \cite{milne-ann-arb} refers to de Rham homology
(see \S I.3).

It follows that, over $Y(\C)$, we have $\W_{(0,0;1)} = \CV_{(0,0;2)}$ and $\W_{(1,0;0)} =
\CV_{(0,-1;1)}$. Thus,
\begin{align*}
  \W_{(a,b;c)} & = (\Sym^{a-b}\otimes \det{}^b)(\W_{(1,0;0)}) \otimes
  \W_{(0,0;c)} \\
  & = (\Sym^{a-b}\otimes \det{}^b)(\CV_{(0,-1;1)}) \otimes
  \CV_{(0,0;2c)} \\
  & = \CV_{(-b,-a;a+b+2c)}
\end{align*}
This is compatible with Hecke action on cohomology since
all isomorphisms respect the equivariant constructions.
\end{proof}

The Weyl chambers $C_0,\dots,C_4 \subset X^*(T)\otimes_{\Z}\R \cong \R^3$ are
defined in Section~\ref{sec:group-gsp_4}. We have
\begin{eqnarray*}
  C_0 &=& \{ (a,b;c) \in \R^3 : a \geq b \geq 0\}\\
  C_1 &=& \{ (a,b;c) \in \R^3 : a \geq -b \geq 0\}\\
  C_2 &=& \{ (a,b;c) \in \R^3 : -b \geq a \geq 0\}\\
  C_3 &=& \{ (a,b;c) \in \R^3 : -b \geq -a \geq 0\}.
\end{eqnarray*}

\begin{theorem}
  \label{prop:char-0-vanish}
Let $\mu = (a,b;c) \in X^*(T)^+_M$. Then:
\[ H^i(X(\C),\W_{\mu})_{(2),\temp}  = 0\]
for all $0\leq i \leq 3$ such that 
\[ \mu = (a-1,b-2;c) \not \in C_i. \]
\end{theorem}

\begin{proof}
In Section~\ref{sec:group-gsp_4r}, we identified $X^*(H_{\C})$ with
$\Z^3$. Under the resulting identification of
$X^*(H_{\C})\otimes_{\Z}\R$ with  $\R^3$, the chambers $C_i$ for $X^*(T)^+_M\otimes_{\Z}\R$ 
above correspond to Weyl chambers in  $X^*(H_{\C})\otimes_{\Z}\R $. 
Let $\sigma = (-b,-a;a+b+2c)$, regarded as an element of
$X^*(H_{\C})$. Then we've seen above that
\[ \CV_{\sigma} \cong \W_{\mu}. \]
Suppose that 
\[ H^i(X(\C),\W_{\mu})_{(2),\temp} = H^i(X(\C),\CV_{\sigma})_{(2),\temp}
\neq 0.\]
Then by Theorem
\ref{thm:coherent-cohom-lie-alg-cohom},
there is some $\pi  = \pi^{\infty}\otimes \pi_\infty$ in
$\CA_{(2)}(G)_{\temp}$ such that
\[ H^i(\Lie P^{-}, K^h; \pi_\infty \otimes V_{\sigma})) \neq 0 .\]
By a theorem of Mirkovi\'{c} \cite[Theorem 3.5]{harris-ann-arb}, $\pi_\infty$ is a discrete series or limit of discrete
series. Hence, using the Harish-Chandra parameterization, we may write
$\pi_{\infty} = \pi(\lambda,C)^* = \pi(-w_0(\lambda),-w_0(C))$ for some
Weyl Chamber $C \in \{ C_0,\dots,C_3\}$ and a weight $\lambda \in C \cap
\left(X^*(H_{\C})+\rho\right)$.
 By \cite[Theorem 3.2.1]{blasius-harris-ramak}, it follows
that:
\[ \lambda = ((\sigma + \rho)|_{\Sp_4(\R)}; -a-b-2c) =
(2-b,1-a;-a-b-2c), \]
and 
\[  i = \# \left(\Phi(C)^+\cap \Phi_n^+ \right),\] 
where $\Phi(C)^+$ is the system of positive roots determined by the
chamber $C$.  
For
$j=0,\dots,3$, we have $\# \left(\Phi(C_j)^+\cap \Phi_n^+ \right) =
3-j$. Hence we must have $C = C_{3-i}$ and $\lambda \in
C_{3-i}$. However, $C_{3-i} = -w_0(C_i)$, so $-w_0(\lambda) \in
C_i$. We have, have:
\begin{align*}
  -w_0(\lambda) & = -w_0 (-b+2,-a+1;-(a+b+2c)) \\
  & = (a-1, b-2; a+b+2c).
\end{align*}
Thus, we deduce that 
$-w_0(\lambda) = (a-1,b-2;a+b+2c)$ lies in $C_i$. This is equivalent
to the condition in statement of the theorem. 
\end{proof}

We also record the following:
\begin{theorem}
  \label{thm:contrib-to-char-0-cohom}
Let $\mu = (a,b;c) \in X^*(T)^+_M$, let $ w = -(a+b+2c)$, and let $\sigma
= (-b,-a;a+b+2c) = (-b,-a;-w)$, regarded as an element of  $X^*(H_{\C})$. Suppose that $\pi  =
\pi^{\infty}\otimes \pi_\infty$ in $\CA_{(2)}(G)$ contributes
to  $H^i(X(\C),\W_{\mu})_{(2)} \cong H^i(X(\C),\CV_{\sigma})_{(2)}$.
\begin{enumerate}
\item The infinitesimal character of
$\pi_{\infty}$ is given under the Harish-Chandra isomorphism by:
\[ \chi_{((- \sigma - \rho)|_{\Sp_4(\R)};-w)} = \chi_{(a-1,b-2;-w)} .\]
\item Let $\widetilde{\pi}_{\infty}$ denote the transfer of
  $\pi_{\infty}$ to $\GL_4(\R)$. Then the infinitesimal character of
  $\widetilde{\pi}_{\infty}$ is given under the Harish-Chandra
  isomorphism by $\chi_{\tau}$ where:
\[ \tau =  \left(\frac{a+b-3-w}{2}, \frac{a-b+1-w}{2}, \frac{-a+b-1-w}{2}, \frac{-a-b+3-w}{2}\right).  \]
\item\label{Mirkovic} If furthermore, $\pi_{\infty}$ is tempered, then $\pi_{\infty}$ is a
discrete series or limit of discrete series representation, and is
given under the Harish-Chandra parameterization by:
\[ \pi_{\infty} \cong \pi((a-1,b-2;-w), C_i). \]
\end{enumerate}
\end{theorem}

\begin{proof}
  For the first part, we have that
\[ H^i(\Lie P^{-}, K^h; \pi_{\infty}\otimes V_{\sigma}) \neq 0.\] 
It follows from \cite[Theorem 3.2.1]{blasius-harris-ramak} that the
infinitesimal character of $\pi_{\infty}$ is equal 
to~$\chi_{((-\sigma-\rho)|_{\Sp_4(\R)};-w)}$. 
The second part can be inferred from \cite[\S
2.1.2]{Sor}.
The last part is due to Mirkovi\'c and was established in the proof of Theorem~\ref{prop:char-0-vanish}.
\end{proof}

\begin{df}
  \label{defn:discrete-series-weight}
A weight $\mu = (a,b;c) \in X^*(T)^+_M$ such that $(a-1,b-2;c)$
lies in the interior of a unique Weyl chamber $C_i$ is said to be a \emph{discrete series
  weight} or a \emph{regular weight}. If $\mu - w_0(\rho)$ lies in the
intersection of exactly two of Weyl chambers $C_i$, we say it is a \emph{limit of
  discrete series weight} or a \emph{non-regular weight}.
\end{df}
From the explicit description of the Weyl chambers $C_i$ above, we see
that the limit of discrete series weights thus come in 3
families:
\begin{align*}
  \mu &= (a,2;c) \text{ with $a\in \Z_{\geq 2}$} \\
  \mu &= (a,3-a;c) \text{ with $a\in \Z_{\geq 2}$} \\
  \mu &= (1,b;c) \text{ with $b\in \Z_{\leq 1}$}.
\end{align*}
Note that for the  corresponding families of vector bundles
$\W_{\mu}=\omega(a,b)$, the first and third are interchanged under the
Serre duality map $\omega(a,b) \mapsto \omega(a,b)^{{\vee}}\otimes \det \Omega^1_{X}\cong\omega(3-b,3-a)(-\infty)$ while the second
family is stable under this operation. (Up to interchanging the canonical and
subcanonical extensions, of course.) The preceding theorem implies
that for all $a \geq 2$, we have:
\begin{eqnarray*}
   H^i(X(\C),\omega(a,2))_{(2),\temp}  &=& 0 \mbox{\;\; for $i=2,3$}\\
   H^i(X(\C),\omega(a,3-a))_{(2),\temp} &=& 0 \mbox{\;\; for $i=0,3$}.
\end{eqnarray*}
(Technically, we should normalize the Hecke action on the cohomology
of $\omega(a,b)$ before we adjoin the subscripts $(2)$ or $\temp$. See
Section~\ref{sec:hecke-operators} below.)
From the result of Lan and Suh,
we deduce the following characteristic $p$ analogue of these vanishing results for
limit of discrete series weights.

\begin{corr}
  \label{cor:ls-vanish-lds}
  \begin{enumerate}
  \item For $4\leq a \leq p$, we have
\[ H^i(X,\omega(a,2)(-\infty)_k)=0\]
for $i=2,3$.
\item For $3 \leq a \leq (p+1)/2$, we have
\[ H^0(X,\omega(a,3-a)_k)= H^3(X,\omega(a,3-a)(-\infty)_k)=0.\]
  \end{enumerate}
\end{corr}

\begin{proof}
The vanishing results for the subcanonical extensions
$\omega(*,*)(-\infty)$ follow directly from Theorem~\ref{thm:lan-suh}. The fact that
\[ H^0(X,\omega(a,3-a)_k) =0 \]
in the second part then follows from Serre duality since:
\[ \omega(a,3-a)^{\vee}\otimes \det \Omega^1_{X/\CO} =
\omega(a-3,-a)\otimes \omega(3,3)(-\infty) = \omega(a,3-a)(-\infty).\] 
\end{proof}

\subsection{Torsion Classes} 
It seems natural to ask
whether one can (explicitly or otherwise) construct classes in $H^0(X,\omega(2,2))$ which
do not lift to characteristic zero.
Let us recall what happens for classical modular forms of weight one.

Suppose that $X_1(N)$ denotes (for this paragraph) the classical modular curve. A non-Eisenstein
Hecke eigenclass
in $H^0(X_1(N),\omega_k)$ gives rise to an irreducible Galois representation
$\rbar:G_{\Q} \rightarrow \GL_2(k)$. Suppose that
 the image of $\rhobar$ contains $\SL_2(k')$ for some $\# k' > 5$.
 Such a representation cannot be the mod-$p$ reduction of a representation
with image isomorphic to some subgroup of $\GL_2(\C)$, and thus
by~\cite{DeligneSerre}, the corresponding mod-$p$ class does
not lift to characteristic zero. (Explicit examples were first found by Mestre for $\#k = 8$ and $N = 1429$.)
A slightly different example can be given as follows.
Suppose that $\Gamma = \Gamma_1(N) \cap \Gamma_0(x)$.
Consider a non-Eisenstein Hecke eigenclass in $H^1(X(\Gamma),\omega_k)$ which is new of level $x$. Then
the restriction of $\rbar$ to $I_x$ is rank two unipotent. Such a class cannot lift to characteristic zero
at minimal level, because otherwise (by~\cite{DeligneSerre} again) the corresponding representation
$\rho$ would simultaneously have finite image and yet $\rho|I_x$ would be unipotent and hence infinite.
Note that (unlike in the first example) it may well be possible to lift $\rhobar$ to characteristic zero at
some non-minimal level.
Examples of the second kind have a natural analogue in the Siegel context.

Suppose that
 $\rbar$ has type {\bf U3\rm} at $x$. If $r$ is any minimal 
lift of $\rbar$, the image of $I_x$ under $r$ will be rank three unipotent.
This will also be true for the restriction of $r$ to any finite extension of $\Q_x$.
Yet, by a theorem of Grothendieck (\cite{SGA7}, Exp.9)  the image of inertia of a semistable
abelian variety is rank two unipotent, i.e.,  satisfies $(\sigma-1)^2 = 0$.
If follows that $r$ cannot contribute to a motive associated to an abelian
variety. Conjecturally, Siegel modular eigenforms of weight $(2,2)$ should be associated
to abelian varieties $M/\Q$ of dimension $2n$ equipped with an injection
$E \rightarrow \End_{\Q}(M) \otimes \Q$ for some totally real field $E$ of degree $n$.
This suggests that such representations $\rbar$  \emph{do not} admit minimal lifts
to characteristic zero when $\sigma = (2,2)$.
 It would be interesting to produce an explicit example of such
a modular representation. Recall that there is an  exceptional isomorphism 
$S_6 \simeq \GSp_4(\F_2)$ coming from identifying the Galois group
of $\mathcal{A}_2[2]$ over $\mathcal{A}_2$ with either the symmetries of the
$2$-torsion points on the universal abelian surface or the action of $S_6$ on 
the (generically) $6$ Weierstrass points~\cite{Faber}. The unipotent element
$\sigma \in \GSp_4(\F_2)$ such that $(\sigma-1)^2 \ne 0$ has conjugacy class
$(1,2,3,4)(5,6) \in S_6$ (this class is preserved by the exotic automorphism of $S_6$).
In particular, if $K/\Q$ is a sextic field with Galois closure $G \subset S_6$
containing $(1,2,3,4)(5,6)$ and acting irreducibly on $\F^4_2$, and
$p$ is an odd prime such that $p = \mathfrak{p}^4 \mathfrak{q}^2$, then $\rbar:\Gal(K/\Q) \simeq \GSp_4(\F_2)$
should give rise to such a representation.
Here is an explicit example coming
from a slight variation of this argument. Suppose that $A$ is the abelian surface corresponding to the Jacobian of the curve:
$$y^2 =x^5 - 2x^4 + 6x^3 - 8x^2 + 4x - 4,$$
which has good reduction outside~$3 \cdot 5 \cdot 19$.
The representation $\rbar:G_{\Q} \rightarrow \GSp_4(\F_2)$ has image~$S_5 \subset S_6$,
and  the image of inertia at~$5$ is conjugate to~$(1,2,3,4)(5,6)$.
Hence~$\rbar$ should give rise to a mod-$2$ torsion class
with trivial level structure outside $3 \cdot 5 \cdot 19$, and the following level structure at these primes:
\begin{enumerate}
\item Iwahori level structure at $p = 5$,
\item Paramodular level structure at $p = 3$ and $19$.
\end{enumerate}
Note that this conjectural torsion class \emph{does} conjecturally lift to characteristic zero at some level since one expects that~$A$ is modular.
(The conductor of~$A$ is~$3 \cdot 5^3 \cdot 19$.)

Common to both examples is the non-existence of  automorphic
representations $\pi$  (associated to either classical modular forms of weight $1$
 or Siegel modular forms of weight $(2,2)$) such that $\pi_x$ is the Steinberg representation.
 For classical modular forms, the non-existence of such $\pi$ follows from a consideration
 of the corresponding Galois representations, 
 an argument which does not obviously generalize to the Siegel case (since one does
 not know how to attach an abelian variety to such a form). However,
  following argument (due to Kevin Buzzard) generalizes nicely:
  
 \begin{theorem} If $\pi$ is a cuspidal automorphic representation associated to a Siegel modular form
 of weight $(2,2)$, then $\pi_x$ is not the Steinberg representation for any $p$.
 \end{theorem}
 
 \begin{proof}
 In weights $(j,k)$ with $j \ge k \ge 2$, the corresponding Frobenius
 eigenvalues of the Weil--Deligne representation
 associated to a Steinberg representation~$\pi_x$ are
 $$\{x^{(w+3)/2}, x^{(w+1)/2}, x^{(w-1)/2}, x^{(w-3)/2}\},$$
  where $w = j + k-3$. Moreover, the
 corresponding eigenvalue of $U_{x,1}$ is $x^{(w-3)/2}$.
 In particular, if $j=k=2$,
 then $w = 1$ and the corresponding eigenvalue of $U_{x,1}$ is $x^{-1}$,
 contradicting the integrality of Hecke eigenvalues (which is
 a consequence of the integrality of the $q$-expansion).
 \end{proof}

\subsection{Hecke operators}
\label{sec:hecke-operators}

For simplicity, we denote the schemes $X_K$ and $X_{K_i(Q)}$ of
\S~\ref{sec:cohom} by $X$ and $X_i(Q)$ respectively. Let $M$ denote an $\CO$-module.

Let $x$ be a rational prime. We define matrices 
\[ \beta_{x,0} = \begin{pmatrix} x &
  0&0&0\\0&x&0&0\\0&0&x&0\\0&0&0&x\end{pmatrix} 
\quad 
\beta_{x,1} = \begin{pmatrix} 1 &
  0&0&0\\0&1&0&0\\0&0&x&0\\0&0&0&x
\end{pmatrix}
\quad 
\beta_{x,2} = \begin{pmatrix} 1 & 0&0&0\\0&x&0&0\\0&0&x&0\\0&0&0&x^2\end{pmatrix}
\] and regard them as
elements of $\GSp_4(\Q_x)$. 
If $x\not \in S$ (resp.\ $x\not\in S\cup Q$) we will consider the Hecke operators
$T_{x,i}=[K\beta_{x,i}K]$ (resp.\
$T_{x,i}=[K_i(Q)\beta_{x,i}K_i(Q)]$) acting on each of the spaces 
\[ H^n(X,\omega(a,b)_M) \text{ (resp.\ $H^n(X_i(Q),\omega(a,b)_M)$)}
  \] 
  as in \cite[\S 1.1.6]{skinner-urban-p-adic} or \cite[\S
  8]{tilouine-bgg}. We also denote $T_{x,0}$ by $S_x$. The definition
  of Hecke operators given in \cite{skinner-urban-p-adic} or
  \cite{tilouine-bgg} applies when $x \neq p$ or when $p$ is
  invertible on $M$. The remaining cases when $x = p$ requires more
  care. In Lemma~\ref{lemma:Gross} below we show that $T_{p,1}$ and
  $Q_{p,2}: = (pT_{p,2}+ (p + p^3)S_p)p^{2-b}$ exist as operators in
  cohomological degree $n=0$ over $M = K/\CO$.
  
Similarly, if $x\in Q$, we have operators
$U_{x,i}=[K_i(Q)\beta_{x,i}K_i(Q)]$ on $H^n(X_i(Q),\omega(a,b)_M)$. As
in \S~\ref{sec:cohom}, the map $X_1(Q)\to X_0(Q)$ is Galois with Galois group
$\Delta_Q:=\prod_{x\in Q}(\Z/x)^\times$.  This gives rise
to an action of $\Delta_Q$ on
$H^n(X_1(Q),\omega(a,b)_M)$. For each $u\in \Delta_Q$, we
denote the corresponding operator on $H^n(X_1(Q),\omega(a,b)_M)$ by
$\langle u \rangle$.

Finally, we shall also exploit  Hecke operators of a slightly
different flavour, which we denote by~$U_{p,1}$
and~$U_{p,2}$ respectively. In the context of this paper,
they may be considered formal operators on~$q$-expansions.
(They can also be interpreted more classically as Hecke operators
with level structure at~$p$.)
Their key property is that the operators~$T_{p,1}$ 
and~$T_{p,2}/p^{k+j-6}$ act by~$U_{p,1}$ and~$U_{p,2}$
for large enough weights, including~$(j,k)$ plus any non-trivial multiple of~$(p-1,p-1)$
for~$j \ge k \ge 2$. Their explicit definition in given in
Lemmas~\ref{eightthree} and \ref{eightfour}.

\begin{remark}
  \label{rem:normalization}
  We note that our definition of the Hecke action is the `natural' one
  twisted by $\nu^{-3}$ (see \cite[1.1.6a]{skinner-urban-p-adic}). We
  saw in the proof of Theorem~\ref{thm:lan-suh}, that for the natural
  action, there is an isomorphism $\omega(a,b) \cong \W_{(a,b;-a-b)}$,
  and hence over $\C$, an isomorphism $\omega(a,b) \cong
  \CV_{(-b,-a;-a-b)}$. Under our normalization of the Hecke action on
  $\omega(a,b)$, we therefore have $\omega(a,b) \cong \W_{\mu}$ and,
  over $\C$, $\omega(a,b) \cong \CV_{\sigma}$ where we take:
  \[ \mu = (a,b;3-a-b) \quad \mbox{and}\quad \sigma =
  (-b,-a;6-a-b). \] 
\end{remark}

\begin{remark}
  \label{rem:convention}
  In view of the previous remark, we will identify the set $\Z^{2,+}:=\{ (a,b)\in \Z^2: a\geq b\}$
  with the subset $(a,b;3-a-b)$ of $X^*(T)^+_M$. Thus it makes sense
  to speak of $\mu =(a,b) \in X^*(T)^+_M$.
\end{remark}

\begin{remark}
  \label{rem:central-chars}
  Let $\mu = (a,b) \in X^*(T)^+_M$ and let $w = a+b-6$. For $x \in
  \Z$, we can similarly define a Hecke operator associated to
  $[K\diag(x,x,x,x)K]$ on the cohomology of $\omega(a,b)$: this
  operator acts as $x^w = x^{a+b-6}$.  Now, suppose that $\pi =
  \pi^{\infty}\otimes \pi_\infty$ in $\CA_{(2)}(G)$ contributes to
  \[ H^i(X(\C),\omega(a,b))_{(2)} 
  \cong H^i( X(\C),\W_{\mu})_{(2)}
  \cong H^i( X(\C),\CV_{\sigma})_{(2)},\]
  where $\sigma = (-b,-a;-w)$. It follows that the central character of
  $\pi_{\infty}$ is given by:
  \[ x \mapsto x^{-w}.\]
  Furthermore, by Proposition~\ref{thm:contrib-to-char-0-cohom}, the
  transfer of $\pi_{\infty}$ to $\GL_4(\R)$ has infinitesimal
  character $\chi_{\tau}$ where
 \[ \tau = \left(0, -(b-2), -(a-1), -(a+b-3)\right)
  +3/2(1,1,1,1).  \]
\end{remark}

We now introduce some Hecke algebras. We note that in the following
definition, we work over $K/\CO$ rather than $\CO$.

\begin{df}
  \label{defn:hecke-alg} Let~$\mu = (a,b)\in X^*(T)^+_M$ with $a \ge b \ge 2$.
  \begin{enumerate}
  \item  The anaemic Hecke algebra 
\[ \Tan_{\mu}(Q) \subset \End_{\CO}(H^0(X_1(Q),\omega(a,b)(-\infty)_{K/\CO})) \]
is the $\CO$-algebra generated by the operators $T_{x,i}$ for $x \not\in
S\cup Q\cup \{p\}$.
\item Similarly, we let $\T_{\mu}(Q)$ be the algebra generated over $\Tan_{\mu}(Q)$
by the operators $U_{x,i}$ for $x\in Q$ and $\langle u \rangle$ for
$u\in \Delta_Q$. When $Q=\emptyset$, we have $ \Tan_{\mu}(\emptyset) =
\T_{\mu}(\emptyset)$ and we denote this algebra by $\T_{\mu}$.
\item Finally, $\WT_{\mu}(Q)$ denotes the
  $\T_{\mu}(Q)$-algebra generated by the operators~$T_{p,1}$ and~$Q_{p,2} = (p T_{p,2}
  + (p + p^3) S_p) p^{2-b}$. (The existence of these operators is
  established in Lemma~\ref{lemma:Gross}.)
  If $Q  =\emptyset$, then we denote $\WT_{\mu}(\emptyset)$ by $\WT_{\mu}$.
\end{enumerate}
\end{df}

Note that the algebras $\Tan_{\mu}(Q) \subset \T_{\mu}(Q) \subset
\WT_{\mu}(Q)$ preserve the subspace
\[ H^0(X_0(Q),\omega(a,b)(-\infty)_{K/\CO}) \subset
  H^0(X_1(Q),\omega(a,b)(-\infty)_{K/\CO}). \]

  We will also need to consider ordinary Hecke algebras.
  Let $e = \varinjlim_n (T_{p,1}Q_{p,2})^{n!}$ denote the ordinary
  idempotent associated to the Hecke operators $T_{p,1}$ and
  $Q_{p,2}$. (We will only consider this operator in contexts where
  the direct limit makes sense.) We define:
\[ H^0(X_0(Q),\omega(a,b)(-\infty)_M)^{\ord} = e
H^0(X_0(Q),\omega(a,b)(-\infty)_M) \]
for $M = \CO, \CO/\varpi^m$ or $M = K/\CO$. We thus have:
{\small
\[ H^0(X_0(Q),\omega(a,b)(-\infty)_M) =
H^0(X_0(Q),\omega(a,b)(-\infty)_M)^{\ord} \bigoplus
(1-e)H^0(X_0(Q),\omega(a,b)(-\infty)_M)\]
}
for such $M$.

\begin{df}
  \label{df:ordinary-hecke-alg}
  Let~$\mu = (a,b)$ with $a \ge b \ge 2$. We define the ordinary Hecke algebras
  $\Tan_{\mu}(Q)^{\ord}$ (resp.\ $\T_{\mu}(Q)^{\ord}$,
  $\wT_{\mu}(Q)^{\ord}$) to be the image of $\Tan_{\mu}(Q)$ (resp.\ $\T_{\mu}(Q)$,
  $\wT_{\mu}(Q)$) in 
\[ \End_{\CO}(H^0(X_0(Q),\omega(a,b)(-\infty)_{K/\CO})^{\ord}). \]
\end{df}

\section{Galois representations associated to modular forms}

As in Section~\ref{sec:cohom}, let $S$ and $Q$ be finite sets of
primes of $\Q$ which are disjoint and do not contain $p$. We allow
the possibility that $Q=\emptyset$. We let $K$ and $K_{i}(Q)$ be open
compact subgroups of $\GSp_4(\A^{\infty})$ as in
Section~\ref{sec:cohom}, and we let $X = X_K$ and $X_i(Q) = X_{K_i(Q)}$ be the
corresponding Siegel threefolds, defined over $\CO$.

\subsection{The Hasse invariant}
\label{sec:hasse-invariant}

We begin with a definition.

\begin{df}
  \label{defn:hasse-invts}
Let $h \in H^0(X, \omega^{p-1}_k)$ be the Hasse invariant and
let $A \in H^0(X, \omega^{r(p-1)})$ be a lift of $h^r$,
for some $r>0$ which we fix for the rest of this section.  
\end{df}

The existence of such a lift $A$ follows
from the Koecher principle and the ampleness of $\omega$ on the
minimal compactification of $X$.

\begin{lemma}
  \label{lem:hasse-mod-p}
Let $\mu = (a,b)\in X^*(T)_M^+$ with $a\geq b \geq 2$. Then:
\begin{enumerate}
\item Multiplication by $h$ defines an injection:
\[ h: H^0(X_1(Q),\omega(a,b)_k) \into
H^0(X_1(Q),\omega(a+(p-1),b+(p-1))_k) \]
which is equivariant for the Hecke operators $T_{x,i}$ for each $x
\not\in S\cup Q \cup \{p\}$ and the operators $U_{x,i}$ for $x \in Q$.
\item\label{ops-at-p}  If $b\geq 3$, then this map is also
equivariant for the operators $T_{p,1}$ and
$Q_{p,2}$.
\end{enumerate}
\end{lemma}

\begin{proof}
  It is well-known that multiplication by $h$ is injective and commutes
with Hecke operators away from $p$. We may thus assume that $b \geq
3$. It is shown in \cite[\S A.3]{Pilloni-hida} and \cite[Lemme
  8.7]{tilouine-bgg} that multiplication by $h$ commutes with the
operators $U_{p,1}$ and $U_{p,2}$. Since $b \geq 3$,
\cite[Lemme 8.5]{tilouine-bgg} implies that $T_{p,1}\equiv U_{p,1} \mod
p$ and $p^{3-b}T_{p,2} \equiv U_{p,2} \mod p$. It follows that
$T_{p,1}$ and $Q_{p,2} = p^{3-b}T_{p,2} + (1+p^2)p^{3-b}S_p$ also
commute with $h$.
\end{proof}

Suppose that $\mu = (a,b) \in X^*(T)_M^+$ with $a\geq b \geq 2$. By the proof of
\cite[Th\'eor\`eme 6.2]{Pilloni-hida}, there exists an integer
$N(\mu)$ as in the following definition.

\begin{df}
  \label{df:suff-large-wt}
  Let $N(\mu)$ be an integer such that for all $t \geq N(\mu)$,
  $i>0$, and $Z \in \{ X, X_0(Q), X_1(Q)\}$, the cohomology group
  \[ H^i(Z, \omega(a+t, b+t)(-\infty)_k)\] vanishes.
  \end{df}
Note that for such $t \geq N(\mu)$, 
the maps
\begin{eqnarray*}
    H^0(X,\omega(a+t, b+t)(-\infty)) &\to& H^0(X, \omega(a+t, b+t)(-\infty)_k) \\
 H^0(X,\omega(a+t, b+t)(-\infty)_K) &\to& H^0(X, \omega(a+t, b+t)(-\infty)_{K/\CO})
\end{eqnarray*}
are both surjective. The same is true over $X_0(Q)$ and $X_1(Q)$.

\begin{lemma}
\label{lem:hasse-mod-p^m}
Let $\mu = (a,b)\in X^*(T)_M^+$ with $a\geq b \geq 2$ and let $m>0$. There exists
an integer $s>0$ such that, if we set $t = rs(p-1)$, then:
\begin{enumerate}
  \item $t \geq N(\mu)$, and
 \item multiplication by $A^s$ defines an injection
\[ H^0(X_1(Q),{\omega(a,b)}_{\CO/\varpi^m}) \into
H^0(X_1(Q),{\omega(a+t,b+t)}_{\CO/\varpi^m}) \]
which is equivariant for the Hecke operators $T_{x,i}$ for each $x
\not\in S\cup Q \cup \{p\}$ and the operators $U_{x,i}$ for each $x
\in Q$.
 \end{enumerate}  
\end{lemma}

\begin{proof}
  The second property holds as long as $p^{m-1} | s$ (see \cite[Theorem
  6.2.1]{goldring}), so it suffices to take $s$ equal to any integer
  greater than $N(\mu)/r(p-1)$ and divisible by $p^{m-1}$.
\end{proof}

Let $\mu = (a,b)\in X^*(T)_M^+$ with $a\geq b \geq 2$. Recall that the Hecke algebras  
\[ \Tan_{\mu}(Q) \subset \T_{\mu}(Q) \subset
\WT_{\mu}(Q) \subset  \End_{\CO}(H^0(X_1(Q),\omega(a,b)(-\infty)_{K/\CO})) \] 
were defined in Definition~\ref{defn:hecke-alg}.

\begin{remark}
  \label{rem:hasse-hecke}
For $\mu = (a,b)$  with~$a \ge b \ge 2$ and each $m >0$, we have
$$H^0(X_1(Q),{\omega(a,b)(-\infty)}_{\CO/\varpi^m})\cong
H^0(X_1(Q),{\omega(a,b)(-\infty)}_{K/\CO})[\varpi^m].$$ Let
$I_{\mu,m}$ (resp.\ $\wI_{\mu,m}$)
denote the annihilator of the former space in $\T_{\mu}(Q)$ (resp.\ $\wT_{\mu}(Q)$). If
$s$ and $t$ are as in Lemma~\ref{lem:hasse-mod-p^m}, then
multiplication by $A^s$ induces a surjective map:
\[ \T_{\mu'}(Q) \onto \T_{\mu}(Q)/I_{\mu,m}, \]
where $\mu' = \mu + (t,t)$.
In particular, any maximal ideal $\m$ of $\T_{\mu}(Q)$ pulls back
under this map to a maximal ideal of $\T_{\mu'}(Q)$ which we will also
denote by $\m$.

Similarly, Lemma~\ref{lem:hasse-mod-p} induces a map
\[ \T_{\mu'}(Q) \onto \T_{\mu}(Q)/I_{\mu,1} \]
where $\mu' = \mu + (p-1,p-1)$ and, if $b\geq 3$, this extends to a
map
\[ \wT_{\mu'}(Q) \onto \wT_{\mu}(Q)/\wI_{\mu,1}. \]
\end{remark}

\subsection{Preliminaries on Galois representations}
\label{sec:prel-galo-repr}

We now turn our attention to Galois representations.

\begin{prop}
  \label{prop:similitude}
Let $\mu = (a,b) \in X^*(T)_M^+$ and let $w = a+b-6$. 
There is a continuous character
\[ \chi_{\mu} : G_{\Q} \to \Tan_{\mu}(Q)^\times \]
such that:
\begin{enumerate}
\item $\chi_{\mu}|G_{\Q_p}$ is crystalline with Hodge--Tate weight $w$;
\item for all $x\not\in S\cup Q\cup\{p\}$, $\chi_{\mu}$ is unramified
  at $x$ and $\chi_{\mu}(\Frob_x) = S_x$.
\end{enumerate}
In particular, 
\[ \chi_{\mu} = \chi_{\mu,0} \eps^{-w} \]
for some finite order character $\chi_{\mu,0}:  G_{\Q} \to \wT_{\mu}(Q)^\times$.
\end{prop}

\begin{proof}
  This follows from the proof of \cite[Proposition 4]{TaylorDuke},
  noting that we have twisted the Hecke action by $\nu^{-3}$ (see Remark~\ref{rem:central-chars}).
\end{proof}

\begin{df}
  \label{df:Hecke-polynomial}
For a prime $x$, we introduce the Hecke polynomial:
\[ Q_x(T) = X^4 - T_{x,1} X^3 + (x T_{x,2} + (x^3 + x) S_x) X^2 -
 x^3 S_xT_{x,1} X + x^6 S_x^2. \]
\end{df}

If a modular form $f$ is an eigenform for a collection of Hecke operators $T$, we
denote by $\lambda_f$ the map such that $T f = \lambda_f(T)f$ for each
$T$. In particular, if $f$ is an eigenform for the operators $T_{x,i}$
at $x$, then we can specialize the polynomial $Q_x(T)$ at $f$ to get $\lambda_f(Q_x(T))$.

\begin{prop}
  \label{prop:weiss}
Let $\mu = (a,b) \in X^*(T)_M^+$ with $a \geq b \geq 3$. Let $w =
a+b-6$ and $\w = w+3 = a+b-3$.
Let 
\[ f \in H^0(X_1(Q), \omega(a,b)(-\infty)) \]
be a cuspidal eigenform for the operators $T_{x,i}$ for all $x \not
\in Q \cup S $ and $i = 0,1,2$. 
Then there is a continuous semisimple representation
\[ r_f : G_{\Q} \to \GSp_4(K') \]
defined over a finite extension $K'/K$ such that:
\begin{enumerate}
\item\label{W-simil} The similitude character $\nu \circ r_f$ is given
  by
\[ \nu \circ r_f = \lambda_f \circ \chi_{\mu}\epsilon^{-3} = \lambda_f
\circ \chi_{\mu,0}\epsilon^{-\w}.\]
\item\label{W-unram} $r_f$ is unramified at primes $x \not\in Q \cup S \cup \{p\}$,
  and at such primes, the characteristic polynomial of $r_f(\Frob_x)$
  is given by:
\[ \det(X - r_f(\Frob_x)) = \lambda_f(Q_x(X)).\]
\item\label{W-at-p} The restriction $r_f | G_{\Q_p}$ is crystalline with Hodge--Tate
  weights $\w,(a-1),(b-2),0$. If, in addition, $f$ is an eigenvalue of
  the Hecke operators at $p$, then the characteristic polynomial of
  $\Phi$ on $D_{\mathrm{cris}}(r_{f}|G_{\Q_p})$ is $\lambda_f(Q_p(X))$.
\item\label{W-ord} Suppose $f$ is ordinary in the sense that it is an
  eigenform for $T_{p,1}$ and
  $Q_{p,2}$ with eigenvalues being $p$-adic units. Then $Q_p(X)$ has distinct eigenvalues
  $\alpha_p, \beta_p, \gamma_p, \delta_p$ with $p$-adic valuations $0,
  b-2, a-1, \w$, respectively. Furthermore, $r_f|G_{\Q_p}$ is
  conjugate in $\GSp_4(K')$ to a representation of the form
\[  
 \left( 
\begin{matrix} 
\lambda(\alpha_p) & * & * & * \\ 
0 & \eps^{-(b-2)} \cdot \lambda(p^{-(b-2)}\beta_p) & *  & * \\
0 & 0 &  \eps^{-(a-1)} \cdot \lambda(p^{-(a-1)}\gamma_p) &   * \\
0 & 0  & 0 &  \eps^{-\w}  \cdot \lambda(p^{-\w}\delta_p) 
\end{matrix}
\  \ \right)
\]
\item\label{W-loc-glob} If $r_f$ is absolutely irreducible, then it
  satisfies local-global compatibility at all primes.
\end{enumerate}
\end{prop}

\begin{proof} 
  The existence of $r_f$ follows from the work of Taylor, Laumon and
  Weissauer. Some of the finer properties are due to Urban,
  Genestier--Tilouine, Gan--Takeda, Sorensen and Mok. 
  Fix an embedding $\imath:K\into \C$ and let $\pi$ be an cuspidal automorphic representation of
  $\GSp_4(\A_{\Q})$ which contributes to the $f$-part of $H^0(X_1(Q),
  \omega(a,b)(-\infty)_{\C})$ under the isomorphism of the first part of
  Theorem~\ref{thm:coherent-cohom-lie-alg-cohom} (with $\sigma =
  (-b,-a;6-a-b)$, as in Remark~\ref{rem:normalization}). 
  
  We take $r_f : G_{\Q} \to \GL_4(\overline{K})$ be the representation
  $R_p$ of \cite[Theorem 3.5]{Mok} associated to $\pi$. When $\pi$ is
  simple, generic in the terminology of \cite{Mok}, the representation
  can be conjugated to take values in $\GSp_4(\overline{K})$, by the main
  theorem of \cite{bellaiche-chenevier-families}.  In the
  remaining cases, the representation $R_p$ is reducible and can
  easily be seen to be symplectic.  The usual
  Baire category argument implies that $r_f$ can be defined over a
  finite extension of $K$.  Thus in all cases, we may take $r_f
  : G_{\Q} \to \GSp_4(K')$. 
  Parts~\eqref{W-simil}-- ~\eqref{W-loc-glob} follow from the statement
  of Theorem~\cite[Theorem 3.5]{Mok}.
  \end{proof}

\begin{lemma}
\label{lem:galois-mod-m}
  Let $\mu = (a,b)\in X^*(T)_M^+$ with $a\geq b \geq 2$ and let $\m$
  be a maximal ideal of $\Tan_{\mu}(Q)$. Then there is a continuous
  semisimple representation
\[ \rbar_{\m} : G_{\Q} \to \GL_4(\Tan_{\mu}(Q)/\m) \]
such that for each $x \not\in S\cup Q \cup \{p\}$, the restriction
$\rbar_{\m}|G_{\Q_x}$ is unramified and $\rbar_{\m}(\Frob_x)$ has characteristic polynomial $Q_x(X)$.

If $\rbar_{\m}$ is absolutely irreducible, then the representation
$\rbar_{\m}$ preserves a symplectic pairing and hence, after
conjugation, we have a representation:
\[  \rbar_{\m} : G_{\Q} \to \GSp_4(\Tan_{\mu}(Q)/\m) \]
\end{lemma}

\begin{proof}
  Choose an integer $s$ as in Lemma~\ref{lem:hasse-mod-p^m} with $m$
  taken to equal $1$ and let $t = rs(p-1)$. Let $f \in
  H^0(X_1(Q),\omega(a+t,b+t)(-\infty))\otimes \overline{K}$ be an
  eigenform for $\WT_{\mu'}(Q)_{\m}$. Let $r_f$ be the Galois
  representation associated to $f$ by Proposition~\ref{prop:weiss} and take
  $\rbar_{\m}$ to be the semisimplification of a reduction of $r_f$ to
  characteristic $p$. The resulting representation is defined over the
  algebraic closure of $\Tan_{\mu}(Q)/\m$, but by the argument of
  \cite[Prop.\ 3.4.2]{CHT}, we see that after conjugation, it may be
  defined over $\Tan_{\mu}(Q)/\m$.

  For the last part: let $\pi$ the transfer to $\GL_4$ (given by
  \cite{arthur-gsp4}) of the automorphic representation generated by
  $f$. Then $\pi$ descends to an automorphic representation $\Pi$ of a
  unitary group over $\Q$. The family of $\ell$-adic Galois
  representations associated to $\Pi$ is the same as that associated
  to $f$. Thus, \cite[Theorem 1.2]{bell-chen} and the fact that
  $\rbar_{\m}$ is absolutely irreducible implies that $r_f$ is
  symplectic. The same is then true of $\rbar_{\m}$ (by absolute
  irreducibility).
\end{proof}

\begin{remark}
  By the same argument, the previous result holds if we replace
  $\Tan_{\mu}(Q)$ by $\T_{\mu}(Q)$ or $\wT_{\mu}(Q)$.
\end{remark}

\begin{df}
  \label{defn:eisenstein}
We say that $\m$ is \emph{non-Eisenstein} if the representation $\rbar_{\m}$
is absolutely irreducible.
\end{df}

\subsection{Galois representations in cohomological weights}
\label{sec:gal-rep-cohom}

Let $\rbar : G_{\Q} \to \GSp_4(k)$ be a representation as in Section
\ref{section:deformations}. 
By Assumption~\ref{assumption:neatness} and by Cebotarev,
there exist infinitely many primes~$q$ such that no pair of eigenvalues of~$\rbar(\Frob_q)$ have ratio~$q \mod p$
and~$q \not \equiv 1 \mod p$. Choose any such~$q$ which is disjoint
to~$p$ and all primes of bad reduction of~$\rbar$.
We take $S = S(\rbar) \cup \{q\}$ and $Q$ a possibly empty set of
primes disjoint from $S\cup \{ p\}$. We define a compact open subgroup
$ K = \prod_x K_x$ of $\GSp_4(\A^{\infty})$ as follows:
\begin{enumerate}
\item If $x = p$ or $\rbar$ is unramified at $x$ and~$x \ne q$, then $K_x = \GSp_4(\Z_x)$.
\item If $x$ is of type {\bf U3\rm}, then $K_x = I(x)$,
where $I(x)$ is the  Iwahori subgroup.
\item If $x$ is of type {\bf U2\rm}, then $K_x = \Pi(x)$,
where $\Pi(x)$ is the Klingen parahoric.
 \item If $x$ is of type {\bf U1\rm}, then $K_x = K(x)$, where
$K(x)$ is the  paramodular group at $x$.
\item If $x$ is of type {\bf P\rm}, then $K_x = \Pi(x)^{+}$ (and
$x-1$ is prime to $p$).
\item If $x$ is of type {\bf H\rm}, then $K_x$ is the full congruence subgroup of
level $x$.
\item If~$x = q$, then  $K_x$ is the full congruence subgroup of
level $x$.
\end{enumerate}
We then let $X = X_K$ and $X_{i}(Q) = X_{K_{i}(Q)}$ as in
Section~\ref{sec:cohom}.

Let $\mu = (a,b)\in X^*(T)_M^+$ with $a\geq b \geq 3$ be a
\emph{regular} weight and let $\m_{\emptyset}$ be a maximal ideal of
$\T_{\mu}^{\ord}$ (the ordinary Hecke algebra with $Q =
\emptyset$) with residue field $k$. Then $\mE$ pulls back to an ideal of $\Tan_\mu(Q)^{\ord}$
which in turn pushes forward to an ideal of $\T_{\mu}(Q)^{\ord}$. We
denote both of these ideals by $\mE$, in a slight abuse of
notation. The ideal $\mE \subset \Tan_\mu(Q)^{\ord}$ is maximal but
$\mE\subset \T_\mu(Q)^{\ord}$ need not be maximal -- there may be
multiple maximal ideals $\m$ of $\T_{\mu}(Q)^{\ord}$ that contain it.
We make the following assumption:

\begin{assumption}
  \label{assumption:hecke-galois-rep-regular}
  Let $\rbar$, $\mu$ and $\mE$ be as above. Then:
\begin{enumerate}
\item We have $\rbar_{\m_{\emptyset}} \cong \rbar$. In particular,
  since $\rbar$ is absolutely irreducible, $\m_{\emptyset}$ is
  non-Eisenstein.
\item\label{ass:TW-at-Q} For each $x \in Q$, $x \equiv 1 \mod p$ and $\rbar|G_x$ is a
  direct sum of four pairwise distinct characters with Frobenius
  eigenvalues $\alpha_x, \beta_x, \gamma_x, \delta_x$. We assume the
  eigenvalues have been labeled so that the plane
  $\lambda(\alpha_x)\oplus \lambda(\beta_x)$ is isotropic, and hence $\alpha_x\delta_x = \beta_x\gamma_x$.
\end{enumerate}
 \end{assumption}
We let $\m \subset \T_{\mu}(Q)^{\ord}$ be any maximal ideal which
contains $\mE$. 
The representations $\rbar_{\m}$, $\rbar_{\m_\emptyset}$ and $\rbar$
are all isomorphic.

We now turn to the prime $p$. Let $\alpha,\beta\in k^\times$ be the
elements associated to $\rbar|G_{\Q_p}$ at the beginning of
Section~\ref{section:deformations}. For $M = \CO, \CO/\varpi^m$ or
$K/\CO$, we define:
\begin{itemize}
\item $H^0(X_1(Q),\omega(a,b)(-\infty)_M)^{\beta}$ to be the subspace
  of $H^0(X_1(Q),\omega(a,b)(-\infty)_M)$ given by the image of the idempotent
  $e_{\beta} = \varinjlim_n ((T_{p,1} -
  \tilde\beta)(Q_{p,2}-\tilde\alpha\tilde\beta))^{n!}$, where
  $\tilde{\alpha}$ and $\tilde{\beta}$ are lifts of $\alpha$ and
  $\beta$ to $\CO$.
\item $\Tan_{\mu}(Q)^{\beta}$ (resp.\ $\T_\mu(Q)^\beta$,
  $\wT_\mu(Q)^\beta$) to be the image of $\Tan_{\mu}(Q)$ (resp.\ $\T_\mu(Q)$,
  $\wT_\mu(Q)$) in 
\[ \End_{\CO}(H^0(X_1(Q),\omega(a,b)(-\infty)_M)^{\beta}). \]
\end{itemize}
We also make the analogous definitions with $\alpha$ and $\beta$
swapping roles.

\begin{theorem} 
\label{theorem:highergood}
Let $\mu = (a,b)$, $\mE$ and $\m$ be as above, and suppose that
Assumption~\ref{assumption:hecke-galois-rep-regular} holds. Let $w =
a+b-6$ and $\w = w+3 = a+b-3$. Then there exists a continuous representation
$$r=r_{\mu,\m}^{\beta}: G_{\Q} \rightarrow
\GSp_4(\T_{\mu}(Q)_{\m}^\beta)$$
lifting $\rbar_{\m}=\rbar$ and such that: 
\begin{enumerate}
  \item\label{reg-simil} The similitude character $\nu\circ r$ is given by:
\[ \nu \circ r = \chi_{\mu} \epsilon^{-3} = \chi_{\mu,0} \epsilon^{-\w},\]
where~$\chi_{\mu,0}$ is a finite order character unramified at~$p$ which is trivial modulo~$\m$.
  \item\label{reg-unram} For each prime $x\not\in S\cup Q\cup
\{p\}$, $r$ is unramified at $x$ and $r(\Frob_x)$ has characteristic
polynomial $Q_x(X)$.
\item\label{reg-ord} 
There are units $d_{p,1},\dots,d_{p,4} \in \T_{\mu}(Q)_\m^\beta$ satisfying
$$Q_p(X) = (X - d_{p,1})(X - p^{b-2} d_{p,2})(X - p^{a-1} d_{p,3})(X - 
p^{\w} d_{p,4}) \in \T_{\mu}(Q)_\m^\beta[X],$$
and such that:
\begin{enumerate}
\item We have $d_{p,1} \mod \m = \beta$ and $d_{p,2} \mod \m =
  \alpha$;
\item  $r|G_{\Q_p}$ is conjugate in $\GSp_4$  to a
representation of the form:
\[ \left( 
\begin{matrix} 
\lambda(d_{p,1}) & * & * & * \\ 
0 & \eps^{-(b-2)} \cdot \lambda(d_{p,2}) & *  & * \\
0 & 0 & \eps^{-(a-1)} \cdot \lambda(d_{p,3}) &   * \\
0 & 0  & 0 & \eps^{-\w}\cdot \lambda(d_{p,4})
\end{matrix}
\  \ \right) \]
\end{enumerate}
\item\label{reg-deformation} 
After twisting by the unique square-root of~$\chi_{\mu,0}$ which is trivial modulo~$\m$,
the deformation $r$ of $\rbar$ satisfies properties
  \eqref{outside-NQ}--~\eqref{at-Q} of Definition~\ref{defn:minimal}.
\end{enumerate}
\end{theorem}

\begin{remark} \emph{We expect that, under the given assumptions, the Hecke rings in
question are torsion free. However, we avoid having to prove this by passing
to sufficiently high weight.}
\end{remark}

\begin{proof}
As in Remark~\ref{rem:hasse-hecke}, $I_{\mu,m}$ denotes the
annihilator of $H^0(X_1(Q),{\omega(a,b)(-\infty)}_{\CO/\varpi^m})$ in $\T_{\mu}(Q)$.
Since
$\T_{\mu}(Q)_{\m}=\varprojlim_m \T_{\mu}(Q)_{\m}/I_{\mu,m}$, it
suffices to construct, for each $m>0$, a representation $r_{m} : G_\Q
\to \GSp_4(\T_{\mu}(Q)_{\m}^\beta/I_{\mu,m})$ satisfying the conditions of the
theorem. We thus fix an $m>0$. Choose an integer $s>0$ as in
Lemma~\ref{lem:hasse-mod-p^m} and let $t = rs(p-1)$. By
Lemma~\ref{lem:hasse-mod-p^m} and Lemma~\ref{lem:hasse-mod-p}~\eqref{ops-at-p}, multiplication by
$A^s$ restricts to a map:
\[ H^0(X_1(Q),\omega(a,b)(-\infty)_{\CO/\varpi^m})^{\beta}_{\m} \into
H^0(X_1(Q),\omega(a+t,b+t)(-\infty)_{\CO/\varpi^m})^{\beta}_{\m}. \]
This in turns gives rise to a surjective map $\T_{\mu'}(Q)_{\m}^\beta \onto
\T_{\mu}(Q)_{m}^{\beta}/I_{\mu,m}$. Thus it suffices to prove the
result in weight $\mu' := (a',b') := (a+t,b+t)$.

Since $t \geq N(\mu)$, we have that
\[ H^0(X_1(Q),\omega(a',b')(-\infty)_{K/\CO}) \cong
H^0(X_1(Q),\omega(a',b')(-\infty))\otimes K/\CO \] 
and hence we may regard $\T_{\mu'}(Q)$ as acting faithfully on both
\[ H^0(X_1(Q),\omega(a',b')(-\infty)) \mbox{\ \ and\ \ }
H^0(X_1(Q),\omega(a',b')(-\infty)_K). \]
Thus we have
\[ \T_{\mu'}(Q)_{\m}^{\beta} \into \prod_i \CO_{K_i} \]
where the $K_i$ are a finite collection of finite extensions of $K$,
one for each minimal prime $\wp_i$ of $T_{\mu'}(Q)_{\m}^{\beta}$. Each
such minimal prime corresponds to an eigenform $f_i$ for
$\T_{\mu'}(Q)_{\m}^{\beta}$. The eigenform $f_i$ has an associated
Galois representation $r_{f_i} : G_{\Q} \to \GSp_4(\CO_{K'_i})$ for
some finite extension $K'_i/K_i$,  by
Proposition~\ref{prop:weiss}. After conjugation, we may assume that
each $r_{f_i}$ reduces to $\rbar$.
 By the argument of the proof of
\cite[3.4.4]{CHT}, using \cite[Lemma 7.1.1]{gee-ger} in place of
\cite[2.1.12]{CHT}, we see that the representation $\prod_i r_{f_i}$
descends to a representation $r : G_{\Q} \to
\GSp_4(\T_{\mu'}(Q)_{\m}^{\beta})$. It follows from
Proposition~\ref{prop:weiss} that $r$ satisfies properties
\eqref{reg-simil}--\eqref{reg-ord} of the theorem. For part \eqref{reg-ord}, note
that $Q_p(X) \in \T_{\mu'}(Q)_{\m}^{\beta}$ factors as
\[ (X-d_{p,1})(X-p^{b-2}d_{p,2})(X-p^{a-1}d_{p,1})(X-p^{\w}d_{p,4}) \]
for units $d_{p,i} \in \T_{\mu'}(Q)_{\m}^{\beta}$. We also have
$T_{p,1} \equiv \beta \mod \m$ and $Q_{p,2} \equiv \alpha\beta \mod \m$
in $\T_{\mu'}(Q)_{\m}^{\beta}$ (by definition of the idempotent
$e_\beta$). Since $Q_p(X) = X^4 - T_{p,1} X^3 +
p^{b-2}Q_{p,2}X^2 - \dots$, we deduce that $d_{p,1} \equiv \beta \mod
\m$ and $d_{p,2} \equiv \alpha \mod \m$.

To show that $r$ satisfies properties \eqref{outside-NQ}--~\eqref{at-Q} of
Definition~\ref{defn:minimal}, it suffices to show that each $r_{f_i}$
does so. In fact, property \eqref{outside-NQ} has already been established
with the exception of the prime~$x = q$.
If~$x = q$, then (by our assumptions)~$\ad^0(\rbar)(1)$ 
as a~$G_{\Q_q}$-module contains no subquotient
isomorphic to~$k$, and so~$H^2(\Q_q,\ad^0(\rbar)) \simeq H^0(\Q_q,\ad^0(\rbar)(1))^{*} = 0$.
Since~$q \ne p$, it follows that~$H^1(\Q_q,\ad^0(\rbar))$ consists entirely of unramified classes.
In particular, all lifts of~$\rbar$ are automatically unramified at~$q$.
 Since $\m$ is non-Eisenstein, it follows from
Proposition~\ref{prop:weiss}\eqref{W-loc-glob} that $r_{f_i}$
satisfies local-global compatibility at all primes. Thus we may apply the
results of \cite[\S 4.5]{Sor}.  We now turn to
property \eqref{at-special} of Definition~\ref{defn:minimal}. If $x \in S(\rbar)$ is of type {\bf U3},
then $\rbar(I_x)$ is unipotent and generated by a conjugate of
$\exp(N_3)$. Since $K_x = I(x)$, \cite[Corollary 1]{Sor} implies that
$r_{f_i}(I_x)$ is topologically generated by a conjugate of
$\exp(N_3)$, $\exp(N_2)$ or $\exp(N_1)$. The latter two cases are
incompatible with the residual representation being of nilpotent rank
3.  Similarly, if $x \in S(\rbar)$ is of type {\bf U2}, then $K_x =
\Pi(x)$ and \cite[Corollary 1]{Sor} implies that $r_{f_i}(I_x)$ is
topologically generated by a conjugate of $\exp(N_2)$ or $\exp(N_1)$.
The latter case is incompatible with the residual representation being
of nilpotent rank 2. 
Finally, if~$x \in S(\rbar)$ is of type {\bf U1}, then~$K_x = K(x)$. It then suffices to note,
following~\cite[\S 4.5]{Sor}, that the corresponding representation~$\pi_x$ is \emph{para-spherical},
that is, has a non-zero fixed vector by a non-special maximal compact subgroup, namely~$K(x)$ itself.
This establishes property
\eqref{at-special}. 
For property \eqref{at-principle}, suppose that $x \in
S(\rbar)$ is of type {\bf P}. Then $K_x = \Pi(x)^+$. It follows from
\cite[Corollary 1]{Sor} that $\Pi(x)$ has no invariants on the
automorphic representation generated by $f_i$ (as otherwise
$r_{f_i}|I_x$ would be unipotent, contradicting the assumption on
$\rbar$ at $x$). Thus $\Pi(x)/\Pi(x)^+$ acts through a
non-trivial character on the space of $\Pi(x)^+$ invariants. By
\cite[Corollary 3]{Sor} all such characters have to lift the character
$\nu \circ \rbar|I_{x}$. However, since $x -1$ is prime to $p$, there
is a unique such character, and the result follows from
\cite[Corollary 3]{Sor}.

Finally, we turn to property \eqref{at-Q} of
Definition~\ref{defn:minimal}. Let $x \in Q$, and recall that $K_x =
\Pi(x)^+$. Let $\pi$ be the automorphic representation generated by
$f_i$.  Consider first the case where $\pi_x$ has non-trivial
$\Pi(x)$-invariants. Then $\pi_x$ 
is a subquotient of an unramified principal series. By
part~\eqref{ass:TW-at-Q} of
Assumption~\ref{assumption:hecke-galois-rep-regular} and
\cite[Prop.\ 3.2.3]{TG}, we see that $\pi_x$ is unramified. In this
case, property \eqref{at-Q} of Definition~\ref{defn:minimal} certainly
holds for $r_{f_i}$. In the remaining case, where $\pi_x$ has no
non-trivial $\Pi(x)$-invariants, we see that $\Pi(x)/\Pi^+(x)$ acts
through a non-trivial character on $\pi_x^{\Pi(x)^+}$, and the
required property holds by~\cite[Corollary 3]{Sor}.
\end{proof}

\subsection{Galois representations in low weights}
\label{sec:gal-rep-low}

We let $\rbar : G_{\Q} \to \GSp_4(k)$, $S = S(\rbar)$, $Q$ and $K
\subset \GSp_4(\A^{\infty})$ be as in the previous section. Recall
that in Section~\ref{section:deformations}, we fixed two units
$\alpha,\beta \in k^\times$ associated to $\rbar|G_{\Q_p}$. We now let
$\sigma = (a,2) \in X^*(T)_M^{+}$ with $a\geq 2$ denote a non-regular
weight.

\begin{df}
  \label{df:katz-modular}
We say that $\rbar$ is \emph{Katz modular of weight $\sigma$} if there exists a maximal ideal
$\m_{\emptyset}$ of $\T_{\sigma}$ such that:
\begin{enumerate}
\item We have $\rbar_{\m_{\emptyset}} \cong \rbar$, and
\item There exists a form $\eta \in
  H^0(X,\omega(a,2)_{K/\OL})[\m_{\emptyset}]$ such that
  \begin{align*}
    T_{p,1} (\eta) &= (\alpha+\beta)\eta \\
    Q_{p,2} (\eta) &= (\alpha\beta)\eta.
  \end{align*}
\end{enumerate}
\end{df}

We now make the following assumption:

\begin{assumption}[Residual Modularity] We assume:
  \label{assumption:hecke-galois-rep-low}
\begin{enumerate}
\item\label{ass:katz-mod} $\rbar$ is Katz modular of weight $\sigma$ with associated maximal ideal
  $\m_{\emptyset}$ and eigenform $\eta$,
\item\label{ass:TW-at-Q-low} For each $x \in Q$, $x \equiv 1 \mod p$ and $\rbar|G_x$ is a
  direct sum of four pairwise distinct characters with Frobenius
  eigenvalues $\alpha_x, \beta_x, \gamma_x, \delta_x$. We assume the
  eigenvalues have been labeled so that the plane
  $\lambda(\alpha_x)\oplus \lambda(\beta_x)$ is isotropic, and hence $\alpha_x\delta_x = \beta_x\gamma_x$.
\end{enumerate}
\end{assumption}
We let $\m$ be any maximal ideal of $\T_{\sigma}(Q)$ containing
$\m_{\emptyset}$.

\label{section:ordinaryprojection}
Let $e_{\alpha,\beta}$ be the idempotent 
\[ \varinjlim_{n}((T_{p,1}-\tilde\alpha-\tilde\beta)(Q_{p,2}-\tilde\alpha\tilde\beta))^{n!},\]
where $\tilde{\alpha}$ and $\tilde{\beta}$ are lifts of $\alpha$ and
$\beta$ to $\CO$, and define:
\[ H^0(X_1(Q),\omega(a,2)(-\infty)_{K/\OL})^{\alpha,\beta} =
e_{\alpha,\beta} H^0(X_1(Q),\omega(a,2)(-\infty)_{K/\OL}).\]
The assumption that $\rbar$ is Katz modular implies that this space is
non-zero after localization at $\m$.
We let $\T_{\sigma}(Q)^{\alpha,\beta}$ denote the image of
$\T_{\sigma}(Q)$ in 
\[ \End_{\CO}(H^0(X_1(Q),\omega(a,2)(-\infty)_{K/\OL})^{\alpha,\beta}).\]
 Our main result in this section is the following.

\begin{theorem} 
  \label{theorem:localglobal} 
Let $\rbar$, $\sigma = (a,2)$  with $p-1> a$ and $\m$ be as above and suppose that Assumption
\ref{assumption:hecke-galois-rep-low} holds. In addition, suppose that:
$$(\alpha^2 - 1)(\beta^2 - 1)(\alpha - \beta)(\alpha^2 \beta^2 - 1) \neq 0.$$
Then there exists a
representation
$$r_{Q}: G_{\Q} \rightarrow \GSp_4(\T_{\sigma}(Q)_{\m}^{\alpha,\beta})$$
which is a minimal deformation of $\rbar$ outside $Q$.
\end{theorem}

\begin{proof}
As in the proof of Theorem~\ref{theorem:highergood}, it suffices to
prove the existence of an appropriate representation $r_m : G_{\Q} \to
\GSp_4(\T_{\sigma}(Q)_{\m}^{\alpha,\beta}/I_{\sigma,m})$ for each $m>0$. We thus fix
an $m>0$. By Theorem~\ref{theorem:qexp} below, there exists a power
$A^s$ of $A$ such that we have injections:
\begin{align*}
  H^0(X_1(Q),\omega(a,2)(-\infty)_{\CO/\varpi^m})^{\alpha,\beta}_{\m}
  &\stackrel{e_{\beta}\circ A^s}{\into}
  H^0(X_1(Q),\omega(a+t,b+t)(-\infty)_{\CO/\varpi^m})^{\beta}_{\m} \\
  H^0(X_1(Q),\omega(a,2)(-\infty)_{\CO/\varpi^m})^{\alpha,\beta}_{\m}
  &\stackrel{e_{\alpha}\circ A^s}{\into}
  H^0(X_1(Q),\omega(a+t,b+t)(-\infty)_{\CO/\varpi^m})^{\alpha}_{\m}
\end{align*}
where $t = rp^{m-1}(p-1)$. 
These in turns give rise to surjections:
\begin{align*}
\T_{\mu'}(Q)_{\m}^\beta &\onto
\T_{\mu}(Q)_{m}^{\alpha,\beta}/I_{\mu,m} \\ 
\T_{\mu'}(Q)_{\m}^\alpha &\onto
\T_{\mu}(Q)_{m}^{\alpha,\beta}/I_{\mu,m},  
\end{align*}
 where $\mu' = \mu + (t,t)$. The first of these surjections together
 with Theorem~\ref{theorem:highergood} implies the existence of a
 representation $r'_m$ satisfying all of the required properties,
 except for conditions~\eqref{det} and~\eqref{at-p} of
 Definition~\ref{defn:minimal}. However, we deduce from the existence
 of both surjections that the representation $r_m|G_p$ contains two
 distinct rank-1 unramified submodules (spanned by basis vectors) --
 one of which having Frobenius eigenvalue lifting $\alpha$, and the
 other having Frobenius eigenvalue lifting $\beta$. By Nakayama's
 Lemma, we deduce that $r_m'$ contains an unramified rank-2 submodule
 of the form required by condition \eqref{at-p} of
 Definition~\ref{defn:minimal}. In order to obtain a representation
 that also satisfies condition~\eqref{det} of
 Definition~\ref{defn:minimal}, we note that $\nu(r'_m) =
 \chi\epsilon^{-(a-1)}\chi_Q$ where $\chi_Q$ is a finite order
 character of $p$-power order which is unramified outside
 $Q$. Since $p$ is odd, we can find a square root of $\chi_Q$ and
 twist $r'_m$ by the inverse of this square root. The resulting
 representation $r_m$ now satisfies all required properties.
\end{proof}

\section{Properties of cohomology groups}
\label{sec:prop-cohom}

As in Section~\ref{sec:cohom}, let $S$ and $Q$ be finite sets of
primes of $\Q$ which are disjoint and do not contain $p$. We allow
the possibility that $Q=\emptyset$. We let $K$ and $K_{i}(Q)$ be open
compact subgroups of $\GSp_4(\A^{\infty})$ as in
Section~\ref{sec:cohom}, and we let $X = X_K$ and $X_i(Q) = X_{K_i(Q)}$ be the
corresponding Siegel threefolds,
The goal of this section is to prove Theorems \ref{thm:no-newforms}
and \ref{thm:balanced} below.

\subsection{Taylor--Wiles primes}
\label{sec:prop-cohom-groups}

Fix $\mu = (a,b) \in X^*(T)_M^+$ with $a\geq b \geq 2$.
Let $\mE$ be a
non-Eisenstein maximal ideal of~$\T_{\mu}$.
The ideal $\mE$ gives rise to ideals of $\Tan_{\mu}(Q)$ and
$\T_{\mu}(Q)$ which we also denote by $\mE$ (see
Section~\ref{sec:gal-rep-cohom}). We will need the following
assumption
(c.f.\ Assumptions~\ref{assumption:hecke-galois-rep-regular} and \ref{assumption:hecke-galois-rep-low}):

\begin{assumption}
  \label{assumption:tw-assump}
  For each $x \in Q$, we have $x \equiv 1 \mod p$, and $\rbar_{\mE}|G_x$ is a
  direct sum of four pairwise distinct characters with Frobenius
  eigenvalues $\alpha_x, \beta_x, \gamma_x, \delta_x$. We assume the
  eigenvalues have been labeled so that the plane
  $\lambda(\alpha_x)\oplus \lambda(\beta_x)$ is isotropic, and hence $\alpha_x\delta_x = \beta_x\gamma_x$.
\end{assumption}

For $x \in Q$, we let $\alpha_x', \beta_x', \gamma_x', \delta_x' \in \CO^\times$ be elements
lifting   $\alpha_x$, $\beta_x$, $\gamma_x, \delta_x \in
k^\times$.
The point of the above assumption is to rule out the possibility of
newforms at level $K_0(Q)$:

\begin{theorem} 
  \label{thm:no-newforms} 
Let $\mu$ and $\mE$ be as above, and suppose that Assumption~\ref{assumption:tw-assump} holds. Let $\m$ denote the ideal of $\T_\mu(Q)$ containing
$\mE$ together with the elements $x U_{x,2} - \alpha'_x\beta'_x$ and
$U_{x,1}-\alpha'_x -\beta'_x$ for each $x\in Q$.
Then $\m$ is maximal and there is an
isomorphism 
\[ \pr_Q \circ i: H^0(X,\omega(a,b)(-\infty)_{K/\CO})_{\mE} \liso
H^0(X_0(Q),\omega(a,b)(-\infty)_{K/\CO})_{\m}.\]
which is equivariant for the operators $T_{x,i}$ for each $x\not \in
S\cup Q \cup \{p\}$ as well as for the operators $T_{p,1}$ and $Q_{p,2}$.

Here $i$ is the natural inclusion and $\pr_Q$ is defined as
follows. For $x\in Q$, let $R_x$ denote the Hecke operator
\[ R_x = (xU_{x,2} - \alpha'_x\gamma'_x)(xU_{x,2} -
\beta_x'\delta_x')(xU_{x,2} - \gamma'_x \delta'_x)\in \T_{\mu}(Q) \]
and let $\pr_x$ denote the idempotent
\[ \pr_x = \lim_{n \rightarrow \infty} (R'_x)^{n!}. \]
Then the $\pr_x$'s commute with one another and $\pr_Q$ denotes their
product.
\end{theorem}

For compactness, we will make use the alternative notation
$\W_{\mu} = \omega(a,b)$, and $\W_{\mu}^{\sub} = \omega(a,b)(-\infty)$. In sufficiently high weight,
Theorem~\ref{thm:no-newforms} is due to Genestier and Tilouine:

\begin{theorem}
  \label{thm:genestier-tilouine}
  Suppose $\mu = (a,b)$ is such that $H^i(X, \W^{\sub}_{\mu,k})$ and
  $H^i(X_0(Q), \W^{\sub}_{\mu,k})$ are 0 for all $i>0$. Then the map
\[ \pr_Q \circ i : H^0(X,\W^{\sub}_{\mu,K/\CO})_{\mE} \liso
H^0(X_0(Q),\W^{\sub}_{\mu,K/\CO})_{\m} \]
is an isomorphism. An explicit inverse is given by the composition
\[ H^0(X_0(Q),\W^{\sub}_{\mu,K/\CO})_{\m} \into
H^0(X_0(Q),\W^{\sub}_{\mu,K/\CO})_{\mE}
\stackrel{d_Q^{-1}\mathrm{tr}}{\to} H^0(X,\W^{\sub}_{\mu,K/\CO})_{\mE}\]
where $d_Q = \prod_{x\in Q}[\GSp_4(\Z_x):\Pi(x)]$ (which is prime to
$p$) and $\mathrm{tr}$ is the trace map associated to $X_0(Q) \to X$.
\end{theorem}

\begin{proof}
By the assumption of cohomology vanishing, it suffices to prove both
statements with $K/\CO$ replaced by $K$. Indeed, if the map over $K$ is
surjective, then so too is the map over $K/\CO$. Furthermore, if
 $d_Q^{-1}\mathrm{tr}$ is an inverse over $K$, then the fact that its
 defined over $\CO$ implies immediately that it also gives an inverse over
 $K/\CO$. The proof of the corresponding result over $K$ follows exactly as in
the proof of \cite[Proposition 11.1.2]{TG}.
\end{proof}

Using this result and the Hasse invariant $h \in H^0(X, \omega^{p-1}_k)$, we can now establish
Theorem~\ref{thm:no-newforms} at the level of
$\varpi$-torsion. (Recall that cohomology in degree 0 over $k$ can be identified
with $\varpi$-torsion in degree 0 cohomology over $K/\CO$.) Note that, for any
weight $\mu$, the cohomology vanishing assumption of the previous
theorem holds in weight $\mu + (t,t)$ as long as $t \geq N(\mu)$
(where $N(\mu)$ is defined in Definition~\ref{df:suff-large-wt}).

\begin{lemma}
  \label{lem:cartesian-hasse} Let $\mu = (a,b)$ with  $a \geq b \geq 2$. Choose an integer
  $t$ such that $(p-1)t \geq N(\mu)$. Let $\mu' = (a',b') = (a+t(p-1),b+t(p-1))$.
Then the following diagrams are co-cartesian: 
\[
\begin{tikzpicture}
  \matrix (m) [matrix of math nodes,row sep=3em,column sep=4em,minimum width=2em]
  {
     H^0(X_0(Q),\W^{\sub}_{\mu,k})_{\m} & H^0(X_0(Q),\W^{\sub}_{\mu',k})_{\m} \\
     H^0(X,\W^{\sub}_{\mu,k})_{\mE} & H^0(X,\W^{\sub}_{\mu',k})_{\mE} \\};
  \path[-stealth]
      (m-2-1) edge node [left] {$\pr_Q\circ i$} (m-1-1)
      edge node [above] {$h^t$} (m-2-2)
    (m-1-1) edge node [above] {$h^t$} (m-1-2) 
      (m-2-2) edge node [left] {$\pr_Q\circ i$} node [right] {$\cong$} (m-1-2);
   \end{tikzpicture}
\]
\[
\begin{tikzpicture}
  \matrix (m) [matrix of math nodes,row sep=3em,column sep=4em,minimum width=2em]
  {
     H^0(X_0(Q),\W^{\sub}_{\mu,k})_{\m} & H^0(X_0(Q),\W^{\sub}_{\mu',k})_{\m} \\
     H^0(X,\W^{\sub}_{\mu,k})_{\mE} & H^0(X,\W^{\sub}_{\mu',k})_{\mE} \\};
  \path[-stealth]
      (m-1-1) edge node [left] {$d_Q^{-1}\mathrm{tr}$} (m-2-1)
      edge node [above] {$h^t$} (m-1-2)
    (m-2-1) edge node [above] {$h^t$} (m-2-2) 
      (m-1-2) edge node [left] {$d_Q^{-1}\mathrm{tr}$} node [right] {$\cong$} (m-2-2);
   \end{tikzpicture}
\]
In particular, the left hand
vertical maps are mutually inverse isomorphisms. 
\end{lemma}

\begin{proof}
  Note that the right hand vertical maps are mutually inverse isomorphisms by
  Theorem~\ref{thm:genestier-tilouine} and the choice of $t$. The diagrams
  are commutative because $h$ commutes with all Hecke operators at the
  primes in $Q$ (Lemma~\ref{lem:hasse-mod-p}). Now, let $f \in H^0(X,\W^{\sub}_{\mu',k})_{\mE}$ and
  let $F = \pr_Q(f) \in H^0(X_0(Q),\W^{\sub}_{\mu',k})_{\m}$. Note
  that $f$ can be recovered from $F$ via the formula
  $f = d_Q^{-1}\mathrm{tr}(F)$. We need to show that $f$ is divisible
  by $h^t$ if and only if $F$ is divisible by $h^t$. 
  But this follows  immediately by the commutativity of the diagrams
  above: if~$f = h^t g$, then~$F = h^t  \pr_Q(g)$,
  and if~$F = h^t g$, then~$f = h^t d_Q^{-1}\mathrm{tr}(g)$. (Note
  that since~$X_0(Q)$ and~$X$ are smooth (and in particular irreducible) over~$k$,
  multiplication by~$h$ is injective on~$H^0$.)
\end{proof}

We will need the analogous result for forms on the non-ordinary locus:
let $S$ (resp.\ $S_0(Q)$) denote the non-ordinary locus of $X_k$
(resp.\ $X_0(Q)_k$).

\begin{lemma}
  \label{lem:no-new-forms-non-ord}
   Let $\mu = (a,b)$ with  $a \geq b \geq 2$. Then the map
\[ \pr_Q \circ i : H^0(S,\W^{\sub}_{\mu,k})_{\mE} \liso
H^0(S_0(Q),\W^{\sub}_{\mu,k})_{\m} \]
is an isomorphism with inverse $d_Q^{-1}\mathrm{tr}$.
\end{lemma}

\begin{proof}
  We first show that the result is true in sufficiently high
  weight. More precisely: let $t \geq N(\mu)+1$. We let
  $\mu' = (a+(t-1)(p-1),b+(t-1)(p-1))$ and
  $\mu'' = (a+t(p-1),b+t(p-1))$. 
   We have a
  commutative diagram:
\[\begin{tikzpicture}
\tikzstyle{every node}=[font=\tiny]
  \matrix (m) [matrix of math nodes,row sep=3em,column sep=4em,minimum width=2em]
  {
   H^0(X_0(Q),\W^{\sub}_{\mu',k})_{\m} & H^0(X_0(Q),\W^{\sub}_{\mu'',k})_{\m}
   & H^0(S_0(Q),\W^{\sub}_{\mu'',k})_{\m} & 0 \\
    H^0(X,\W^{\sub}_{\mu',k})_{\mE} & H^0(X,\W^{\sub}_{\mu'',k}))_{\mE}  &
   H^0(S,\W^{\sub}_{\mu'',k})_{\mE} & 0\\ };
  \path[-stealth]
      (m-2-1) edge node [left] {$\pr_Q$} node [right] {$\cong$} (m-1-1)
              edge node [above] {$h$} (m-2-2)
    (m-1-1) edge node [above] {$h$} (m-1-2) 
    (m-2-2) edge node [left] {$\pr_Q$} node [right] {$\cong$} (m-1-2)
    (m-2-3) edge node [left] {$\pr_Q$} (m-1-3)
      (m-1-2) edge (m-1-3) 
     (m-1-3) edge  (m-1-4)
      (m-2-2) edge  (m-2-3) 
     (m-2-3) edge  (m-2-4)
      ;
   \end{tikzpicture}\]
The choice of $t$ guarantees that the rows are short exact sequences. From the previous lemma, we deduce that the right hand vertical map is
an isomorphism with inverse $d_Q^{-1}\mathrm{tr}$. 

Now we imitate the proof of the previous lemma to deduce the result in
smaller weights. For this we use the existence of the Hasse invariant 
\[ \tilde{h} \in H^0(S, \omega(p^2-1,p^2-1)_k). \]
Such a form was constructed in unpublished work of the second author with Goldring, but
is also constructed in greater generality in \cite{Boxer} and
\cite{Wushi}. In \cite[Theorem~B.2]{Boxer} (see also~\cite[Theorem~6.2.3]{Boxer}), it is shown that $\tilde{h}$ extends to the boundary,
(by the normality of the~$p$-rank 1 locus) and that
multiplication by $\tilde{h}$ is Hecke equivariant away
from $p$ (see \cite[Theorem~4.5.4(3)]{Boxer}).
 (It is also true, but not relevant here, that~$\tilde{h}$ vanishes on the
1-dimensional Ekedahl--Oort stratum of $S$ to precise order~$2$: see
the references in proof of Theorem~\ref{theorem:boxer} below for more discussion on this point.)

We choose an integer $s$
such that $ t:= s(p+1) \geq N(\mu)+1$. Let $\mu'' = \mu + s(p^2 -1,
p^2 -1) = \mu + t(p-1,p-1)$. Then we have a commutative diagram:
\[
\begin{tikzpicture}
  \matrix (m) [matrix of math nodes,row sep=3em,column sep=4em,minimum width=2em]
  {
     H^0(S_0(Q),\W^{\sub}_{\mu,k})_{\m} & H^0(S_0(Q),\W^{\sub}_{\mu'',k})_{\m} \\
     H^0(S,\W^{\sub}_{\mu,k})_{\mE} & H^0(S,\W^{\sub}_{\mu'',k})_{\mE} \\};
  \path[-stealth]
      (m-2-1) edge node [left] {$\pr_Q\circ i$} (m-1-1)
      edge node [above] {$\tilde{h}^s$} (m-2-2)
    (m-1-1) edge node [above] {$\tilde{h}^s$} (m-1-2) 
      (m-2-2) edge node [left] {$\pr_Q\circ i$} node [right] {$\cong$} (m-1-2);
   \end{tikzpicture}
\]
The right hand vertical map is an isomorphism with inverse
$d_Q^{-1}\mathrm{tr}$ by the first paragraph. The lemma now follows by
the same argument as the previous lemma.
\end{proof}

We will also need the analogous result for first degree cohomology over $k$:

\begin{lemma}
  \label{lem:no-new-forms-H1-mod-p}
Suppose $\mu = (a,b)$ where $a \geq b \geq 2$.  Then the map
\[ \pr_Q \circ i : H^1(X,\W^{\sub}_{\mu,k})_{\mE} \liso
H^1(X_0(Q),\W^{\sub}_{\mu,k})_{\m} \]
is an isomorphism with inverse $d_Q^{-1}\mathrm{tr}$.
\end{lemma}

\begin{proof}
  If $N(\mu) = 0$, then both sides of the map are zero, so we may assume
  that $N(\mu)>0$. Let $t \geq N(\mu)$, and let
  $\mu' = (a+(t-1)(p-1),b+(t-1)(p-1))$ and
  $\mu'' = (a+t(p-1),b+t(p-1))$. Consider the
  diagram with exact rows:
\[\begin{tikzpicture}
\tikzstyle{every node}=[font=\tiny]
  \matrix (m) [matrix of math nodes,row sep=3em,column sep=4em,minimum width=2em]
  {
   H^0(X_0(Q),\W^{\sub}_{\mu',k})_{\m} & H^0(X_0(Q),\W^{\sub}_{\mu'',k})_{\m}
   & H^0(S_0(Q),\W^{\sub}_{\mu'',k})_{\m} & H^1(X_0(Q),\W^{\sub}_{\mu',k})_{\m} & 0 \\
    H^0(X,\W^{\sub}_{\mu',k})_{\mE} & H^0(X,\W^{\sub}_{\mu'',k})_{\mE}  &
   H^0(S,\W^{\sub}_{\mu'',k})_{\mE} & H^1(X,\W^{\sub}_{\mu',k})_{\mE} & 0\\ };
  \path[-stealth]
      (m-2-1) edge node [left] {$\pr_Q$} node [right] {$\cong$} (m-1-1)
              edge node [above] {$h$} (m-2-2)
    (m-1-1) edge node [above] {$h$} (m-1-2) 
    (m-2-2) edge node [left] {$\pr_Q$} node [right] {$\cong$} (m-1-2)
    (m-2-3) edge node [left] {$\pr_Q$}   node [right] {$\cong$} (m-1-3)
    (m-2-4) edge node [left] {$\pr_Q$}  (m-1-4)
    (m-1-2) edge (m-1-3) 
    (m-1-3) edge  (m-1-4)
    (m-1-4) edge (m-1-5)
    (m-2-2) edge  (m-2-3) 
    (m-2-3) edge  (m-2-4)
    (m-2-4) edge (m-2-5)
    ;
   \end{tikzpicture}\]
The first three vertical maps are isomorphisms with inverse
$d_Q^{-1}\mathrm{tr}$ by the previous two lemmas.
We deduce that
the rightmost vertical map above is an isomorphism with inverse
$d_Q^{-1}\mathrm{tr}$. This proves the lemma
in weight $\mu'$. The general case then follows by a similar argument
using a reverse induction on $t$.
\end{proof}

We are finally in a position to prove Theorem~\ref{thm:no-newforms} in the
general case.

\begin{proof}[Proof of Theorem~\ref{thm:no-newforms}] 
For each $n\geq 1$, let $\CO_n := \CO/\varpi^n$. We have a commutative diagram:
\[\begin{tikzpicture}
\tikzstyle{every node}=[font=\tiny]
  \matrix (m) [matrix of math nodes,row sep=3em,column sep=4em,minimum width=2em]
  {
   H^0(X_0(Q),\W_{\mu,k}^{\sub})_{\m} & H^0(X_0(Q),\W_{\mu,\CO_n}^{\sub})_{\m}
   & H^0(X_0(Q),\W_{\mu,\CO_{n-1}}^{\sub})_{\m} &  H^1(X_0(Q),\W_{\mu,k}^{\sub})_{\m} \\
    H^0(X,\W_{\mu,k}^{\sub})_{\mE} & H^0(X,\W_{\mu,\CO_n}^{\sub})_{\mE}  &
   H^0(X,\W_{\mu,\CO_{n-1}}^{\sub})_{\mE}& H^1(X,\W_{\mu,k}^{\sub})_{\mE}\\ };
  \path[-stealth]
      (m-2-1) edge node [left] {$\pr_Q$} node [right] {$\cong$}  (m-1-1)
      (m-2-2) edge node [left] {$\pr_Q$}  (m-1-2)
      (m-2-3) edge node [left] {$\pr_Q$}  (m-1-3)
      (m-2-4) edge node [left] {$\pr_Q$} node [right] {$\cong$} (m-1-4)
      (m-1-1) edge (m-1-2)
      (m-2-1) edge (m-2-2)
     (m-1-2) edge  node [above] {$\varpi$} (m-1-3) 
     (m-1-3) edge  (m-1-4)
     (m-2-2) edge   node [above] {$\varpi$}  (m-2-3) 
     (m-2-3) edge  (m-2-4)
          ;
   \end{tikzpicture}\]
The vertical maps on the ends are isomorphisms by
Lemma~\ref{lem:cartesian-hasse} and Lemma~\ref{lem:no-new-forms-H1-mod-p}.
 By induction on $n$ and the Five Lemma we deduce that the map
\[ \pr_{Q}\circ i : H^0(X,\W_{\mu,\CO_n}^{\sub})_{\mE} \to
  H^0(X_0(Q),\W_{\mu,\CO_n}^{\sub})_{\m} \] is an isomorphism for all
$n$. This shows that the map of Theorem~\ref{thm:no-newforms} is an
isomorphism after passing to 
$\varpi^n$-torsion, for any $n$. The result follows.
\end{proof}

\subsection{The balanced property}
\label{sec:balanced-property}

In this section we assume that $\mu = (a,2)$ is a limit of discrete series
weight, where~$p > a - 2$. 
Let $\Delta$ be a quotient of
$\Delta_Q:=\prod_{x\in Q}(\Z/x)^\times$ and let
$X_{\Delta}(Q)\to X_0(Q)$ denote the corresponding sub-cover of
$X_1(Q)\to X_0(Q)$.  If $\CL$ is a vector bundle on $X_\Delta(Q)$, we
 define
\[ H_i(X_\Delta(Q),\mathcal L):=
H^i(X_\Delta(Q),(\omega^3 \otimes\CL^{{\vee}}(-\infty))_{K/\CO})^{\vee} \]
for all $i$. Note that $\omega^3(-\infty)$ is the dualizing sheaf on
$X_{\Delta}(Q)$. 

We now take $\CL = \omega(1,3-a)$, so that $\omega^3 \otimes
\CL^{\vee}(-\infty) \cong \omega(a,2)(-\infty)$.
Here we use our bound~$p > a - 2$ to
deduce that there is an equality~$(\Sym^{a-2})^{\vee} \simeq \Sym^{a-2} \otimes \det^{2-a}$
as~$\OL$-modules.
 Thus, $\T_{\mu}(Q)$ acts on
$H_0(X_{\Delta}(Q),\omega(1,3-a))$. We fix a non-Eisenstein maximal
ideal $\m$ of $\T_{\mu}(Q)$. We will need the following
assumption:

\begin{assumption}
  \label{ass:cohom-vanishing}
The space $H^2(X_{\Delta}(Q), \omega(a,2)(-\infty)_k)_{\m}$ is trivial.
\end{assumption}

 There is a
slight abuse of notation here in that $\T_{\mu}(Q)$ does not act on
$H^2(X_{\Delta}(Q), \omega(a,2)_k)$. The localization at $\m$ refers
to the localization at the corresponding maximal ideal of the
polynomial ring over $\CO$ generated by the Hecke operators.

\begin{remark}
  \label{rmk:lan-suh}
We note that if $p \geq a \geq 4$, then the assumption above holds, even
before localization at $\m$, by Theorem~\ref{thm:lan-suh}.
\end{remark}

\begin{lemma}
\label{lem:cohom-torsion-free}
 Suppose Assumption~\ref{ass:cohom-vanishing} holds.
 Then
 $H_1(X_\Delta(Q),\omega(1,3-a))_{\m}$
  is $p$-torsion free.
\end{lemma}

\begin{proof}
The claim is equivalent to the divisibility of
$H^1(X_{\Delta}(Q), \omega(a,2)(-\infty)_{K/\CO})_{\m}$. 
 Since $X_{\Delta}(Q)$
is flat over $\CO$, there is an exact sequence
$$0 \rightarrow {{\omega(a,2)(-\infty)}}_k \rightarrow 
{{\omega(a,2)(-\infty)}}_{K/\OL} \stackrel{\varpi}{\ra} {{\omega(a,2)(-\infty)}}_{K/\OL} \rightarrow 0.$$
Taking cohomology, this reduces to the claim that
$H^2(X_\Delta(Q),{{\omega(a,2)(-\infty)}}_{k})_{\m}$ vanishes. 
\end{proof}

The following lemma uses only the assumption that $\m$ is
non-Eisenstein: it holds in all weights and in all prime to $p$
levels. We just state it in the case we need:

\begin{lemma}
  \label{lem:harris-zucker} 
The map
\[ H^i(X_{0}(Q), \omega(a,2)(-\infty))_{\m} \otimes K \lra H^i(X_{0}(Q),
\omega(a,2))_{\m} \otimes K \]
is an isomorphism for all $i$. 
\end{lemma}

\begin{proof}
Let $\partial X$ denote the boundary of $X_{0}(Q)$.  It suffices to show that the boundary cohomology
\[ H^i(\partial X, \omega(a,2))_{\m} \otimes K \]
vanishes for all $i$. 
However, over $\C$ the cohomology of the
boundary is computed by the nerve spectral sequence:
\[ E_1^{r,s} = \bigoplus_{r(R) = r+1} E_1^{r,s}(R) \implies
H^{r+s}(\partial X, \omega(a,2)_{\C}) .\]
See \cite{HZIII} (3.2.4). Here $R$ is a $\Q$-parabolic of $G$ and $r(R)$ is its parabolic
rank. By~\cite[Corollary 3.2.9]{HZIII}, and freely using the notation of this paper. the space $E_1(R)^{r,s}$ is
the space of $K$-invariants in:
\[  \Ind^{G(\A^{\infty})}_{R(\A^{\infty})}\bigoplus_{i\geq 0, w \in
  W^{R,p}}I^{R}\left(
\widetilde{H}^{s-i-\ell(w)}(X(G_{h, R}), \CV_{\lambda(h,w)}) \otimes 
H^i(X(G_{\ell, R}),
\widetilde{\mathbf{V}}_{\lambda(\ell,w)})\right).\]
If $R = \Pi$ is the Klingen parabolic, then $G_{h, R} = \GSp_2 = \GL_2$ and
$G_{\ell, R} = \GL_1$. If $R$ is the Siegel parabolic or the Borel
subgroup, then $G_{h,R}$ is trivial and $G_{\ell, R} = L_R$ is the
Levi component of $R$ (and hence is either $\GL_2\times \GL_1$ or
$\GL_1^{3}$). In all cases, $\CV_{\lambda(h,w)}$ is the canonical
extension of an automorphic vector bundle on the Shimura variety
$X(G_h)$ and $\widetilde{\mathbf{V}}_{\lambda(\ell, w)}$ is a local
system on $X(G_{\ell})$ associated to an algebraic representation of
$G_{\ell}$. 
See \cite[(3.6.1)]{HZI} for the highest weight formulas. The functor $I_R$ is an intermediate induction
defined in \cite[(3.2.8)]{HZIII}.

Since each of the groups $G_h$ and $G_{\ell}$ are products of copies
of $\GL_2$ and $\GL_1$, we see that to any Hecke eigenclass in any
$H^i(\partial X, \omega(a,2)_{\C})$, we can associate a compatible
system of reducible $\GSp_4$-valued $l$-adic representations of
$G_{\Q}$. Since the ideal $\m$ is non-Eisenstein, it follows that
$H^i(\partial X, \omega(a,2))_{\m}\otimes K = 0$, as required.
\end{proof}

We come to the main result of this section:

\begin{theorem}
  \label{thm:balanced}
  Let $\Delta$ be a quotient of $\Delta_Q$ which is of $p$-power
  order. As above, let $\mu = (a,2)$ with~$p - 2 > a$ and let $\m$ be a non-Eisenstein
  ideal of $\T_{\mu}(Q)$. Suppose that
  Assumption~\ref{ass:cohom-vanishing} holds.
 Then the $\CO[\Delta]$-module
 \[ H_0(X_{\Delta}(Q), \omega(1,3-a))_{\m} = H^0(X_{\Delta}(Q), \omega(a,2)(-\infty)_{K/\CO})^{\vee}_\m \] is balanced in the sense
 of Definition \ref{defn:balanced}.
\end{theorem}

\begin{proof}
The argument proceeds exactly as in the proof of Prop.~3.8
of~\cite{CG}. If we let $M = H_0(X_{\Delta}(Q), \omega(1,3-a))_{\m}$
and $S = \CO[\Delta]$, then the defect $d_{S}(M)$ is given by:
\[ d_{S}(M) = r - \dim_k \Tor_1^S(M,\CO)/\varpi \]
where $r$ is the $\CO$-rank of $M_{\Delta}$. Thus we need to show that $r \ge
\dim_k \Tor_1^S(M,\CO)/\varpi$. 

Let $\CL = \omega(1,3-a)$. Applying Pontryagin duality to the Hochschild--Serre spectral
sequence, we get a spectral sequence:
\[ \Tor_i^S\left(H_j(X_{\Delta}(Q), \CL)_{\m}, \CO\right) \implies
H_{i+j}(X_{0}(Q),\CL)_{\m} \]
This spectral sequence tells us that:
\begin{enumerate}
\item $M_{\Delta} \iso H_0(X_0(Q), \CL)_{\m}$, and
\item we have a short exact sequence
\[ \left(H_1(X_{\Delta}(Q), \CL)_{\m}\right)_{\Delta} \lra H_1(X_0(Q),
\CL)_{\m} \lra \Tor_1^S(M, \CO) \lra 0.\]
\end{enumerate}
To prove that $d_S(M) \geq 0$, it follows from the second point that
it is sufficient to show that $H_1(X_0(Q),\CL)_{\m}$ is free of rank
at most $r$ over $\CO$. Lemma~\ref{lem:cohom-torsion-free} tells us that this space is
$p$-torsion free. Passing to characteristic 0 and using the first
point, we are therefore reduced to establishing the inequality:
$$\dim_K H_1(X_0(Q),\CL)_\m \otimes K \le \dim_K
H_0(X_0(Q),\CL)_\m\otimes K.$$
In other words, we need to show:
$$\dim_K H^1(X_0(Q),\omega(a,2)(-\infty))_\m \otimes K \le \dim_K
H^0(X_0(Q),\omega(a,2)(-\infty))_\m\otimes K.$$ 
By Lemma~\ref{lem:harris-zucker}, we are reduced to showing that
\[ \dim_K \overline{H}^1(X_0(Q),\omega(a,2))_\m \otimes K \le \dim_K
\overline{H}^0(X_0(Q),\omega(a,2))_\m\otimes K \]
where $\overline{H}^i$ denotes the interior cohomology (the image of
$H^i(\omega(a,b)(-\infty))$ in $H^i(\omega(a,b))$). 

As recalled in
Theorem~\ref{thm:coherent-cohom-lie-alg-cohom}, the interior cohomology can be
computed in terms of square integrable automorphic forms on
$G$. By Remark~\ref{rem:normalization}, the cohomology of
$\omega(a,b)$ agrees with that of $\W_{\mu}\cong \CV_{\sigma}$ where $\mu = (a,b;3-a-b)
= (a,2;1-a)$ and $\sigma = (-2,-a;4-a)$. 
Theorem~\ref{thm:coherent-cohom-lie-alg-cohom} then implies that:
\[ \overline{H}^i(X_{0}(Q), \omega(a,2)_{\C}) \subset \bigoplus_{\pi
  \in \CA_{(2)}(G)} \left(\left(\pi^{\infty}\right)^{K_0(Q)} \otimes
  H^i(\Lie P^{-}, K^h; \pi_{\infty}\otimes V_{\sigma}) \right)^{\oplus
  m_{(2)}(\pi)}\] where $m_{(2)}(\pi)$ denotes the multiplicity of $\pi$ in
$\CA_{(2)}(G)$.
Fix a degree $i \in \{0,1\}$
and let $\pi \in \CA_{(2)}(G)$ be such that $\pi$ contributes to
$H^i(X_{0}(Q),\omega(a,2))_{\m}\otimes \C$ under the above
inclusion (for some embedding $K \into \C$). Let $\widetilde{\pi}$ denote the transfer of $\pi$ to
$\GL_4(\A)$ under the Classification Theorem of \cite{arthur-gsp4}.
Then, by Remark~\ref{rem:central-chars}, 
the infinitesimal character of
$\widetilde{\pi}_{\infty}$ is $\chi_{(0,0-(a-1),-(a-1))+3/2(1,1,1,1)}$.
Let
$\chi_{\pi}$ denote the central character of $\pi$.

The representation $\widetilde{\pi}$ falls
into one of 6 classes (a)--(f) given in Section 5 of
\cite{arthur-gsp4}.
We show now that we can rule out all classes other
than class (a). In
cases (e) and (f), $\widetilde{\pi}$ is an isobaric sum of idele class
characters.
In case (d), $\widetilde{\pi}$
is of the form $\lambda|\cdot|^{1/2} \boxplus \lambda |\cdot|^{-1/2}
\boxplus \mu$ where $\lambda$ is an idele class character and $\mu$ is
a cuspidal automorphic representation of $\GL_2(\A)$ such that its
central character $\chi_{\mu}$ satisfies $\chi_{\mu} = \lambda^2 =
\chi_{\pi}$. Considering the infinitesimal character of
$\widetilde{\pi}_{\infty}$, we see that we must have $a = 2$
and $\mu$ must correspond to a classical modular eigenform of weight
2.
In case (c), there is a cuspidal
automorphic representation $\mu$ of orthogonal type of $\GL_2(\A)$
such that $\widetilde{\pi} = \mu|\cdot|^{1/2} \boxplus \mu |\cdot |^{-1/2}$. Being of orthogonal type means that $\mu$ is
induced from a quadratic extension of $\Q$. In case (b), $\widetilde{\pi} = \mu_1 \boxplus
\mu_2$ where the $\mu_i$ are distinct cuspidal automorphic
representations of $\GL_2(\A)$ with $\chi_{\mu_1} = \chi_{\mu_2} =
\chi_{\pi}$. Considering the infinitesimal character of
$\widetilde{\pi}_{\infty}$ and the fact that the $\mu_i$ have
the same central character, it follows that $\mu_i$ are both associated to classical
modular eigenforms of weight $a$. Thus, in all cases (b) -- (f), we can associate a compatible family of reducible
$l$-adic Galois representations to $\widetilde{\pi}$. This contradicts the fact
that $\m$ is non-Eisenstein.

The only remaining case is case (a) where $\widetilde{\pi}$ is a cuspidal
automorphic representation of $\GL_4(\A)$ that is $\chi_{\pi}$-self
dual.
By Clozel's Purity Lemma \cite[Lemme 4.9]{clozel-ann-arb},
$\widetilde{\pi}_{\infty}$ is essentially tempered. (We thank
  Olivier Ta\"{\i}bi for pointing this out to us.) It follows that
$\pi_{\infty}$ is also essentially tempered, since its $L$-parameter is
essentially bounded.
 Then by
 Theorem~\ref{thm:contrib-to-char-0-cohom}\eqref{Mirkovic}, $\pi_{\infty}$ is the limit of
 discrete series representation $\pi(\lambda,C_i)$ where $\lambda = (a-1,0;4-a)$.
Furthermore, by a Theorem of Wallach \cite[Theorem 2.3]{Mok}, it
follows that $\pi$ is cuspidal. 

By the first part of Theorem~\ref{thm:coherent-cohom-lie-alg-cohom}
, the cuspidal
cohomology $\CH^i_{\cusp,\sigma}$ maps injectively to the interior
cohomology:
\[ \overline{H}^i(X_{0}(Q), \omega(a,2)_{\C})_{\cusp} \cong \bigoplus_{\pi
  \in \CA_{0}(G)} \left(\left(\pi^{\infty}\right)^{K_0(Q)} \otimes
  H^i(\Lie P^{-}, K^h; \pi_{\infty}\otimes V_{\sigma}) \right)^{\oplus
  m_0(\pi)}\]
where $m_0(\pi)$ is the multiplicity of $\pi$ in $\CA_0(G)$.
 Thus, at this point, we can prove that the dimensions
\[ \dim_K \overline{H}^j(X_0(Q),\omega(a,2))_\m \otimes K \]
are equal for $j=0,1$ if we can establish:
\begin{enumerate}
\item The spaces $H^j(\Lie P^{-}, K^h; \pi(\lambda, C_j)\otimes
  V_{\sigma})$ have the same dimension for $j = 0, 1$.
\item The representation $\pi' = \pi^{\infty}\otimes \pi(\lambda,
  C_{1-i})$ also lies in $\CA_{(2)}(G)$;
\item The multiplicities $m_0(\pi)$, $m_{(2)}(\pi)$, $m_0(\pi')$ and
  $m_{(2)}(\pi')$ are all equal.
\end{enumerate}
The first point follows from \cite[Theorem 3.4]{harris-ann-arb} which
says that both spaces are one dimensional. The second point follows from 
\cite{arthur-gsp4}. Indeed, since $\pi(\lambda, C_i)$ is essentially tempered, the
local packet $\Pi_{\psi_{\infty}}$  (where $\psi = \widetilde{\pi}
\boxplus 1$, in the notation of
\cite{arthur-gsp4}) is in fact an L-packet by \cite[Theorem
2.1]{Mok}. 
Furthermore, it consists of the pair of representations $\{
\pi(\lambda, C_0), \pi(\lambda, C_1) \}$ (see \cite[\S 3.1]{Mok}).
Since the group $\mathcal{S}_{\psi}$ is
trivial in Case (a) of \cite{arthur-gsp4}, it then follows from part (ii)
of the Classification Theorem that $\pi'$ is also
automorphic. Finally, for the third point, the theorem of Wallach
quoted above implies that $\pi$ and $\pi'$ are both cuspidal. Part (iii) of
the Classification Theorem then implies that each of the
multiplicities in point (3) is 1.
We have thus shown that 
$$\dim_K H^1(X_0(Q),\omega(a,2)(-\infty))_\m \otimes K = \dim_K
H^0(X_0(Q),\omega(a,2)(-\infty))_\m\otimes K,$$
as required. 
\end{proof}

\section{\texorpdfstring{$q$}{q}-expansions of Siegel modular forms}
\label{section:qstuff}

As in Section~\ref{sec:cohom}, let $S$ and $Q$ be finite sets of
primes of $\Q$ which are disjoint and do not contain $p$. We allow
the possibility that $Q=\emptyset$. We let $K$ and $K_{i}(Q)$ be open
compact subgroups of $\GSp_4(\A^{\infty})$ as in
Section~\ref{sec:cohom}, and we let $X = X_K$ and $X_i(Q) = X_{K_i(Q)}$ be the
corresponding Siegel threefolds, with open subspaces $Y$ and $Y_i(Q)$,
all defined over $\CO$.

\subsection{\texorpdfstring{$q$}{q}-expansions of Siegel modular forms}
\label{section:qexpone}

(For more background and details on the results quoted in this section,
see~\S~3.1 of~\cite{TilouineDocumenta}.)
Recall that $Y_1(Q)$ has good reduction at $p$.
Let $R$ be an $\OL$-module (we will exclusively be interested in the case when
either $R = \OL/\varpi^n$ for some $n$, or when $R = K/\OL$).
Let $\sigma=(j,k) \in X^*(T)_M^+$ be a weight and associate to
$\sigma$ the representation
\[ U = \left(\Sym^{j-k}(\CO^2) \otimes_{\CO} \det(\CO^2)^{\otimes k}\right) \otimes_{\CO}R \]
of $\GL_2$ over $R$. Associated to $\sigma$, we also have the vector
bundle $\W_{\sigma} = \omega(j,k)$.
There is a $q$-expansion map:
$$H^0(Y_1(Q),\omega(j,k)_{R}) \rightarrow R[[q,q',\zeta]][\zeta^{-1}] \otimes_R U.$$

\begin{theorem}
The $q$-expansion map is injective.
\end{theorem}

\begin{proof} This is a standard fact (see, for example, Prop.~3.2 of~\cite{TilouineDocumenta}).
\end{proof}

\subsection{Explicit Formulae}
\label{section:explicit}

Let $\MMM$ be the product of the primes in $S$ and $Q$, so that~$X_1(Q)$
has good reduction outside~$\MMM$. Let~$R$ be a~$\Z_p$-module and thus
a~$\Z[1/\MMM]$-algebra. Any  $F \in H^0(X_1(Q),\W_{\sigma,R})$ has a ``$q$-expansion'':
$$F =  \sum_{\XX}  a(F,Q) q^Q,$$
where $\XX$ denotes the $2 \times 2$ positive semi-definite matrices which
take on $\Z[1/\MMM]$-integral arguments for integral vectors, or equivalently,
$$\XX = \left( \begin{matrix} m & \frac{1}{2} r \\ \frac{1}{2}r & n \end{matrix} \right), m,n,r \in \Z[1/\MMM].$$
The set~$\XX$ is naturally a subset of~$M_2(\Q)$. The group~$\GL_2(\Q)$
acts on~$M_2(\Q)$ by the following formula:
\begin{equation*}
M.Q := (\det M)^{-1} M Q M^{T}
\end{equation*}
where the right hand side is multiplication. We may naturally extend the definition of~$a(F,Q)$
for~$Q \in M_2(\Q)$ by setting~$a(F,Q) = 0$ for all~$Q$ not in~$\XX$.
In any~$q$-expansion, the coefficients~$a(F,Q)$ will also vanish unless the denominators occurring in~$Q$
are bounded by some fixed power of~$\MMM$ which depends only on the level structure.  (Since our arguments in this section
are all~$p$-adic, there is little harm in imagining that~$\MMM = 1$.)
Let~$V = \OL^2$ be the standard representation of~$\SL_2(\Z)$ over $\CO$.
The elements $a(F,Q)$ are elements of the representation $U$, where,
if $\sigma$ has weight $(j,k)$, then
$$U = \Sym^{j-k}(V) \otimes R.$$
Let $\rho: \SL_2(\Z) \rightarrow \GL(U)$ denote the corresponding representation.
The representation~$\rho$ extends to a homomorphism from~$M_2(\Z)$ to~$\End(U)$ over~$R$ which we denote by~$\rho$,
where once more~$\rho$ only depends on~$j-k$ and (more relevantly) preserves integrality.
 We may write the~$q$-expansion of a form~$F$ as
 $$F  =  \sum_{\substack{n, m \ge 0  \\ r^2 - 4mn \le 0 }} a_F(n,r,m) q^n \zeta^r {q'}^{m}$$ 
where $a_F(n,r,m) = a(F,Q)$ satisfies, for~$M \in \Gammabar \subset \SL_2(\Z)$, the equality
$$a(F,M.Q) = \rho(M) a(F,Q).$$
Here~$\Gammabar$ is the congruence subgroup of~$\SL_2(\Z)$ defined on p.807 of~\cite{TilouineDocumenta}; since we are working
at spherical  level at~$p$ the group~$\Gammabar$ has level prime to~$p$.  (It will do the reader little harm to pretend
 that~$\Gammabar$ is just~$\SL_2(\Z)$.)

\begin{remark} \label{remark:parity}
\emph{We shall assume that either~$j \ge 4$ or~$j = k = 2$. Since we are most interested in representations with  similitude
character~$\nu$ is equal to~$\epsilon^{j+k-3}$,  the oddness condition forces the congruence~$j \equiv k \mod 2$,
and so if~$j > k \ge 2$ then~$j \ge 4$. In cases (coming from Taylor--Wiles primes) where there is non-trivial Nebentypus
character at the auxiliary primes~$q|Q$, we may twist (at the cost of increasing the level at~$Q$) to force the Nebentypus
character to be trivial. The only change this has is to make the~$q$-expansions below less unpleasant --- the addition
of a Nebentypus character only introduces a notational difficulty. We note, however, that with non-trivial Nebentypus
character the case of weight~$(j,k) = (3,2)$ is possible, but our arguments would not cover this case.}
\end{remark}

\subsection{Hecke Operators at~\texorpdfstring{$p$}{p}}
Since we will exclusively be interested in Hecke operators at~$p$,
we drop the subscript~$p$ from the notation. Similarly, we drop the subscript~$1$,and so~$T_{p,1}$ and~$U_{p,1}$ are denoted~$T$ and~$U$,
whereas~$T_{p,2}$ and~$U_{p,2}$ are denoted~$T_2$ and~$U_2$ respectively.
One has the following explicit description of the Hecke operator~$T$:
\begin{lemma} \label{eightthree} \label{lemma:T} In weight~$\sigma = (j,k)$  
there is an identity of formal operators
$T = U +  p^{k-2} Z + p^{k+j-3} V$, where~$U$, $Z$, and~$V$ preserve
formal integral~$q$-expansions, and such that the following identities hold:
$$a(UF,Q) = a(F,pQ),$$
$$a(ZF,Q) = \sum_{\Se} \rho(M) a(F,M^{-1}.Q).$$
Here $\Se$ denotes (any) set of representatives in~$M_2(\Z)$ for the left coset 
decomposition of
$$\Gammabar \left( \begin{matrix} p & 0 \\ 0 & 1 \end{matrix} \right) \Gammabar.$$
Moreover, $a(F,S^{-1} Q) = 0$ unless $S^{-1} Q$ is a $p$-integral binary quadratic form.
\end{lemma}

Note that the coset decomposition of~$\Gammabar \left( \begin{matrix} p & 0 \\ 0 & 1 \end{matrix} \right) \Gammabar$
for a congruence subgroup~$\Gammabar$ prime to~$p$
is essentially the same as the coset decomposition of
$\SL_2(\Z) \left( \begin{matrix} p & 0 \\ 0 & 1 \end{matrix} \right) \SL_2(\Z)$.
These formulae are well known. See, for example, Prop~10.2 of~\cite{geer}. To compare our formula with \emph{ibid}, note that
we have normalized the matrices in~$\Se$ to be integral of determinant~$p$, and absorbed the action of the determinant into the coefficient (since we are concerned here
with issues of~$p$-integrality). 
We have a similar description of~$T_2$ which can be obtained by a laborious computation (following the arguments
of~\S3.2 and~\S3.3 of~\cite{Andrianov}:

\begin{lemma} \label{eightfour} In weight~$\sigma = (j,k)$  
there is an identity of formal operators $T_2 = p^{k+j-6} U_2 + p^{k-3} Z_2 + p^{2k+j-6} V_2,$ where~$U_2$, $Z_2$, and $V_2$ 
 preserve
formal integral~$q$-expansions, and the following identities hold:
$$a(Z_2 F,Q) = \sum_{\Se} \rho(M) a(F, M^{-1}.pQ).$$
where~$\Se$ is as in the description of~$Z$ in Lemma~\ref{lemma:T}.
If~$Q \not\equiv 0 \mod p$, then
$$a(U_2 F,Q) = \left(-1 +  p \left(\frac{\det(Q)}{p} \right) \right) a(Q) = \left( -1 + p \left( \frac{r^2 - 4 m n}{p} \right)  \right) a(Q).$$
If~$Q \equiv 0 \mod p$, then
$$a(U_2 F,Q) = (-1 + p^2 ) a(Q).$$
\end{lemma}

For those wanting a more explicit description, note that in weight~$(k,k)$ we have the 
possibly more familiar identities:
$$a(ZF,n,r,m) =
a(pn,r,m/p) + \sum_{0 \le  \alpha < p} a((n + \alpha r + \alpha^2 m)/p, r + 2 m \alpha, p m),$$
$$a(Z_2 F,n,r,m) = a(p^2 n,p r,m) +  \sum_{0  \le \alpha < p} a((n + \alpha r + \alpha^2 m), p(r  + 2 m \alpha), p^2 m).$$
Note also that there is a formal identity~$Z_2 = U Z$.

\begin{df} Let~$X_2$ denote the formal operator on~$q$-expansions such that
$$U_2 = -1 + p \cdot X_2.$$
Explicitly, if~$Q \not\equiv 0 \mod p$, then~$a(X_2 F,Q) = a(F,Q)$  times~$(D/p)$, where~$D$ is the determinant of the quadratic form associated to~$Q$, and~$(D/p)$ is the Legendre symbol. If~$Q \equiv 0 \mod p$, then~$a(X_2 F,Q) = p a(F,Q)$.
In all cases, we see that~$a(X_2 F,Q) = (D/p) a(F,Q) \mod p$.
\end{df}

\begin{lemma}\label{lemma:iszero}  Over~$k = \OL/\varpi$, we have~$Z_2 X_2 = 0$.
\end{lemma}

\begin{proof}  We have~$a(X_2 F,Q) = 0$ if~$\det(Q) \equiv 0 \mod p$, but~$a(Z_2 F,Q)$ is a sum 
over terms of the form~$a(F,R)$ with~$\det(R) = 0$.
\end{proof}

\begin{df}
\emph{A binary quadratic form $Q$  is \emph{$p$-primitive} if it is not of the form
$pR$ for an $p$-integral form $R$.}
\end{df}

\subsection{Hecke Operators on forms of  in characteristic~\texorpdfstring{$p$}{p}}

\label{section:hecke} 
Let~$Q_2 = (p \cdot T_2 + (p + p^3) S) p^{2-k}$. 

\begin{lemma} \label{lemma:Gross} There is an action of $T$  and~$Q_2$ on $H^0(X_1(Q),{\omega(j,k)}_{K/\OL})$ which
commutes with the other Hecke operators and acts on~$q$-expansions via the above formula.
\end{lemma}

\begin{proof} The argument is very similar to Prop.~4.1 of~\cite{Gross}.
It suffices to prove the result with coefficients in~$\OL/\varpi^m$.
The natural approach to defining these operators is
using correspondences, as for modular curves. There are two issues which arise. The first is
that the projection maps from the Siegel modular varieties with appropriate parahoric level
structures are not finite over~$X$. The second is that the definition involving correspondences is 
some power of~$p$ times the actual Hecke operator of interest. 
A general approach to resolving these questions has been recently found
by Pilloni~\cite{Pilloni}, who constructs all the operators used in this paper.
More importantly, his method also allows one to give an action
of these operators on higher  higher coherent cohomology as well. We use
a more pedestrian approach.
We can resolve the normalization issue by using the~$q$-expansion principle. 
The first issue is more subtle. The geometric maps involved are certainly
proper; the failure of finiteness is thus a failure of quasi-finiteness. The source
of quasi-finiteness arises from the fact that the kernel of Frobenius of an abelian surface~$A$
could (for example) equal~$\alpha_p \times \alpha_p$, which contains ``too many'' subgroup schemes of type~$\alpha_p$.
 On the other hand, this issue does not arise over the ordinary
locus nor over the larger almost ordinary locus consisting of abelian surfaces 
(those  with~$p$ rank~$\ge 1$)
where  subgroup schemes such as~$\alpha_p \times \alpha_p$ cannot occur.
This shows how to resolve the issue by
the following ad hoc method: by Hartogs'
Lemma, it suffices to construct~$T$ over the global sections of a subvariety~$X' \subset X$ whose
complement has codimension~$\ge 2$. In particular, we may replace~$X$ by the moduli space of
almost ordinary abelian surfaces for which the corresponding
maps are indeed finite. 
Implicit in this argument is a verification that the formulas above (in Lemmas~\ref{eightthree} and~\ref{eightfour})
preserve integrality --- for~$Q_2$ this is verified in Lemma~\ref{lemma:verify} below.
	\end{proof}	

Note that this argument is not sufficient to construct these operators on
$$H^1(X_1(Q),\omega(j,k)_{K/\OL}),$$
however, we have no need to the consider the action of Hecke operators at~$p$ on these spaces.

We shall also need to use various properties of theta operators.  We begin by recalling their basic properties:

\begin{prop} Let~$p > 3$, let~$j-2 \ge k \ge 2$, and let~$p-2 > j-k$.
\begin{enumerate}
\item
There is a map
$$\Theta: H^0(X_1(Q),\omega(k,k)_{\OL/\varpi^m}) \rightarrow H^0(X_1(Q),\omega(k+p+1,k+p+1)_{\OL/\varpi^m})$$
whose action on~$q$-expansions is given by
$$\Theta \sum a_Q q^Q = \sum \det(Q) a_Q q^Q.$$
\item
There is a map
$$\theta_1: H^0(X_1(Q),\omega(j,k)_{\OL/\varpi^m}) \rightarrow H^0(X_1(Q),\omega(j+p-1,k+p+1)_{\OL/\varpi^m})$$
whose action on~$q$-expansions is given by
$$\theta_1 \sum a_Q q^Q = \sum \det(Q) \con(a_Q \otimes Q) q^Q,$$
where~$\con: \Sym^{j - k} \otimes \Sym^2 \rightarrow \Sym^{j-k-2}$ is the natural ~$\SL_2(\Z)$-equivariant projection.
\end{enumerate}
\end{prop}

\begin{proof} The operator~$\Theta$ is defined in~\cite[Prop~3.9]{Yamauchi}, and the operator~$\theta_1$
is defined in~\cite[Prop~3.12]{Yamauchi}.
\end{proof}

(Some of these maps were also considered in previous unpublished work of Ghitza~\cite{Ghitza}).
The main results we need concerning these operators are given by the next two theorems.

\begin{theorem} \label{theorem:boxer} Let~$p > 3$ and~$p+1 \ge k$, and assume~$p \nmid k(2k-1)$ ---
so in particular $k = 2$ and~$k = p + 1$ are admissible values of~$k$. Then  the map
$$\Theta: H^0(X_1(Q),\omega(k,k)_{\OL/\varpi^m}) \rightarrow H^0(X_1(Q),\omega(k+p+1,k+p+1)_{\OL/\varpi^m})$$
is injective. In particular, if~$\Theta F = 0$,
 we must have~$F = 0$.
\end{theorem}

\begin{proof} We may immediately reduce to the case~$m = 1$ and~$\OL/\varpi = k$.
Suppose that~$F$ lies in the kernel, so~$\Theta F = 0$. After possibly replacing~$(k,k)$ by~$(k-(p-1),k-(p-1))$, we may assume that~$F$ is not divisible
by the Hasse invariant. Following Theorem~4.7 of~\cite{Yamauchi}, it suffices to show that~$F$ is not zero on the superspecial
locus if it is not divisible by the Hasse invariant. Hence~$F$ has
 non-trivial specialization to the~$p$-rank~$1$ strata. The supersingular locus on this strata is a Cartier divisor cut out by a section of~$\omega^{(p^2-1)/2}$ for~$p > 2$,
so since~$2k < p^2 - 1$ (for~$p > 3$), the restriction of~$F$ is non-zero on the supersingular locus. 
(That the supersingular locus
is a Cartier divisor inside the~$p$-rank~$1$ locus when~$p > 2$
was proved
by Koblitz, see~p.193 of~\cite{Koblitz}. The exact
order of vanishing can also be found in~\cite{vanderGeer},
Theorem~2.4.)
Finally, each irreducible component of the supersingular locus
is a copy of~$\PP^1$ with~$p^2 + 1$ superspecial points on it. Moreover, the line bundle~$\omega$ restricts to~$\OL(p-1)$ on each
of these~$\PP^1$s. Hence the restriction to the superspecial points is injective as long has~$k(p-1) \le p^2 + 1$, which holds for~$k \le p+1$.
\end{proof}

We also require a related result  for non parallel weight.

\begin{theorem} \label{theorem:theta} Let~$p -1 > j \ge 4$.
The map:
$$\theta_1: H^0(X_1(Q),\omega(j,2)_{\OL/\varpi^m}) \rightarrow H^0(X_1(Q),\omega (j + p-1,p + 3)_{\OL/\varpi^m})$$
is injective.
\end{theorem}

\begin{proof} It suffices to work over~$k = \OL/\varpi.$
Suppose that~$\theta_1 F =  0$,
and that~$F$ is non-zero after restriction to the
 superspecial locus. Then the result follows directly from
Theorem~3.20 of~\cite{Yamauchi}. As stated, the result does not apply in weight~$(6,2)$, although
the    same argument works in this weight providing that one may assume (in the notation of \emph{ibid}.) that~$F_2|_X \ne 0$, which
 can be achieved under the action of~$\Gammabar \subset \SL_2(\Z)$ for~$j < p-1$, since the level of~$\Gammabar$ is prime to~$p$
 and so surjects on to~$\SL_2(\F_p)$. The corresponding representation of~$\SL_2(\F_p)$ is irreducible,
 and thus for there exists an element which applied to~$F$ has~$F_i|_{X} \ne 0$ for any fixed choice of~$i$.
Hence it remains to show that the restriction of~$F$ to the superspecial  locus is non-zero.
Let~$X=X_1(Q)$, and denote the rank one strata (respectively, the supersingular locus, respectively, the superspecial locus) by~$Y$, $Z$, and~$S$ respectively.
We are assuming that the restriction of~$F$ to~$Y$ is nonzero. Suppose the restriction of~$F$ to~$Z$ is zero. There is an exact sequence:
$$0 \rightarrow H^0(Y,\omega(j,2)_k \otimes \omega^{-m}) \rightarrow H^0(Y,\omega(j,2)_k) \rightarrow H^0(Z,\omega(j,2)_k),$$
where~$m = (p^2 - 1)/2$. 
If~$F$ restricts to zero, we obtain a non-zero class in the first group. Yet there is also a sequence:
$$H^0(X,\omega(j,2)_k \otimes \omega^{-m}) \rightarrow 
H^0(Y,\omega(j,2)_k \otimes \omega^{-m}) \rightarrow 
H^1(X,\omega(j,2)_k \otimes  \omega^{-m-(p-1)}).$$
The first term vanishes. To see that the final term vanishes, we use the fact that Serre duality shows that the last term is dual
to
$$H^2(X,\omega(m + p,m + p + 2 - j)_k(-\infty)),$$
which vanishes by Theorem~\ref{thm:lan-suh}. We now have to establish non-vanishing from~$Z$ to~$S$.
The restriction of the Hodge bundle to any~$\PP^1$ on~$Z$ is~$\OL(-1) \oplus \OL(p)$.
Hence we need to show that no class in
$$H^0(\PP^1,\Sym^{j-2} (\OL(-1) \oplus \OL(p)) \otimes \OL(2(p-1)))$$
can vanish at~$p^2+1$ points. This is valid as long as
$$jp - 2 = (j-2)p + 2(p-1) \le p^2 + 1,$$
which holds provided~$j \le p$.

\end{proof}

\subsection{Relationship between Hecke eigenvalues and crystalline Frobenius}
\label{relationship}
Suppose that~$F$ is a cuspidal eigenform of weight~$\sigma = (j,k)$ of level prime to~$p$, and let~$r: G_{\Q} \rightarrow \GSp_4(\Qbar_p)$ be
the associated Galois representation. One expects (and knows in regular weights, see Theorem~\ref{theorem:highergood}) that~$r$ is crystalline at~$p$ and that crystalline Frobenius has
eigenvalues which are the roots of the following polynomial:
$$X^4 - \lambda X^3 + (p \mu + (p^3 + p) p^{k+j- 6}) X^2 -  \lambda p^{k+j-3}X + p^{2k+2j - 6},$$
where~$\lambda$ is the eigenvalue of~$T$ and~$\mu$ is the eigenvalue of~$T_2$.
We may write the eigenvalues of this polynomial  as follows:
$$\alpha ,  \beta p^{k-2} ,  \beta^{-1} p^{j-1} ,   \alpha^{-1} p^{k+j-3}, $$
where~$\alpha$ and~$\beta$ have non-negative~$p$-adic valuation.
That means that the coefficient of crystalline Frobenius should have characteristic polynomial:
$$X^4 - (\alpha + \ldots)  + (\alpha \beta p^{k-2} + O(p^{k-1})) X^2 + \ldots $$
On the other hand, we know that the coefficient of~$X^2$ should be:
$$p^{k-2} Q_2:=p \cdot T_2 + (p + p^3) S,$$
where the operator~$Q_2$ is defined by this formula. In particular, the eigenvalues of this operator ($Q_2$) should all be integral.

\begin{lemma} \label{lemma:verify} Let~$\sigma = (j,k)$ with~$j \ge k \ge 2$.
If~$(j,k) \ne (2,2)$,
there is a congruence of operators on formal~$q$-expansions:
$$Q_2 = (p \cdot T_2 + (p  + p^3) S) p^{2-k} \equiv Z_2 \mod p.$$
In particular, if~$F$ is an ordinary form of regular weight~$\sigma$ with crystalline eigenvalues as above, the eigenvalue of~$Z_2$ 
is~$\alpha \beta \mod p$. 
If~$\sigma = (2,2)$, there is a congruence
$$Q_2 = (p \cdot T_2 + (p  + p^3) S) p^{2-k} \equiv Z_2   + X_2 \mod p.$$
\end{lemma}

\begin{proof}
The operator~$S$ acts by a scalar which is equal to~$p^{j+k-6}$. 
Note that
$$p^3 \cdot p^{j + k- 6}  \cdot p^{2-k}  \equiv 0 \mod p.$$
 Thus we can ignore the~$p^3 S$ term above.
 We have
$$
\begin{aligned}
(p \cdot T_2 + (p  + p^3) S) p^{2-k}  = &  p^{3-k}( p^{j+k-6} U_2 + p^{k-3} Z_2 + p^{j + 2k-6} V_2) + p^{j-3}    \mod p  \\
 = & \ p^{j-3} U_2 + Z_2 + p^{j + k-3} V_2 + p^{j-3}   \mod p  \\
= &  \ -p^{j-3} +   p^{j-2} X_2 + Z_2 + p^{j+k-3}  V_2 + p^{j-3}  \mod p  \\
= & \ p^{j-2} X_2 + Z_2 \mod p \end{aligned}$$
and we are done.
\end{proof}

\subsection{The Main Theorem on~\texorpdfstring{$q$}{q}-expansions}
Our main theorem is as follows (we use the notation of~\S\ref{sec:gal-rep-low}).

\begin{theorem} \label{theorem:qexp} Let~$\sigma = (j,2)$ for some~$p-1 > j \ge 2$.
Assume that~$\rbar$ is as in Assumption~\ref{assumption:hecke-galois-rep-low}.
Assume, moreover, that
$$\alpha \beta (\alpha^2 - 1)(\beta^2 - 1)(\alpha - \beta)(\alpha^2 \beta^2 - 1) \neq 0.$$
Let~$\m$ denote the corresponding ideal of the Hecke algebra away from~$p$.
Let~$A$ denotes a non-trivial power of the Hasse invariant of weight~$k$.
Then the composite map:
$$\begin{diagram}
H^0(X_1(Q),\omega(j,2)_{\OL/\varpi^m})^{\alpha,\beta}_{\m} & 
\rTo^{A} &   H^0(X_1(Q),\omega(j + k,2 + k)_{\OL/\varpi^m})_{\m} \\
& & \dTo^{\pi_{\beta}} \\
& & H^0(X_1(Q),\omega(j + k,2+k)_{\OL/\varpi^m})^{\beta}_{\m}, 
\end{diagram}
$$
is injective, where~$\pi_{\beta}$ denotes the projection onto the summand where~$U - \beta$ and~$Q_2 - \alpha \beta$ 
(equivalently $Z_2 - \alpha \beta$) are nilpotent.
\end{theorem}

Note that, by symmetry, the same result holds with~$\beta$ replaced by~$\alpha$.
Before beginning the proof of this theorem, we first prove a much easier analogue for~$\GL(2)$:

\begin{lemma} \label{lemma:easy} Let~$X_1(N)$ denote the modular curve,
and let~$\rhobar: G_{\Q} \rightarrow \GL_2(\Fbar_p)$ be a modular representation of level~$N$ and weight one over~$\F_p$
such that~$\rhobar(\Frob_p)$ has eigenvalues~$\alpha$ and~$\beta$. Let~$\m$ denote the corresponding ideal of the Hecke algebra
away from~$p$.
Assume that
$$\alpha - \beta \ne 0.$$
If~$A$ denotes a suitable power of the Hasse invariant of weight~$k$, then the composite map:
$$\begin{diagram}
H^0(X_1(N),\omega_{\OL/\varpi^m})_{\m} & 
\rTo^{A} &  H^0(X_1(N),\omega^{k+1}_{\OL/\varpi^m})_{\m}  \\
& & \dTo^{\pi_{\beta}} \\
& & H^0(X_1(N),\omega^{k+1}_{\OL/\varpi^m})^{\beta}_{\m}, 
\end{diagram}
$$
is injective, where~$\pi_{\beta}$ denotes the projection onto the quotient of homology where~$U - \beta$ is nilpotent.
\end{lemma}

In both results, all of the corresponding maps are equivariant with respect to Hecke operators away from~$p$. It suffices to show that the image of the~$\T$-socle maps injectively,
and hence we may work with coefficients over a finite field~$k  = \OL/\varpi$ of characteristic~$p$.

\begin{proof}[Proof of Lemma~\ref{lemma:easy}]
Let~$M = H^0(X_1(N),\omega_{\OL/\varpi^m})_{\m}$ and~$N = H^0(X_1(N),\omega^{k+1}_{\OL/\varpi^m})_{\m}$. The map ~$M \rightarrow N$ is certainly injective, as can be seen by the~$q$-expansion principle (the map is the identity on~$q$-expansions). Let~$U$ denote the action of~$T$ on~$N$. Then~$U$ satisfies the polynomial~$U^2 - T U + \langle p \rangle = 0$ on the image of~$M$, and so~$M$ lies inside the ordinary subspace of~$N$,
and so inside~$N_{\alpha} \oplus N_{\beta}$, where~$N_{\gamma}$ is the factor of~$N$ on which~$(U - \gamma)$ is nilpotent.
We have operators~$U$ and~$V$ defined by the formulae
$$U \left( \sum a_n q^n \right) = \sum a_{np} q^n, \qquad V \left( \sum a_n q^n\right) =  \sum a_n q^{np},$$
and~$T = U + \langle p \rangle V$ in weight~$1$, whereas~$T = U$ in higher weight.
The projection operator:
$$\pi_{\beta}: N_{\alpha} \oplus N_{\beta} \rightarrow N_{\beta}$$
is given by~$\pi_{\beta} = (U - \alpha)^m$ for some integer~$m$. Suppose that~$F \in M$ satisfies~$\pi_{\beta}(F) = 0$. 
We have the identity~$U V F = F$, and we may reduce to the case that~$\langle p \rangle F = \alpha \beta F$. 
We are assuming that~$F  = F_{\alpha} \in N_{\alpha}$. Let us write
$$(U - \alpha) F_{\alpha} = G_{\alpha} \quad \Longrightarrow \quad U F_{\alpha} = \alpha F_{\alpha} + G_{\alpha} \quad \Longrightarrow \quad \alpha U^{-1} F_{\alpha} = F_{\alpha} - U^{-1} G_{\alpha}.$$
Note that~$U$ is invertible on~$N_{\alpha}$.
Since~$T F$ also lies in~$N_{\alpha} \oplus N_{\beta}$, we deduce that~$VF$ lies in~$N_{\alpha} \oplus N_{\beta}$. Yet~$UVF  = F \in N_{\beta}$, and so~$VF \in N_{\alpha}$,
and moreover
$\langle p \rangle VF = \alpha \beta U^{-1} F_{\alpha}$.
It follows that
$$\begin{aligned}
(T - \alpha - \beta) F = & \ (U - \alpha) F_{\alpha} + (\langle p \rangle V - \beta) F_{\alpha} = (U - \alpha) F_{\alpha} + 
\alpha \beta U^{-1} F_{\alpha} - \beta F_{\alpha} \\
= & \ G_{\alpha} + \beta F_{\alpha} - \beta U^{-1} G_{\alpha} - \beta F_{\alpha} = G_{\alpha} - \beta U^{-1} G_{\alpha}.
 \end{aligned}$$
 If~$G_{\alpha} \ne 0$, then the latter expression is non-zero, since applying~$U$ gives~$U G_{\alpha} - \beta G_{\alpha}$ and~$\beta \ne \alpha$.
On the other hand,~$G_{\alpha}$ is deeper in the filtration of~$N_{\alpha}$ given by
$$N_{\alpha} \supset (U - \alpha) N_{\alpha} \supset (U - \alpha)^2 N_{\alpha} \ldots $$
and hence,
 replacing~$F$ by~$(T - \alpha - \beta)F$ sufficiently many times, we may assume that~$G_{\alpha} = 0$, that~$U F_{\alpha} = \alpha F_{\alpha}$, and that
 $(T - \alpha - \beta) F_{\alpha} = 0$. We are thus left with a form~$F$ such that:
 $$T F = (\alpha + \beta) F, \qquad U F = \alpha F, \qquad V F = \beta F.$$
We may now achieve a contradiction based purely on a computation with formal~$q$-expansions. For example, the identity~$V F = \beta F$ is impossible as soon
as either~$\beta \ne 1$ or~$F$ is a cusp form, simply by considering the exponent of the smallest coefficient. Alternatively, a non-formal argument using properties of modular forms would be to
 note that~$\theta V F = 0$, and then
use the fact that~$\theta$ has no kernel in low weight (by~\cite{Katz}).
 \end{proof}

A different proof of this theorem is given in~\cite{CG}; the point is that the proof given here avoids
any geometry.
 The proof below is somewhat in this spirit --- using some elementary reductions, we arrive, given an element of~$\ker(\pi_{\beta})$,
and a form~$F$ which is simultaneously acted upon by a collection of formal operators in a 
very constrained way. The identities we get are not quite enough to deduce that~$F = 0$
as formal~$q$-expansions, however, they are enough to produce forms of low weight
inside the kernel of various theta operators, which will be enough to
produce a contraction by Theorems~\ref{theorem:theta} and~\ref{theorem:boxer}.
 No doubt (see~\S\ref{january}) there will be better geometric replacements for this argument, so we apologize in advance for the somewhat messy approach that we
present here.

\medskip

As in the proof above, let use write:
$$M = H^0(X_1(Q),\omega(j,2)_k)_{\m}, \qquad N = H^0(X_1(Q),\omega(j+k,2+k)_{k})_{\m}.$$
The map~$M \rightarrow N$ is certainly injective, as can be seen by the~$q$-expansion principle (the map is the identity on~$q$-expansions).
By abuse of notation, we view~$M \subset N$ under this map. Since~$\alpha \beta \ne 0$, the operator~$Q_2$ acts invertibly on~$M$.
Depending on the weight~$\sigma$, the operator~$Q_2$ acts on~$M$ either as~$Z_2$ or as~$Z_2 + X_2$.

\begin{lemma} \label{lemma:injective} Assume that~$\alpha$ and~$\beta$ are as in Theorem~\ref{theorem:qexp}. Suppose that~$\sigma = (j,2)$ with~$j > 2$. Then~$M = Q_2 M = Z_2 M$, and~$M$ is a subspace
of the submodule of~$N$ on which~$U$ is invertible. If~$\sigma = (2,2)$, then~$Z_2$ acts on~$N$,  the map~$M \rightarrow Z_2 M$ is injective, and~$Z_2 M \subset N$ is a subspace
of the submodule of~$N$ on which~$U$ is invertible.
\end{lemma}

\begin{proof} In the first case, by assumption we know that~$Q_2 - \alpha \beta$ is nilpotent, and so~$Q_2$ induces an isomorphism of~$M$. On the other hand, the operator~$Q_2$ acts
via the formal operator~$Z_2$. In weight~$\tau = (j+k,2+k)$, the corresponding operator~$Q_2$ also acts via~$Z_2$, and so we deduce that~$Q_2 - \alpha \beta$ acts on~$M \subset N$
and acts nilpotently. Yet~$Q_2$ only acts invertibly on the ordinary part of~$N$, as can be seen by lifting to characteristic zero.
Now let us consider the case of weight~$\sigma = (2,2)$. We have
$$M = Q_2 M = (Z_2 + X_2)M.$$
Now~$Q_2$ acts in weight~$N$ by~$Z_2$, so certainly~$Z_2 M \subset N$. 
Since~$Q_2$ acts by~$Z_2 + X_2$ on $M$, 
there is a commutative diagram as follows:
$$\begin{diagram}
M & \rTo & Z_2 M \\
\dTo & & \dTo \\
Q_2 M & \rTo & Z_2 (Z_2 + X_2)  M = Z^2_2 M \\
\end{diagram}
$$
where (by Lemma~\ref{lemma:iszero}) we use the fact that~$Z_2 X_2 = 0$.
Since the left hand side is an isomorphism, it follows that~$Z^2_2 M = Z_2 M$,
and hence that~$Z_2$ acts invertibly on~$Z_2 M$, and as in the previous argument
it follows that~$Z_2$ and hence~$U$ is invertible on this space.

 Hence it suffices to show that~$Z_2 F \ne 0$ for any~$F \in M$.
Suppose that~$Z_2 F = 0$. Then~$Q_2 F = Z_2 F + X_2 F = X_2 F$. Since~$Q_2 F \in M$, we have~$X_2 F \in M$. Yet then (again by Lemma~\ref{lemma:iszero}) we have
$Q^2_2 F = (Z_2 + X_2) X_2 F = X^2_2 F$, and then~$Q^3_2 F = X^3_2 F = X_2 F$,
and so~$Q_2 F = X_2 F \ne 0$ is an eigenvector of~$Q_2$ with eigenvalue~$\lambda$ satisfying $\lambda^2 = 1$. Yet the only generalized eigenvalue of~$Q_2$ is~$\alpha \beta$, and by assumption~$(\alpha \beta)^2 \ne 1$.
\end{proof}
 
(Note that this is the point in this paper which uses the assumption~$(\alpha \beta)^2 \ne 1$
 rather than the weaker claim~$\alpha \beta \ne 1$ which is sufficient for arguments
 on the Galois side.)
 
\begin{lemma} The operator~$U(U - \alpha)(U - \beta)$ acts nilpotently on~$N$.
\end{lemma}

\begin{proof} This follows by lifting to characteristic zero and  noting that the only possible unit crystalline eigenvalues of Frobenius of a lift of~$\rbar$ are~$\alpha$ or~$\beta$ modulo~$\m$.
\end{proof}

\begin{lemma} \label{lemma:explicitstuff} Suppose that the composite~$\pi_{\beta} : Z_2 M \rightarrow N_{\beta}$ is not injective.
\begin{enumerate}
\item If~$(\sigma) = (j,2)$ with~$j > 2$, there exists a nonzero form~$F  = F_{\alpha} \in M \cap N_{\alpha}$ such that
$$U F= \alpha F, \qquad T F = (\alpha + \beta) F, \qquad Z F = \beta F.$$
\item If~$(\sigma) = (2,2)$, there exists a nonzero form~$F = F_{\alpha} + F_0$ with~$F_{\alpha} \in N_{\alpha}$ and~$F_0 \in N_{0}$ such that:
$$U F_{\alpha} = \alpha F_{\alpha}, \qquad TF = (\alpha + \beta) F, \qquad X_2 F = \alpha \beta F_0.$$
\end{enumerate}
\end{lemma}

\begin{proof} First note that~$T F = (U + Z) F \in M$, and that~$U F \in N$, so~$Z F \in N$. 
Assume that~$\sigma = (j,2)$ with~$j > 2$. Note that~$Z_2$ commutes with~$U$. Hence, 
after replacing~$F \in \ker(\pi_{\beta})$ by~$(Z_2 - \alpha \beta)^m F = (Q_2 - \alpha \beta)^m F$
for sufficiently large~$m$, we may assume that~$Z_2 F = \alpha \beta F$. The assumption~$\pi_{\beta}(F) = 0$ implies that~$F = F_{\alpha} \in N_{\alpha}$.
Clearly~$U F \in N_{\alpha}$ also, and so~$ZF = TF - UF \in N_{\alpha} \oplus N_{\beta}$.
Yet~$Z_2 = U Z$, so we have
$$U Z F = \alpha \beta  F_{\alpha} \Rightarrow Z F = \alpha \beta U^{-1} F_{\alpha} \in N_{\alpha}$$
(There can be no component in~$N_{\beta}$ because~$U$ is invertible on that space.)
Write~$(U - \alpha) F_{\alpha} = G_{\alpha}$, so~$U F_{\alpha} - G_{\alpha} = \alpha F_{\alpha}$, or
$$\alpha U^{-1} F_{\alpha} = F_{\alpha} - U^{-1} G_{\alpha}.$$
We infer that
$$\begin{aligned}
(T - \alpha - \beta) F = & \  (U + Z - \alpha - \beta) F_{\alpha} = (U - \alpha) F_{\alpha} + (Z - \beta) F_{\alpha} \\
= & \  G_{\alpha} + \beta F_{\alpha} - \beta U^{-1} G_{\alpha} - \beta F_{\alpha} \\
= & \  G_{\alpha} - \beta U^{-1} G_{\alpha}. \end{aligned}$$
We claim that if~$G_{\alpha} \ne 0$, then the last expression is non-zero. This is because~$U$ acts invertibly on~$N_{\alpha}$, and applying~$U$ we get
$$U (G_{\alpha} - \beta U^{-1} G_{\alpha}) = (U - \alpha) G_{\alpha} + (\alpha - \beta) G_{\alpha},$$
and~$(U - \alpha) G_{\alpha}$ has a smaller nilpotence level than~$G_{\alpha}$, and~$(\alpha - \beta) \ne 0$. 
In particular, replacing~$F$ by~$(T - \alpha - \beta) F$, we may find more elements in~$M$ which also lie in the kernel of~$\pi_{\beta}$, and reduce to the case
where~$U F_{\alpha} = \alpha F_{\alpha}$ and~$Z_2 F_{\alpha} = UZ F_{\alpha} = \alpha \beta F_{\alpha}$. However, in this case, we also see that~$Z F_{\alpha} = \beta F_{\alpha}$, and the required equalities follow.

\medskip

Now suppose that~$\sigma = (2,2)$. 
Let us write~$\pi_{\beta}: Z_2 M \subset N_{\alpha} \oplus N_{\beta} \rightarrow N_{\beta}$ as
$(U - \alpha)^m$, and so~$(U - \alpha)^m Z_2 F = 0$ for some~$F \ne 0$.
Since~$Z_2$ formally commutes with~$U$, we also get
$$(U - \alpha)^m (Z^2_2) F = Z_2 (U - \alpha)^m Z_2 F = 0.$$
so~$Z_2$ preserves the property of~$Z_2 F$ lying in the kernel of~$\pi_{\beta}$.
But
$$Z_2 (Z_2 + X_2) F = Z^2_2 F,$$
because~$Z_2 X_2 = 0$. Hence, if~$Z_2 F$ lies in the kernel of~$\pi_{\beta}$, then so does
$$Z_2 Q_2 F = Z_2 (Z_2 + X_2)F.$$
Hence we may repeatedly replace~$F$ by~$(Q_2 - \alpha \beta)F = (Z_2 + X_2 - \alpha \beta)F$, and thus replace~$F$ by a form such 
that~$Q_2 F = \alpha \beta F$ and~$Z_2 F \in N_{\alpha}$.
Now, as above, we may write
$$F = F_{\alpha} + F_0  = (F_{\alpha},0,F_0) \in N_{\alpha} \oplus  N_{\beta} \oplus N_0.$$
We are assuming that~$Q_2 F = \alpha \beta F$, and so
$$Q_2 F = (\alpha \beta F_{\alpha}, 0,\alpha \beta F_{0}).$$
Thus we deduce that
$X_2 F = (0,0,\alpha \beta F_{0})$ and
$Z_2 F =  (\alpha \beta F_{\alpha}, 0, 0)$.
We once more would like to use that~$T = U + Z$ implies that~$Z F \in N$.
However, we no longer know (or expect) that~$ZF$ it is ordinary.
However, since~$Z_2 = U Z$ and~$ZF \in N$, we certainly deduce that 
$$Z F = (U^{-1} \alpha \beta F_{\alpha}, 0, G_{0}),$$
for some~$G_0$ in the kernel of~$U$.
Are arguments are similar to those used above.
We write~$(U - \alpha) F_{\alpha} = G_{\alpha}$, so
$U F_{\alpha} - G_{\alpha} = \alpha F_{\alpha}$, or
$$\alpha U^{-1} F_{\alpha} = F_{\alpha} - U^{-1} G_{\alpha}.$$
This implies that
$$\begin{aligned}
 G:=(T - \alpha - \beta) F = & \ (U + Z - \alpha - \beta) (F_{\alpha},0,F_{0}) \\
 = & \ (\alpha F_{\alpha} + G_{\alpha},0,0) + (\beta F_{\alpha} -  \beta U^{-1} G_{\alpha},0,G_0)
 - (\alpha+ \beta) F \\
 = & \ (G_{\alpha} -   \beta U^{-1} G_{\alpha},0,G_0 - (\alpha + \beta) F_{0})
\end{aligned} 
$$
The first term lies in a space where~$(U - \alpha)$ is nilpotent, but it has a smaller nilpotence level than~$F_{\alpha}$
by construction. Moreover, if it is equal to zero, then
$$0 = U (G_{\alpha} - \beta U^{-1} G_{\alpha}) = \alpha G_{\alpha} + H_{\alpha} - \beta G_{\alpha},$$
where~$(U - \alpha) G_{\alpha} = H_{\alpha}$ has yet a  higher level of nilpotence. In particular,
this can equal zero only if either~$\alpha = \beta$ or~$G_{\alpha} = 0$. Since we are explicitly forbidding the former,
we may assume, by induction, that~$F_{\alpha} \ne 0$ is a~$U$-eigenvector, and so
$$(T - \alpha  - \beta) F = (0,0,G_0  - (\alpha + \beta) F_{0}).$$
This implies that~$Z_2 (T - \alpha - \beta) F = 0$, and thus (from the injectivity of~$Z_2$ in Lemma~\ref{lemma:injective}) 
that~$(T - \alpha - \beta) F = 0$, or that~$F$ is a~$T$-eigenform.
The required identities follow immediately upon writing~$F = F_{\alpha} + F_0$ where~$F$ is a~$T$-eigenform, $U F_{\alpha} = \alpha F_{\alpha}$,
and~$X_2 F = \alpha \beta F_0$.
\end{proof}

At this point, to prove Theorem~\ref{theorem:qexp}, it suffices to show that there are no Siegel modular forms which satisfy the above identities. For example,
in weights~$\sigma = (j,2)$ with~$j > 2$, we would like to show that there is no form~$F$ which is an eigenform for both~$T$ and~$U$. 
We now examine what constraints these identities place on the Fourier coefficients of~$F$.

\begin{remark}[Tripling] \label{trip} \emph{A theme of~\cite{CG}, following previous work of Wiese~\cite{Wiese}, was to prove
that certain Galois representations were ordinary in two different ways by {\bf doubling\rm}, that is, mapping the form
of low weight to forms of heigh weight in two different ways. This is also our argument in weights~$(j,2)$ for~$j \ge 4$.
However, in weight~$(2,2)$, we see some new phenomena. When we pass to weight~$(p+1,p+1)$, we see not only
the the space of low weight forms has been doubled, but rather tripled, with the image generating (under the map~$X_2$)
is mapped to  the kernel of~$Z_2$. What this must mean is that, in weight~$(p+1,p+1)$, any ordinary Galois
representation coming from weight~$(2,2)$ should have a non-ordinary lift in weight~$(p+1,p+1)$. This phenomena doesn't
happen for~$\GL(2)$, since forms of weight~$p$ which are ordinary modulo~$p$ are ordinary in characteristic zero by
(boundary cases of) Fontaine--Laffaille theory. For~$\GSp(4)$, however, the Hodge--Tate weights in weight~$(p+1,p+1)$
are~$[0,p-1,p,2p-1]$, which are well beyond the Fontaine--Laffaille range.
One can also ask  what is the exact relationship between tripling argument here in weight~$(2,2)$
and the doubling  version of~\cite{BCGP} at Klingen level. For our purposes,
this
 would require proving that there exists a (Hecke equivariant away from~$p$)
injection from
from our space of forms~$M$ at spherical  level to a space of ordinary forms
(with respect to the operator denoted~$U_{\mathrm{Kli},2}$ in~\cite{BCGP}) at Klingen level
also in weight~$(2,2)$.
While this should certainly be true, we have not attempted to prove it.
}
\end{remark}

\subsection{Binary quadratic forms}

\begin{df}  \emph{We define a set with multiplicities $\Forms(Q)$ of equivalence classes of $p$-integral
binary quadratic
forms as follows.
For each $M \in \Se$ (with~$\Se$ as defined in Lemma~\ref{eightthree}), we add $[P]$ to  $\Forms(Q)$ if and only if
there exists a $P \in [P]$ such that $Q = M.P$. In particular, $M$ contributes
a class $[P]$ if and only if $[M^{-1}.Q]$ is $p$-integral.}
\end{df}

An easy lemma shows that $\Forms(Q)$ only depends on $[Q]$.
A binary quadratic form defines a section of $\OL(2)$ on $\PP^1(\F_p)$, the latter of which
is in natural bijection to $\Se$ (recall that~$\Se$ is the coset space of~$\diag(1,p)$ in~$\Gammabar \subset \SL_2(\Z)$).  We
see that $M^{-1}.Q$ is $p$-integral if any only if the corresponding quadratic form has a zero at the
corresponding point in $\PP^1(\F_p)$. In particular, $\Forms(Q)$ is empty if $Q$ does
not represent zero.
Moreover, the cardinality of $\Forms(Q)$ is given by the number of zeros of $Q$, and
is thus equal to $0$, $1$, or $2$ if $Q$ is $p$-primitive. (If $Q$ is not $p$-primitive,
then $Q \equiv 0 \mod p$ and $\Forms(Q)$ has cardinality $p+1$).

The definition of $\Forms(Q)$ is motivated by the following observation: There is an identity
$$a(ZF,Q) = \sum_{[P] \in \Forms(Q)} \rho(M_P) a(F,P),$$
where $P \in [P]$ is some (any) element in $[P]$ such that $M_P.P = Q$ for $M_P \in \Se$.

\begin{lemma} \label{lemma:symmetric} If $[P] \in \Forms([Q])$, 
then $[Q]  \in \Forms([P])$.
\end{lemma}

\begin{proof}  Replacing $Q$ by $g.Q$ for some $g \in \Gammabar \subset \SL_2(\Z)$, we may assume
that $Q = M.P$  where 
$$M = \left( \begin{matrix} 1 & 0 \\ 0 & p \end{matrix} \right).$$
Yet then $p M^{-1}.Q = M^{-1}.Q = P$, and $p M^{-1} \in \Se$.
\end{proof}

Let $d(Q)$ denote the discriminant of $Q$.

\begin{lemma} Suppose that $Q$ is $p$-primitive. Let $D = d(Q)$.
Then either:
\begin{enumerate}
\item $(D/p) = -1$, and  $\Forms([Q])$ is empty.
\item $(D/p) = 0$, and $\Forms([Q])$ has exactly one element.
\item $(D/p) = +1$, and $\Forms([Q])$ has exactly two elements.
\end{enumerate} \label{lemma:zeroonetwo}
\end{lemma}

\begin{proof} This follows from the fact that a $p$-primitive form $Q$ has exactly $0$,  $1$,
or $2$ solutions in $\PP(\F_p)$, depending on whether $(D/p)$ is $-1$, $0$, or $1$
respectively. Note that (in the final case) $\Forms([Q])$ may consist of the same class with
multiplicity two. This happens, for example, if $(D/p) = 1$ and the class number of $D$ is one.
\end{proof}

In light of Lemma~\ref{lemma:explicitstuff}, to prove Theorem~\ref{theorem:qexp}, it suffices to prove the following. 

\begin{theorem}  \label{theorem:nodice}
Suppose that $F = \sum a(F,Q) q^Q$ is a Siegel modular $q$-expansion of weight~$\sigma = (j,2)$ in characteristic~$p$,
where~$p-1 > j$.
\begin{enumerate}
\item  Let~$\sigma = (j,2)$ with~$j \ge 4$, and suppose that~$UF = \alpha F$
and~$ZF = \beta F$ for some~$\alpha, \beta$ with~$\alpha \beta (\beta^2 - 1) \ne 0$, then~$F = 0$.
\item
Let~$\sigma = (2,2)$, and suppose that~$F = F_{\alpha} + F_0$, where
$U F_{\alpha} = \alpha F_{\alpha}$, $X_2 F = \alpha \beta F_0$, and $Z F = \beta F  + \alpha F_0$
for some~$\alpha, \beta$ with~$\alpha \beta (\beta^2 - 1)(\alpha^2 \beta^2 - 1) \ne 0$. 
Then~$F = 0$.
\end{enumerate}
\label{theorem:Zbegone}
\end{theorem}

\begin{proof}
We first prove that that there exists a~$Q$ with~$\det(Q) \not\equiv 0 \mod p$.
In particular, in weight~$(2,2)$, 
 we may also assume that~$F_0 = (\alpha \beta)^{-1} X_2 F = 0$,
 and thus have the
 equalities:
$$U F = \alpha F,  \quad ZF = \beta F.$$
In fact, we may assume these equalities hold in both cases, since we are assuming such an equality
holds in the case of non-parallel weight.
If $a(F,pP) \ne 0$, then, since $a(F,pP) = a(UF,P) = \alpha \cdot a(F,P)$, we have
$a(F,P) \ne 0$. Hence, if $F \ne 0$, 
 there exists a $p$-primitive form $Q$ with $a(F,Q) \ne 0$.
 Without loss of generality, assume that $Q$ is a
  $p$-primitive form of minimal discriminant with $a(F,Q) \ne 0$.
  By Lemma~\ref{lemma:zeroonetwo}, $\Forms(Q)$ consists of a single class $[P]$.
  It follows that 
  $$a(F,P) =  \rho(M_Q) a(ZF,Q) = \beta \cdot \rho(M_Q) a(F,Q).$$
  If $P$ is \emph{not} $p$-primitive, then $P = pR$ for some $R$, and then
  $a(F,R) \ne 0$, contradicting the minimality of
  $Q$ (note that $P$ and $Q$ have the same discriminant). Hence $P$ is also $p$-primitive.
  Yet then $\Forms(P)$ consists of a single element, which must be $[Q]$ by
  Lemma~\ref{lemma:symmetric}.
  Yet then it follows that
 $$\beta^2 a(F,Q) = a(Z^2 F,Q) = \rho(M_P) a(ZF,P) = \rho(M_Q) \rho(M_P) a(F,Q) = 
 \begin{cases} 0, & j > 2 \\ a(F,Q), & j = 2 \end{cases}$$
Here we use that $P = M_Q.Q = M_Q.M_P.P$, and thus
 $\rho(M_Q.M_P) = \rho(p \cdot I)$ is the identity in weight~$(2,2)$ and zero in higher weight.
  If~$j > 2$ we are done, and if~$\sigma = (2,2)$, we are done since~$\beta^2 -1 \ne 0$.

\begin{remark}
\emph{
As an alternative to this argument, one could use an analogue of Theorem~\ref{theorem:boxer}
to show that the kernel of~$\Theta$ is trivial in low weight (but this would require formulating and then
proving such a theorem for non-parallel weight).
}
\end{remark}

We may therefore assume that~$a(F,Q) \ne 0$ for some~$Q$ of
discriminant~$D$ prime to~$p$.

\subsection{The case~\texorpdfstring{$\sigma = (2,2)$}{sigma=(2,2)}.}
Let us now assume that~$\sigma = (2,2)$.
The coefficient~$a(X_2 F,Q)$ is equal to~$(D/p) a(F,R)$, where~$D = D_Q$ is the
discriminant of~$Q$. Hence, since~$ZF =  \beta F + \beta^{-1} X_2 F$,
we deduce that, if~$(D/p) = -1$, that
$$0 = a(ZF,Q) = \beta a(F,Q) - \beta^{-1} a(F,Q) = (\beta - \beta^{-1}) a(F,Q).$$
Assuming that~$\beta^2 \ne 1$, we deduce that~$a(F,Q) = 0$.
It follows that the only~$Q$ with~$a(F,Q) \ne 0$ have~$D = \det(Q)$ satisfying~$(D/p) = 0,1$.
In particular, the form
$$F - X_2 F \in M$$
lies in the kernel of~$\Theta$. Yet this implies that~$F - X_2 F$ trivial by Theorem~\ref{theorem:boxer}. But this implies that~$Z_2 F = Z_2 X_2 F = 0$,
and this contradicts the injectivity of~$Z_2: M \rightarrow N$ in Lemma~\ref{lemma:injective}.

 \subsection{The case~\texorpdfstring{$\sigma = (j,2)$}{sigma=(j,2)} with~\texorpdfstring{$j \ge 4$}{j ge 4}}
 We may assume that~$a(F,Q) \ne 0$, where~$Q$ is~$p$-primitive
 and~$D = d(Q)$ is non-zero. If~$(D/p) = -1$, then~$a(ZF,Q) = 0$,
 contradicting the non-vanishing of~$a(F,Q)$ and the identity~$ZF = \beta Z$.
 Hence we may assume that~$(D/p)  = 1$. The action of~$\Gammabar \subset \SL_2(\Z)$ on binary quadratic forms of discriminant~$D$ has a finite orbit which may be identified
with a ray class group. The assumption
on~$D$ implies that~$Q$ has exactly two zeros in~$\PP^1(\F_p)$. 
For either of the zeros (say~$\xi$), we may consider the corresponding quadratic form
$$P = M.Q : = M Q M^{T} \cdot \det(M)^{-1},$$
where~$M$ is a representative of an element in~$\Se$ corresponding to~$\xi$. The class of~$P$ in the class group does not depend on the choice of representative of~$M$.
The quadratic form~$P$ also has two roots. We claim that, for one of those roots, there is a choice of representative~$N$ for the element in~$\Se$ such that
$$N.P:= N P N^{T} \cdot \det(N)^{-1} = Q, \qquad MN = \left( \begin{matrix} p & 0 \\ 0 & p \end{matrix} \right).$$
Indeed, if~$N = p M^{-1}$, then the corresponding identity is trivially satisfied.
We may view the process of applying~$Z$ dynamically as follows:
The coefficient corresponding to 
a quadratic form~$Q$ of discriminant~$D$ with~$(D/p) = +1$ of~$ZF$ is given by a sum~$\rho(M_P) a(F,P) + \rho(M_R) a(F,R)$ for a pair of quadratic
forms~$P$ and~$R$ also of the same discriminant. The ray class group corresponding to~$Q$ is partitioned by this process~$Q \rightarrow \{P,R\}$  into a finite number of cyclic orbits, on which this operation takes a binary quadratic form to its two nearest neighbours (if the orbit has fewer than two elements, this pair of neighbours may have multiplicity). 
Let us now consider the coefficient~$a(Z^2 F,Q)$. This consists of two pairs of two terms coming from the neighboring quadratic forms~$P$ and~$R$ respectively.
From the above, for each neighbour~$P$, there will be a term of the form
$$\rho(M) \rho(N) a(F,Q) = \rho(MN) a(F,Q) = 0,$$
where the identity~$\rho(MN) = 0$ requires the assumption that~$j > k$.
Hence~$a(Z^2 F,Q)$ will also be a sum of two terms coming from the quadratic forms of distance~$2$ away from~$Q$ inside its cyclic orbit.
Let us consider one orbit of size~$s$. Then, we also see, modifying~$M_s$ by an element of~$\Gammabar$ if necessary, that
$$ Q = M_s M_{s-1} \ldots M_1.Q = A Q A^{t} \cdot \det(A)^{-1},$$
where~$A = M_{s} M_{s-1} \ldots M_1 \in M_2(\Z)$ has~$\det(A) = p^s$.
Cycling the other way, we deduce the following:

\begin{lemma} \label{eighttwentythree} Suppose that~$F$ is a formal Siegel modular
form of weight~$(j,2)$
which is an eigenform of~$Z$ with eigenvalue~$\beta$.
Suppose that~$Q$ has discriminant~$D$~with~$(D/p) = 1$.
Then there exists an integer~$s > 0$ such that
$$\begin{aligned}
\beta^s a(F,Q) =  & \ \rho(A) a(F,Q) + \rho(B) a(F,Q),\\
=  & \ \Sym^{j-2}[A] a(F,Q) + \Sym^{j-2}[B] a(F,Q), \end{aligned}
$$
where
$$A Q A^{t} = p^s Q, \qquad B Q B^{t} = p^s Q.$$
\end{lemma}

We now make a small recap:
At the beginning of the 
of the proof of Theorem~\ref{theorem:nodice}, we
proved that we could assume that~$F$
had a non-zero coefficient~$a(F,Q)$ where~$Q$
has non-zero discriminant modulo~$p$.
If~$(D/p) = -1$, then~$a(ZF,Q) = 0$, which (with~$ZF = \beta F$)
would imply that~$F = 0$.  Hence we may assume
there is a non-zero  coefficient with~$(D/p) = +1$ (which we exploit below)
and use the following proposition to reach the final contradiction.

\begin{prop} \label{prop:justnow}
Suppose that~$F$ is a formal Siegel modular
form of weight~$(j,2)$ modulo~$p$ 
which is an eigenform of~$Z$ with eigenvalue~$\beta$ such that
$\beta \ne 0$, and suppose that~$p > j - 2$.
Suppose that~$F$ has a non-zero coefficient~$a(F,Q)$
where~$(D/p) = 1$.
Then~$\theta_1 F = 0$.
\end{prop}

\begin{proof}
The map~$\theta_1$ is induced from the contraction map
$$\con: \Sym^{j-2} \otimes \Sym^2 \rightarrow \Sym^{j-4} \otimes \det$$
(this is well defined integrally as long as~$p > j - 2$).
In particular,  we have the identity
$$a(\theta_1 Z^s F,Q) = 
\con(\Sym^{j-2}[A] a(F,Q) \otimes Q^{\vee}) +
\con(\Sym^{j-2}[B] a(F,Q) \otimes Q^{\vee}),$$
where~$\con$ denotes the contraction map. We claim that~$\con(\Sym^{j-2}[A] x \otimes Q^{\vee}) = 0$
for any~$x \in \Sym^{j-2} V$, where~$V = k^2$. Once we have this, 
we deduce that~$\beta^s a(\theta_1 F,Q) = a(\theta_1 Z^s F, Q) = 0$, and since~$\beta \ne 0$,
we have~$a(\theta_1 F,Q) = 0$ and~$\theta_1 F = 0$.

While there is probably an easy coordinate free way to prove the required claim, it is also
simple enough to do the computation explicitly by writing everything out in terms of bases.
Let us write down a standard basis~$\{f_1,f_2\}$ for~$V$ and a standard basis~$\{e_1,e_2\}$ for~$V^{\vee}$.
To be explicit, we choose bases such that
a form
$$Q =  \left( \begin{matrix} m & \frac{1}{2} r \\ \frac{1}{2}r & n \end{matrix} \right)$$
gives rise to the element~$m f^2_1 + r f_1 f_2 + n f^2_2$,  and~$Q^{\vee}$ gives
rise to~$m e^2_1  + r e_1 e_2 + n e^2_2$.
With respect to this choice, the contraction map on~$\Sym^2 \otimes \Sym^2$ (up to scalar) corresponds to
 sending~$e^2_1 f^2_2$ and~$e^2_2 f^2_1$ to~$-2$
and~$e_1 f_1 e_2 f_2$ to~$1$,  and sending all other monomials to zero.
As a consistency check, note that
$$\con(Q \otimes Q^{\vee}) = r^2 - 4 m n = -4 \det(Q).$$
Similarly, the contraction mapping on~$\Sym^{j-2} \otimes \Sym^2$ for~$p > j-2$ satisfies
$$\con(f^{i}_1 f^{j-2-i}_2 e^k_1 e^{2-k}_2) =  0 \ \text{unless} \
2 \le i + k \le j - 4.$$

The formula~$Q =  A Q A^{t} \det(A)^{-1}$ continues to hold if
we replace~$Q$ by~$M.Q = MQM^{t}$ and~$A$ by~$MAM^{-1}$ some invertible~$M$. 
In particular, we may replace~$A$ by any integral conjugate. We
consider two cases.
\begin{enumerate}
\item $A$ has a non-zero eigenvalue mod~$p$. 
In this case (by Hensel's Lemma), the matrix~$A$ has an eigenvalue over~$\Z_p$, and a second eigenvalue which has valuation~$s$.
In particular,  after a change of basis, we may write
 $$A = \left(\begin{matrix} u & 0 \\ 0 & 0 \end{matrix} \right) \mod p^s,
 \quad Q =  \left( \begin{matrix} m & \frac{1}{2} r \\ \frac{1}{2}r & n \end{matrix} \right).$$
 The conditions~$AQA^{t} = \det(A) Q $ and~$\det(A) = p^s$  imply that~$n \equiv 0 \mod p^s$
 (multiply out and consider the bottom right entry), 
 and thus that~$Q^{\vee} = m e^2_1 + r e_1 e_2 \mod p$. But now the
image of~$A$ on~$k$ is generated by~$f_1$, and so
 the image of~$\Sym^{j-2}[A] x$ is given by~$f^{j-2}_1$.
 But this forces the contraction after tensoring with~$Q^{\vee}$ to be zero over~$k$,
 because the only monomial which~$f^{j-2}_1$ contracts with non-trivially with is~$e^2_2$.
 \item $A$ is nilpotent modulo~$p$. If~$A$ is trivial modulo~$p$ there is nothing to prove. On the other hand,
 if
 $$A \equiv  \left(\begin{matrix} 0 & 1 \\ 0 & 0 \end{matrix} \right) \mod p,$$
 then once again the image of~$A$ is generated by~$f$, and the conditions~$A Q A^{t} = \det(A) Q$ and~$\det(A) = p^s$
 imply once more that~$n \equiv 0 \mod p$ (multiply out as above but now consider the top left entry), and the proof  proceeds as in
 the previous case.
 \end{enumerate}
This completes the proof of the proposition.
\end{proof}

Combining Prop.~\ref{prop:justnow} with Lemma~\ref{lemma:explicitstuff}
and Theorem~\ref{theorem:theta}, we obtain a contradiction, and this completes
the proof of Theorem~\ref{theorem:qexp}.

\end{proof}

\section{Modularity Lifting} 

The following theorem is the main result of this paper.

\begin{theorem}
  \label{thm:main-thm}
Let $\rbar : G_{\Q} \to \GSp_4(k)$ be a continuous, odd, absolutely
irreducible Galois representation. 
Suppose that $\nu(\rbar) =
\epsilon^{-(a-1)}$ where $p-1 > a \geq 2$.
 Suppose that the following hold:
\begin{enumerate}
\item There exist units $\alpha$ and $\beta$ in $k$ such that
$$\rbar | G_{p} \sim
\left( \begin{matrix} 
\lambda(\alpha)  & 0 & * & * \\ 
0 &  \lambda(\beta) & * & * \\
0 & 0 & \nu(\rbar) \cdot \lambda( \beta^{-1}) & 0 \\
0 & 0 & 0 &  \nu(\rbar)  \cdot \lambda(\alpha^{-1})
\end{matrix} \  \ \right),$$
and moreover $(\alpha^2 - 1)(\beta^2 - 1)(\alpha^2 \beta^2 - 1)(\alpha -
\beta) \ne 0$.
\item  Let $S(\rbar)$ denote the set of primes of $\Q$ away from $p$ at which $\rbar$ is
ramified. Then for each $x\in S(\rbar)$, the restriction $\rbar|G_{x}$
falls into one of the cases of Assumption~\ref{assumption:ramification}.
\item $($Big Image$)$ The restriction $\rbar|G_{\Q(\zeta_p)}$ has big image in the
  sense of Assumption~\ref{assumption:bigimage}.
\item The representation $\rbar$ is Katz modular of weight $\sigma :=
  (a,2) \in X^*(T)^+_M$ in the sense of
  Definition~\ref{df:katz-modular}. 
    \item  $($Neatness$)$ $\rbar$ satisfies Assumption~\ref{assumption:neatness}.
\end{enumerate}
We now introduce some notation: let $K \subset \GSp_4(\A^{\infty})$ be the compact open subgroup
  defined as in the beginning of Section~\ref{sec:gal-rep-cohom}. Let
  $X = X_K$, and for any set of primes $Q$ disjoint from $S(\rbar)\cup
  \{p\}$, let $X_i(Q) = X_{K_i(Q)}$. Let the Hecke algebras
  $\T_{\sigma}$ and $\T^{\an}_{\sigma}(Q)$ be as in
  Definition~\ref{defn:hecke-alg}. The assumption that $\rbar$ is Katz
  modular implies that there is a maximal ideal $\m_{\emptyset}$ of
  $\T_{\sigma}$ associated to $\rbar$. The pullback of
  $\m_{\emptyset}$ to $\T^{\an}_{\sigma}(Q)$ is also denoted
  $\m_{\emptyset}$. We further assume:
  \begin{enumerate}[resume]
  \item \label{requiredvanishing} If $Q$ satisfies
    Assumption~\ref{assumption:hecke-galois-rep-regular}~\eqref{ass:TW-at-Q},
    then 
\[ H^2(X_i(Q),\omega(a,2)(-\infty)_k)_{\m_{\emptyset}} = \{ 0 \}.\]
  \end{enumerate}

Let $R^{\min}$ be the universal deformation ring classifying minimal
deformations of $\rbar$ in the sense of Definition~\ref{defn:minimal}
(with $Q$ taken to be empty). Then the map
\[ R^{\min} \to \T_{\sigma,\m_{\emptyset}}^{\alpha,\beta}, \]
which classifies the minimal deformation of
Theorem~\ref{theorem:localglobal} (with $Q$ taken to be empty), is an
isomorphism. Furthermore, the space
\[ H^0(X, \omega(a,2)(-\infty)_{K/\CO})^{\alpha,\beta,\vee}_{\m_{\emptyset}} \]
is a free $\T_{\sigma,\m_{\emptyset}}^{\alpha,\beta}$ module.
\end{theorem}

Note that, for~$p \ge a \ge 4$, the hypothesis~\ref{requiredvanishing}   holds by Theorem~\ref{thm:lan-suh}.

\begin{proof}
  To prove the theorem, we apply Proposition~\ref{prop:patching}, as
  follows:
  \begin{enumerate}
  \item Take $R = R^{\min}$ and $H = H^0(X, \omega(a,2)(-\infty)_{K/\CO})^{\alpha,\beta,\vee}_{\m_{\emptyset}}$.
  \item Let $q$ and the sets $Q_N$ be as in
    Proposition~\ref{prop:tw-primes-w1}.
  \item The ring $R_{\infty}$ is the power series ring
    $\CO[[x_1,\dots,x_{q-1}]]$.
  \item For each $N \geq 1$, we define a surjection $R_{\infty} \onto
    R$ as follows: Let $R_{Q_N}$ denote the universal deformation ring
    classifying deformations of $\rbar$ which are minimal outside $Q$, in the sense
    of Definition~\ref{defn:minimal}. Choose any surjection
    $R_{\infty} \onto R_{Q_N}$ (possible by
    Proposition~\ref{prop:tw-primes-w1}) and let $R_{\infty}\onto R$
    be the composite of this surjection with the natural map $R_{Q_N}
    \onto R^{\min}$. 

    We define the module $H_N$ as follows: let $\Delta$ be the unique
    quotient of $\Delta_{Q_N} = \prod_{x\in Q_N}(\Z/x)^{\times}$ which is
    isomorphic to $(\Z/p^N\Z)^{q}$, and let $X_{\Delta}(Q_N) \to
    X_0(Q_N)$ be as in Section~\ref{sec:balanced-property}.
 Let $\m_N$
    be the ideal $\m \subset \T_{\sigma}(Q_N)$ of Theorem~\ref{thm:no-newforms} when $Q$ is
    taken to be $Q_N$. We then take
    \[ H_N := H^0(X_{\Delta}(Q_N), \omega(a,2)(-\infty)_{K/\CO})^{\alpha,\beta,\vee}_{\m_N} \]
    and we regard it as an $R_{\infty}$-module via the
    surjection $R_{\infty} \onto R_{Q_N}$ chosen above, and the
    classifying map $R_{Q_N} \onto \T_{\sigma}(Q)_{\m_N}^{\alpha,\beta}$ associated to
    the deformation $r_{Q_N}$ of
    Theorem~\ref{theorem:localglobal}. The $S_N$-module structure on
    $H_N$ is given by choosing an identification $\Delta \cong
    (\Z/p^N\Z)^{q}$.
  \end{enumerate}
We need to check that, given these definitions, the conditions of
Proposition~\ref{prop:patching} hold.
\begin{enumerate}[label=(\alph*)]
\item The image of $S_N$ in $\End_{\CO}(H_N)$ is contained in the
  image of $R_{\infty}$ because under the Galois representation
  $r_{Q_N}$ of Theorem~\ref{theorem:localglobal}, the image of an
  element $\sigma \in I_x$, for $x$ a prime in $Q_N$, is conjugate to a
  matrix of the form $\diag(1,1,\langle u \rangle, \langle u \rangle)$
  where $\Art_x(u) = \sigma$. This follows from~\cite[Corollary 3]{Sor}.
\item We have
  \begin{eqnarray*}
    (H_N)_{\Delta_N} &=& \left( \left(H^0(X_{\Delta}(Q_N), \omega(a,2)(-\infty)_{K/\CO})^{\alpha,\beta}_{\m_N}\right)^{\Delta_N}\right)^{\vee} \\
& = & H^0(X_{0}(Q_N), \omega(a,2)(-\infty)_{K/\CO})^{\alpha,\beta,\vee}_{\m_N}.
  \end{eqnarray*}
Combining this with the isomorphism of Theorem~\ref{thm:no-newforms},
we obtain an isomorphism:
\[ \psi_N : (H_N)_{\Delta_N} \liso H. \]
\item Finally, $H_N$ is finite and balanced over $S_N$ by
  Theorem~\ref{thm:balanced}. 
  \end{enumerate}
We can thus apply Proposition~\ref{prop:patching}, and we deduce that
$H$ is a finite free $R$-module. Since the action of $R$ on $H$
factors through $\T_{\sigma,\m_{\emptyset}}^{\alpha,\beta}$, the
conclusions of Theorem~\ref{thm:main-thm} follow immediately.
\end{proof}

%
%
%
%
%

\bibliographystyle{amsalpha}
\bibliography{CG}
 
 \end{document}